\documentclass{amsproc}
\usepackage[arrow, curve, frame, matrix, graph, tips]{xy}
\usepackage{amssymb}
\usepackage{amsthm}
\usepackage{amsmath}
\usepackage{enumerate}
\usepackage[colorlinks=true,urlcolor=blue,linkcolor=blue]{hyperref}

\newtheorem{theorem}{Theorem}[section]
\newtheorem{lemma}[theorem]{Lemma}
\newtheorem{proposition}[theorem]{Proposition}
\newtheorem{corollary}[theorem]{Corollary}

\theoremstyle{definition}
\newtheorem{definition}[theorem]{Definition}
\newtheorem{example}[theorem]{Example}

\newtheorem{non-example}[theorem]{Non-Example}

\newtheorem{question}[theorem]{Question}
\newtheorem{terminology}[theorem]{Terminology}

\theoremstyle{remark}
\newtheorem{remark}[theorem]{Remark}

\newcommand{\tn}[1]{\textnormal{#1}}
\newcommand{\tnb}[1]{\textnormal{\bf #1}}
\newcommand{\tensor}{\otimes}
\newcommand{\PSh}[1]{\widehat{#1}}
\newcommand{\C}{\mathbb{C}}
\newcommand{\D}{\mathbb{D}}

\newcommand{\N}{\mathbb{N}}
\newcommand{\comp}{\circ}
\newcommand{\id}{\tn{id}}
\newcommand{\ca}[1]{\mathcal{#1}}
\newcommand{\ladj}{\dashv}
\newcommand{\iso}{\cong}
\newcommand{\catequiv}{\simeq}
\newcommand{\Set}{\tnb{Set}}

\newcommand{\Cat}{\tnb{Cat}}
\newcommand{\CAT}{\tnb{CAT}}
\newcommand{\op}{\tn{op}}

\renewcommand{\implies}{\Rightarrow}

\DeclareMathOperator*{\Tbar}{\overline{T}}
\DeclareMathOperator*{\Sbar}{\overline{S}}

\DeclareMathOperator*{\opE}{E}
\DeclareMathOperator*{\opEpr}{E^{\prime}}
\DeclareMathOperator*{\opF}{F}

\DeclareMathOperator*{\opTx}{T^{\times}}

\newcommand{\Enrich}[1]{#1{\textnormal{-Cat}}}

\DeclareMathOperator*{\colim}{\textnormal{colim}}

\newcommand{\Coll}[1]{#1{\textnormal{-Coll}}}
\newcommand{\NColl}[1]{#1{\textnormal{-Coll}_0}}

\newcommand{\RdColl}[1]{{\tn{Rd-}}#1{\textnormal{-Coll}}}
\newcommand{\PtRdColl}[1]{{\tn{PtRd-}}#1{\textnormal{-Coll}}}
\newcommand{\RdOp}[1]{{\tn{Rd-}}#1{\textnormal{-Op}}}
\newcommand{\CtRdOp}[1]{{\tn{CtRd-}}#1{\textnormal{-Op}}}
\newcommand{\Mult}[1]{#1{\textnormal{-Mult}}}
\newcommand{\Op}[1]{#1{\textnormal{-Op}}}
\newcommand{\NOp}[1]{#1{\textnormal{-Op}_0}}

\newcommand{\MND}{\textnormal{MND}}

\newcommand{\DISTMULT}{\textnormal{DISTMULT}}
\newcommand{\ADM}{\textnormal{Acc-DISTMULT}}
\newcommand{\AGMnd}{\textnormal{Acc-}{\ca G}{\tn{-MND}}}
\newcommand{\AFO}{\textnormal{Acc-FOp}}

\begin{document}

\title{Multitensor lifting and strictly unital higher category theory}

\author{Michael Batanin}
\address{Department of Mathematics, Macquarie University}
\email{mbatanin@ics.mq.edu.au}
\thanks{}
\author{Denis-Charles Cisinski}
\address{Universit\'{e} Paul Sabatier, Institut de Math\'{e}matiques de Toulouse}
\email{denis-charles.cisinski@math.univ-toulouse.fr}
\thanks{}
\author{Mark Weber}
\address{Department of Mathematics, Macquarie University}
\email{mark.weber.math@gmail.com}
\thanks{}

\begin{abstract}
In this article we extend the theory of lax monoidal structures, also known as multitensors, and the monads on categories of enriched graphs that they give rise to. Our first principal result -- the lifting theorem for multitensors -- enables us to see the Gray tensor product of 2-categories and the Crans tensor product of Gray categories as part of this framework. We define weak $n$-categories with strict units by means of a notion of reduced higher operad, using the theory of algebraic weak factorisation systems. Our second principal result is to establish a lax tensor product on the category of weak $n$-categories with strict units, so that enriched categories with respect to this tensor product are exactly weak $(n{+}1)$-categories with strict units.
\end{abstract}

\maketitle
\tableofcontents

\section{Introduction}

This paper continues the developments of \cite{BataninWeber-EnHop} and \cite{Weber-MultitensorsMonadsEnrGraphs} on the interplay between monads and multitensors in the globular approach to higher category theory, and expands considerably on an earlier preprint \cite{BatCisWeb-EnHopII}. To take an important example, according to \cite{Weber-MultitensorsMonadsEnrGraphs} there are two related combinatorial objects which can be used to describe the notion of Gray category. One has the monad $A$ on the category $\ca G^3(\Set)$ of 3-globular sets whose algebras are Gray categories, which was first described in \cite{Batanin-MonGlobCats}. On the other hand there is a multitensor (ie a lax monoidal structure) on the category $\ca G^2(\Set)$ of $2$-globular sets, such that categories enriched in $E$ are exactly Gray categories. The theory described in \cite{Weber-MultitensorsMonadsEnrGraphs} explains how $A$ and $E$ are related as part of a general theory which applies to all operads of the sort defined 
originally in \cite{Batanin-MonGlobCats}.

However there is a third object which is missing from this picture, namely, the Gray tensor product of 2-categories. It is a simpler object than $A$ and $E$, and categories enriched in $\Enrich 2$ for the Gray tensor product are exactly Gray categories. The purpose of this paper is to exhibit the Gray tensor product as part of our emerging framework. This is done by means of the lifting theorem for multitensors -- theorem(\ref{thm:lift-mult}) of this article.

Strict $n$-categories can be defined by iterated enrichment, and the lifting theorem leads one to hope that such an inductive definition can be found for a wider class of higher categorical structures. The second main result of this article -- theorem(\ref{thm:iterative-defn}) -- says that one has a similar process of iterated enrichment to capture weak $n$-categories with strict units. In this result the appropriate (lax) tensor product $\ca L_{\leq n}$ of weak $n$-categories with strict units is identified, so that categories enriched in weak $n$-categories with strict units using this tensor product are exactly weak $(n{+}1)$-categories with strict units.

In order to formulate this second result, we identify a new class of higher operad, these being the reduced $\ca T_{\leq n}$-operads of section(\ref{ssec:reduced-operads}). Moreover we describe the notion of contractibility for such operads which enables one to formalise the idea that a given higher categorical structure should have strict units. Then weak $n$-categories with strict units are defined as algebras of the universal contractible reduced $\ca T_{\leq n}$-operad, analogous to how weak $n$-categories were defined operadically in \cite{Batanin-MonGlobCats}. It turns out that the lifting theorem is particularly compatible with these new notions, and it is this fact which is chiefly responsible for theorem(\ref{thm:iterative-defn}).

Let us turn now to a more detailed introduction to this article. Recall \cite{BataninWeber-EnHop,Weber-MultitensorsMonadsEnrGraphs} that a \emph{multitensor} $(E,u,\sigma)$ on a category $V$ consists of n-ary tensor product functors $E_n:V^n \to V$, whose values on objects are denoted in any of the following ways
\[ \begin{array}{ccccccc} {E(X_1,...,X_n)} && {E_n(X_1,...,X_n)} && {\opE\limits_{1{\leq}i{\leq}n} X_i} && {\opE\limits_i X_i} \end{array} \]
depending on what is most convenient, together with unit and substitution maps
\[ \begin{array}{lcr} {u_X:Z \rightarrow E_1X} && {\sigma_{X_{ij}}:\opE\limits_i\opE\limits_j X_{ij} \rightarrow \opE\limits_{ij} X_{ij}} \end{array} \]
for all $X$, $X_{ij}$ from $V$ which are natural in their arguments and satisfy the obvious unit and associativity axioms. It is also useful to think of $(E,u,\sigma)$ more abstractly as a lax algebra structure on $V$ for the monoid monad{\footnotemark{\footnotetext{Recall that for a category $V$, an object of $MV$ is a finite sequence $(X_1,...,X_n)$ of objects of $V$, that one only has morphisms in $MV$ between sequences of the same length, and that such a morphism $(X_1,...,X_n) \to (Y_1,...,Y_n)$ consists of morphisms $f_i:X_i \to Y_i$ for $1 \leq i \leq n$.}}} $M$ on $\CAT$, and so to denote $E$ as a functor $E:MV \to V$. The basic example to keep in mind is that of a monoidal structure on $V$, for in this case $E$ is given by the $n$-ary tensor products, $u$ is the identity and the components of $\sigma$ are given by coherence isomorphisms for the monoidal structure.

A \emph{category enriched in $E$} consists of a $V$-enriched graph $X$ together with composition maps
\[ \kappa_{x_i} : \opE\limits_i X(x_{i-1},x_i) \rightarrow X(x_0,x_n) \]
for all $n \in \N$ and sequences $(x_0,...,x_n)$ of objects of $X$, satisfying the evident unit and associativity axioms. With the evident notion of $E$-functor (see \cite{BataninWeber-EnHop}), one has a category $\Enrich E$ of $E$-categories and $E$-functors together with a forgetful functor
\[ U^E : \Enrich E \rightarrow \ca GV. \]
When $E$ is a distributive multitensor, that is when $E_n$ commutes with coproducts in each variable, one can construct a monad $\Gamma E$ on $\ca GV$ over $\Set$. The object map of the underlying endofunctor is given by the formula
\[ \Gamma EX(a,b) = \coprod\limits_{a=x_0,...,x_n=b} \opE\limits_iX(x_{i-1},x_i), \]
the unit $u$ is used to provide the unit of the monad and $\sigma$ is used to provide the multiplication. The identification of the algebras of $\Gamma E$ and categories enriched in $E$ is witnessed by a canonical isomorphism $\Enrich E \iso \ca G(V)^{\Gamma E}$ over $\ca GV$. In section(\ref{sec:MultMndReview}) we recall the relevant aspects of the theory of multitensors and monads from \cite{Weber-MultitensorsMonadsEnrGraphs}.

If one restricts attention to unary operations, then $E_1$, $u$ and the components $\sigma_X:E_1^2X \to E_1X$ provide the underlying endofunctor, unit, and multiplication for a monad on $V$. This monad is called the \emph{unary part} of $E$. When the unary part of $E$ is the identity monad, the multitensor is a \emph{functor operad}. This coincides with existing terminology, see \cite{McClureSmith-CosimplicialAndLittleCubes} for instance, except that we don't in this paper consider any symmetric group actions. Since units for functor operads are identities, we denote any such as a pair $(E,\sigma)$, where as for general multitensors $E$ denotes the functor part and $\sigma$ the substitution.

By definition then, a functor operad is a multitensor. On the other hand, as observed in \cite{BataninWeber-EnHop} lemma(2.7), the unary part of a multitensor $E$ acts on $E$, in the sense that as a functor $E$ factors as
\[ \xygraph{{MV}="l" [r] {V^{E_1}}="m" [r] {V}="r" "l":"m"^-{}:"r"^-{U^{E_1}}} \]
and in addition, the substitution maps are morphisms of $E_1$-algebras. Moreover an $E$-category structure on a $V$-enriched graph $X$ includes in particular an $E_1$-algebra structure on each hom $X(a,b)$ of $X$ with respect to which the composition maps are morphisms of $E_1$-algebras. These observations lead to
\begin{question}\label{q:lift}
Given a multitensor $(E,u,\sigma)$ on a category $V$ can one find a functor operad $(E',\sigma')$ on $V^{E_1}$ such that $E'$-categories are exactly $E$-categories?
\end{question}
The first main result of this paper, theorem(\ref{thm:lift-mult}), says that question(\ref{q:lift}) has a nice answer: when $E$ is distributive and accessible and $V$ is cocomplete, one can indeed find a \emph{unique} distributive accessible such $E'$.

One case of our lifting theorem in the literature is in the work of Ginzburg and Kapranov on Koszul duality \cite{GinzburgKapranov-KoszulDualityOperads}. Formula (1.2.13) of that paper, in the case of a $K$-collection $E$ coming from an operad, implicitly involves the lifting of the multitensor corresponding (as in \cite{BataninWeber-EnHop} example(2.6)) to the given operad. For instance, our lifting theorem gives a general explanation for why one must tensor over $K$ in that formula. As we shall see in section(\ref{ssec:convolution-via-lifting}), another existing source of examples comes from Day convolution \cite{Day-Convolution}.

Taking $E$ to be the multitensor on $\ca G^n(\Set)$ such that $\Gamma E$ is the monad whose algebras are weak $(n{+}1)$-categories, one might hope that the statement that ``$E'$-categories are exactly $E$-categories'' in this case expresses a sense in which weak $(n{+}1)$-categories are categories enriched in weak $n$-categories for an appropriate tensor product $E'$. However as we explain in section(\ref{ssec:weak units}), the presence of \emph{weak} identity arrows as part of the structure of weak $(n{+}1)$-category prevents such a pleasant interpretation directly, because composition with weak identity arrows gives the homs of an $E$-category extra structure, as can already be seen for the case $n=1$ of bicategories.

Having identified this issue, it is natural to ask
\begin{question}\label{q:iterative}
Does the lifting theorem produce the tensor products enabling weak $n$-categories with strict units to be obtained by iterated enrichment?
\end{question}
\noindent For this question to be well-posed it is necessary to define weak $n$-categories with strict units.

To do this it is worth first meditating a little on the operadic definition of weak $n$-category. Indeed since such a definition first appeared in \cite{Batanin-MonGlobCats}, as the algebras of a weakly initial higher operad of a certain type, there have been a number of works that have refined our understanding. In \cite{Leinster-HDA-book} this topic was given a lovely exposition, the focus was shifted to considering operads that were strictly initial with the appropriate properties rather than weakly so, and an alternative notion of contractibility was given which, for expository purposes, is a little simpler. Thus in this article we use Leinster contractibility throughout, however as we point out in remark(\ref{rem:Bat-defn}), one can give a completely analogous development using the original notions.

On a parallel track, Grandis and Tholen initiated an algebraic study of the weak factorisation systems arising in abstract homotopy theory in \cite{GrandisTholen-NatWFS}, this was refined by Garner in \cite{Garner-SmallObjectArgument} and then applied to higher category theory in \cite{Garner-LeinsterDefinition} and \cite{Garner-HomHigherCats}. We devote a substantial part of section(\ref{sec:weak-n-cat-review}) to a discussion of these developments, both by way of an updated review of the definition of weak $n$-category, and because we make essential use of Garner's theory of weak factorisation systems later in the article.

Assimilating these developments one can present an operadic definition of weak $n$-categories in the following way. One has a presheaf category $\NColl{\ca T_{\leq n}}$ of $\ca T_{\leq n}$-collections over $\Set$ (called ``normalised collections'' in \cite{Batanin-MonGlobCats}) and there are two finitary monads on $\NColl{\ca T_{\leq n}}$ the first of whose algebras are $\ca T_{\leq n}$-operads over $\Set$. To give an algebra structure for the second of these monads is to exhibit a given $\ca T_{\leq n}$-collection over $\Set$ as contractible, and one obtains this monad from an algebraic weak factorisation system, in the sense of \cite{Garner-SmallObjectArgument}, on $\NColl{\ca T_{\leq n}}$. In general the monad coproduct of a pair of finitary monads on a locally finitely presentable category is itself finitary, and so its category of algebras will also be locally finitely presentable. Applied to the two monads in question, the algebras of their monad coproduct are contractible $\ca T_{\leq n}$-operads over $\Set$, the 
initial such being the operad for weak $n$-categories.

We now adapt this to give an operadic definition weak $n$-category with strict units in the following way. First, one considers only reduced $\ca T_{\leq n}$-operads, which are those that contain a unique unit operation of each type. Second, one strengthens the notion of contractibility, by giving an algebraic weak factorisation system on the category $\PtRdColl{\ca T_{\leq n}}$ of pointed reduced $\ca T_{\leq n}$-collections. This notion of underlying collection includes the substitution of unique unit operations as part of the structure. The purpose of this refined contractibility notion is to encode the idea that the unique unit operations of the operads under consideration really do behave as strict identities. So as before one has a pair of finitary monads, but this time on the category $\PtRdColl{\ca T_{\leq n}}$, one whose algebras are reduced operads, and the other whose algebras are pointed reduced collections that are exhibited as contractible in this stronger sense. One then imitates the scheme laid out in the previous 
paragraph, defining the operad for weak $n$-categories with strict units as the initial object of the category of algebras of the coproduct of these two monads on $\PtRdColl{\ca T_{\leq n}}$.

A novelty of this definition is that it uses the extra generality afforded by the theory of algebraic weak factorisation system's of \cite{Garner-SmallObjectArgument}. Namely, while in the standard theory of weak factorisation systems one considers those that are generated, via the small object argument of Quillen, by a set of generating cofibrations. In Garner's algebraic version of the small object argument, one can consider a category of generating cofibrations. The algebraic weak factorisation system we consider on $\PtRdColl{\ca T_{\leq n}}$ has a category of generating cofibrations, and it is the morphisms of this category by means of which the strictness of units is expressed.

Having obtained a reasonable definition of weak $n$-category with strict units, we then proceed to answer question(\ref{q:iterative}) in the affirmative, this being our second main result theorem(\ref{thm:iterative-defn}). For any $\ca T_{\leq n{+}1}$-operad $B$ over $\Set$, there is a $\ca T_{\leq n}$-operad $h(B)$, whose algebras are by definition the structure possessed by the homs of a $B$-algebra. This is the object map of a functor $h$, and if one restricts attention to reduced $B$, and then contractible reduced $B$, $h$ in either of these contexts becomes a left adjoint, which is the formal fact that enables theorem(\ref{thm:iterative-defn}) to go through. 

This paper is organised in the following way. In section(\ref{sec:MultMndReview}) we review the theory of multitensors from \cite{Weber-MultitensorsMonadsEnrGraphs}. Then in section(\ref{sec:lifting-theorem}) we discuss the lifting theorem. The theorem itself is formulated and proved in section(\ref{ssec:lifting-theorem-proof}), applications are presented in sections(\ref{ssec:GrayCrans-key-examples}) and (\ref{ssec:convolution-via-lifting}), and the 2-functorial aspects of the lifting theorem are presented in section(\ref{ssec:functoriality-lifting}). We give a review of the definition of weak $n$-category and the theory of algebraic weak factorisation systems in section(\ref{sec:weak-n-cat-review}). In section(\ref{sec:strictly-unital-weak-n-cats}) we present our definition of weak $n$-category with strict units, and then the tensor products which exhibit this notion as arising by iterated enrichment are produced in section(\ref{sec:strictly-unital-weak-n-cats-via-iterated-enrichment}).

\emph{Notation and terminology}. Efforts have been made to keep the notation and terminology of this article consistent with \cite{Weber-MultitensorsMonadsEnrGraphs}. The various other standard conventions and abuses that we adopt include regarding the Yoneda embedding $\C \to \PSh{\C}$ as an inclusion, and so by the Yoneda lemma regarding $x \in XC$ for $X \in \PSh{\C}$ as an arrow $x:C \to X$ in $\PSh{\C}$. To any such $X$, and more generally any pseudo-functor $X:\C^{\op} \to \Cat$, the Grothendieck construction gives the associated fibration into $\C$, the domain of which we denote as $\tn{el}(X)$ and call the ``category of elements of $X$''. The objects of the topologists' simplicial category $\Delta$ are as usual written as ordinals $[n] = \{0<...<n\}$ for $n \in \N$, and regarded as a full subcategory of $\Cat$. Thus in particular $[1]$ is the category consisting of one non-identity arrow, and $\ca E^{[1]}$ denotes the arrow category of a category $\ca E$. We use a different notation for the 
algebraists' simplicial category $\Delta_+$, objects here being regarded as ordinals $n = \{1<...<n\}$ for $n \in \N$. In section(\ref{ssec:reduced-operads}), we adopt the abuse of identifying the left adjoint reflection of the inclusion of a full subcategory with the idempotent monad generated by the adjunction.

\section{Review of the theory of multitensors}
\label{sec:MultMndReview}

Given a category $V$, $\ca GV$ is the category of graphs enriched in $V$. Thus $\ca G1=\Set$ and $\ca G^n(\Set)$ is equivalent to the category of $n$-globular sets. As a category $\ca GV$ is at least as good as $V$. Two results from \cite{Weber-MultitensorsMonadsEnrGraphs} which express this formally, are proposition(2.2.3) which explains how colimits in $\ca GV$ are constructed from those in $V$, and theorem(2.2.7) from which it follows that if $V$ is locally presentable (resp. a Grothendieck topos, resp. a presheaf topos) then so is $\ca GV$.

Certain colimits in $\ca GV$ are very easy to understand directly, namely those connected colimits in which the arrows of the diagram are identities on objects. For in this case, as pointed out in remark(2.2.5) of \cite{Weber-MultitensorsMonadsEnrGraphs}, one can just take the object set of the colimit to be that of any of the $V$-graphs appearing in the diagram, and compute the colimit one hom at a time in the expected way. These colimits will play a basic role in the developments of section(\ref{ssec:lifting-theorem-proof}).

Formal properties  of the endofunctor $\ca G:\CAT \to \CAT$ were discussed in section(2.1) of \cite{Weber-MultitensorsMonadsEnrGraphs}. In particular $\ca G$ preserves Eilenberg Moore objects, so that given a monad $T$ on a category $V$, one has an isomorphism $\ca G(V)^{\ca G(T)} \iso \ca G(V^T)$ expressing that the algebras of the monad $\ca GT$ on $\ca GV$ are really just graphs enriched in the category $V^T$ of algebras of $T$.

As explained in the introduction to this article, given a distributive multitensor $E$ on a category $V$ with coproducts, the formula
\[ \Gamma EX(a,b) = \coprod\limits_{a=x_0,...,x_n=b} \opE\limits_iX(x_{i-1},x_i) \]
describes the object map of the underlying endofunctor of the monad $\Gamma E$ on the category $\ca GV$ of graphs enriched in $V$, whose algebras are $E$-categories. The assignment $(V,E) \mapsto (\ca GV,\Gamma E)$ is itself the object map of a locally fully faithful 2-functor
\[ \Gamma : \DISTMULT \longrightarrow \MND(\CAT/\Set) \]
from the full sub-2-category $\DISTMULT$ of the 2-category of lax monoidal categories, lax monoidal functors and monoidal natural transformations, consisting of those $(V,E)$ where $V$ has coproducts and $E$ is distributive, to the 2-category of monads, as defined in \cite{Street-FTM}, in the 2-category $\CAT/\Set$.

From this perspective then, the category $\ca GV$ is regarded as being ``over $\Set$'', that is to say, it is taken together with the functor $(-)_0 : \ca GV \to \Set$ which sends a $V$-graph to its set of objects. A monad in $\CAT/\Set$ is called a monad \emph{over} $\Set$, and the monad $\Gamma E$ is over $\Set$ because:
(1) for any $V$-graph $X$, the objects of $\Gamma E X$ are those of $X$;
(2) for any morphism of $V$-graphs $f:X \to Y$, the object map of $\Gamma E f$ is that of $f$; and (3) the components of the unit and multiplication of $\Gamma E$ are identity-on-objects morphisms of $V$-graphs. From the point of view of structure, $\Gamma E$ being over $\Set$ says that the structure of an $E$-category gives nothing at the level of objects.

Section(3), and especially theorem(3.3.1) of \cite{Weber-MultitensorsMonadsEnrGraphs}, contains a careful analysis of how properties of $E$ correspond to properties of the induced monad $\Gamma E$. The particular fact following from that discussion that we shall use below is recorded here in lemma(\ref{lem:EAcc->GammaEAcc}). Recall that a multitensor $E$ is $\lambda$-accessible for some regular cardinal $\lambda$, when the underlying functor $E:MV \to V$ preserves $\lambda$-filtered colimits, and that this condition is equivalent to each of the $E_n:V^n \to V$ preserving $\lambda$-filtered colimits in each variable (see \cite{Weber-MultitensorsMonadsEnrGraphs} section(3.3) for more discussion).
\begin{lemma}\label{lem:EAcc->GammaEAcc}
Suppose that $V$ is a cocomplete category, $\lambda$ is a regular cardinal and $E$ is distributive multitensor on $V$. Then $E$ is $\lambda$-accessible iff $\Gamma E$ is $\lambda$-accessible.
\end{lemma}
Monads $T$ on $\ca GV$ over $\Set$ of the form $\Gamma E$ for some distributive multitensor $E$ on $V$ a category with coproducts, are those that are path-like and distributive in the sense that we now recall. These properties concern only the functor part of the given monad, and so for the sake of the proof of theorem(\ref{thm:lift-mult}) below, we shall formulate these definitions more generally than in \cite{Weber-MultitensorsMonadsEnrGraphs}, for functors over $\Set$ between categories of enriched graphs.

When the category $V$ has an initial object $\emptyset$, sequences $(X_1,...,X_n)$ of objects of $V$ can be regarded as $V$-graphs. The object set for the $V$-graph $(X_1,...,X_n)$ is $\{0,...,n\}$, for $1 \leq i \leq n$ the hom from $(i{-}1)$ to $i$ is $X_i$, and the other homs are $\emptyset$. An evocative informal picture of this $V$-graph is
\[ \xygraph{{0}="p0" [r] {1}="p1" [r] {...}="p2" [r(1.3)] {n{-}1}="p3" [r(1.2)] {n}="p4" "p0":"p1"^-{X_1}:"p2"^-{X_2}:"p3"^-{X_{n{-}1}}:"p4"^-{X_n}} \]
because it encourages one to imagine $X_i$ as being the object of ways of moving from $(i{-}1)$ to $i$ in this simple $V$-graph.

Given a functor $T:\ca GV \to \ca GW$ over $\Set$, one can define a functor $\overline{T}:MV \to W$ whose object map is given by
\[ \begin{array}{rcl} {\Tbar\limits_{1{\leq}i{\leq}n} X_i} & {=} & {T(X_1,...,X_n)(0,n).} \end{array} \]
Note that since $T$ is over $\Set$, the $W$-graph $T(X_1,...,X_n)$ also has object set $\{0,...,n\}$, and so the expression on the right hand side of the above equation makes sense as a hom of this $W$-graph. By definition $\overline{T}$ amounts to functors $\overline{T}_n:V^n \to W$ for each $n \in \N$.
\begin{definition}\label{def:distributive}
Let $V$ and $W$ be categories with coproducts. A functor $T:\ca GV \to \ca GW$ over $\Set$ is \emph{distributive} when for each $n \in \N$, $\overline{T}_n$ preserves coproducts in each variable. A monad $T$ on $\ca GV$ over $\Set$ is \emph{distributive} when its underlying endofunctor is distributive.
\end{definition}
\noindent As in \cite{Weber-MultitensorsMonadsEnrGraphs} section(4.2), this definition can be reexpressed in more elementary terms without mention of $\overline{T}$.

Intuitively, path-likeness says that the homs of the $W$-graphs $TX$ are in some sense some kind of abstract path object. Formally for a functor $T:\ca GV \to \ca GW$ over $\Set$, given a $V$-graph $X$ and sequence $x=(x_0,...,x_n)$ of objects of $X$, one can define the morphism
\[ \overline{x} : (X(x_0,x_1),X(x_1,x_2),...,X(x_{n-1},x_n)) \to X \]
whose object map is $i \mapsto x_i$, and whose hom map between $(i-1)$ and $i$ is the identity. For all such sequences $x$ one has
\[ T(\overline{x})_{0,n} : \Tbar\limits_i X(x_{i-1},x_i) \to TX(x_0,x_n) \]
and so taking all sequences $x$ starting at $a$ and finishing at $b$ one induces the canonical map
\[ \pi_{T,X,a,b} : \coprod\limits_{a=x_0,...,x_n=b} \Tbar\limits_i X(x_{i-1},x_i) \rightarrow TX(a,b) \]
in $W$.
\begin{definition}\label{def:path-like}
Let $V$ and $W$ be categories with coproducts. A functor $T:\ca GV \to \ca GW$ over $\Set$ is \emph{path-like} when for all $X \in \ca GV$ and $a,b \in X_0$, the maps $\pi_{T,X,a,b}$ are isomorphisms. A monad $T$ on $\ca GV$ over $\Set$ is \emph{path-like} when its underlying endofunctor is path-like.
\end{definition}
A basic result concerning these notions that will be useful below is
\begin{lemma}\label{lem:dpl}
Let $V$, $W$ and $Y$ be categories with coproducts and $R:V \to W$, $T:\ca GV \to \ca GW$ and $S:\ca GW \to \ca GY$ be functors.
\begin{enumerate}
\item If $R$ preserves coproducts then $\ca GR$ is distributive and path-like.\label{lemcase:dpl-unary}
\item If $S$ and $T$ are distributive and path-like, then so is $ST$.\label{lemcase:dpl-composition}
\end{enumerate}
\end{lemma}
\begin{proof}
(\ref{lemcase:dpl-unary}): Since $R$ preserves the initial object one has $\ca GR(Z_1,...,Z_n) = (RZ_1,...,RZ_n)$ and so $\overline{\ca GR}:MV \to W$ sends sequences of length $n \neq 1$ to $\emptyset$, and its unary part is just $R$. Thus $\ca GR$ is distributive since $R$ preserves coproducts, and coproducts of copies of $\emptyset$ are initial. The summands of the domain of $\pi_{\ca GR,X,a,b}$ are initial unless $(x_0,...,x_n)$ is the sequence $(a,b)$, thus $\pi_{\ca GR,X,a,b}$ is clearly an isomorphism, and so $\ca GR$ is path-like.

(\ref{lemcase:dpl-composition}): Since $S$ and $T$ are path-like and distributive one has
\[ ST(Z_1,...,Z_n)(0,n) \iso \coprod\limits_{0=r_0{\leq}...{\leq}r_m=n} \Sbar\limits_{1{\leq}i{\leq}m} \,\,\, \Tbar\limits_{r_{i-1}{<}j{\leq}r_i} Z_j  \]
and so $ST$ is path-like and distributive since $S$ and $T$ are, and since a coproduct of coproducts is a coproduct. 
\end{proof}
The aforementioned characterisation of monads of the form $\Gamma E$ is then given by the following result.
\begin{theorem}\label{thm:GammaE-characterisation}
(\cite{Weber-MultitensorsMonadsEnrGraphs} theorem(4.2.4)).
For a category $V$ with coproducts, a monad $T$ on $\ca GV$ over $\Set$ is of the form $\Gamma E$ for a distributive multitensor $E$ iff $T$ is path-like and distributive, in which case one can take $E = \overline{T}$.
\end{theorem}
So far we have discussed two important constructions, $\Gamma$ which produces a monad from a multitensor, and $\overline{(-)}$ which produces a multitensor from a monad. We now recall a third construction $(-)^{\times}$ which also produces a multitensor from a monad. If $V$ has finite products and $T$ is any monad on $V$, then one can define a multitensor $T^{\times}$ on $V$ as follows
\[ \begin{array}{rcl} {\opTx\limits_{1{\leq}i{\leq}n}X_i} & {=} & {\prod\limits_{i=1}^n TX_i.} \end{array} \]
A category enriched in $T^{\times}$ is exactly a category enriched in $V^T$ using the cartesian product. As explained in section(5.2) of \cite{Weber-MultitensorsMonadsEnrGraphs}, this is an instance of a general phenomenon of a distributive law of a multitensor over a monad, in this case witnessed by the fact that any monad $T$ will be opmonoidal with respect to the cartesian product on $V$. If moreover $V$ has coproducts and $V$'s products distribute over them, so that cartesian product $\prod$ on $V$ is a distributive multitensor, and if $T$ preserves coproducts, then $T^{\times}$ is a distributive multitensor. The 2-functoriality of $\Gamma$ and the formal theory of monads \cite{Street-FTM} then ensures that at the level of monads one has a distributive law $\ca G(T)\Gamma(\prod) \to \Gamma(\prod)\ca G(T)$ between the monads $\Gamma(\prod)$ and $\ca G(T)$ on $\ca GV$.

Fundamental to the globular approach to higher category are the monads $\ca T_{\leq n}$, defined for $n \in \N$, on the category $\ca G^n(\Set)$ of $n$-globular sets, whose algebras are strict $n$-categories. Higher categorical structures in dimension $n$ in this approach are by definition algebras of $\ca T_{\leq n}$-operads, and a $\ca T_{\leq n}$-operad is by definition a monad $A$ on $\ca G^n(\Set)$ equipped with a cartesian monad morphism{\footnotemark{\footnotetext{That is, a natural transformation between underlying endofunctors, whose naturality squares are pullbacks, and which is compatible with the monad structures.}}} $\alpha:A \to \ca T_{\leq n}$. In terms of the theory described so far, the monads $\ca T_{\leq n}$ have a simple inductive description
\begin{itemize}
\item  $\ca T_{{\leq}0}$ is the identity monad on $\Set$.
\item  Given the monad $\ca T_{{\leq}n}$ on $\ca G^n\Set$, define the monad
$\ca T_{{\leq}n+1} = \Gamma \ca T^{\times}_{{\leq}n}$
on $\ca G^{n+1}\Set$.
\end{itemize}
The fact that $\ca T_{\leq n}$ algebras are strict $n$-categories is immediate from this definition, our understanding of what the constructions $\Gamma$ and $(-)^{\times}$ correspond to at the level of algebras and enriched categories, and the definition of strict $n$-category via iterated enrichment using cartesian product. Moreover, an inductive definition of the $\ca T_{\leq n}$ via an iterative system of distributive laws \cite{Cheng-IteratedDist}, is also an immediate consequence of this point of view.

That one has this good notion of $\ca T_{\leq n}$-operad expressable so generally in terms of cartesian monad morphisms, so that it fits within the framework of Burroni \cite{Burroni-TCats}, Hermida \cite{Hermida-RepresentableMulticategories} and Leinster \cite{Leinster-HDA-book}, is due to the fact that $\ca T_{\leq n}$ is a cartesian monad. That is, its functor part preserves pullbacks and the naturality squares for its unit and multiplication are pullback squares. The underlying endofunctor $\ca T_{\leq n}$ satisfies a stronger condition than pullback preservation, namely, it is a \emph{local right adjoint}, that is to say, for any $X \in \ca G^n(\Set)$, the functor
\[ (\ca T_{\leq n})_X : \ca G^n(\Set)/X \longrightarrow \ca G^n(\Set)/\ca T_{\leq n} X \]
obtained by applying $\ca T_{\leq n}$ to arrows into $X$, is a right adjoint. A cartesian monad whose functor part satisfies this stronger property is called a \emph{local right adjoint monad}{\footnotemark{\footnotetext{In \cite{BergMellWeber-MonadsArities} the terminology \emph{strongly cartesian monad} is used.}}} in this article, and this notion is important because for such monads one can define the nerve of an algebra in a useful way. See \cite{Berger-CellularNerve}, \cite{Weber-Fam2fun}, \cite{Mellies-Segal} and \cite{BergMellWeber-MonadsArities} for further discussion.

All these pleasant properties enjoyed by the monads $\ca T_{\leq n}$ can be understood in terms the compatibility of these properties with the constructions $\Gamma$ and $(-)^{\times}$. See \cite{Weber-MultitensorsMonadsEnrGraphs}, especially theorem(3.3.1) and example(5.2.4), for more details.

A $\ca T_{\leq n}$-operad over $\Set$ is a $\ca T_{\leq n}$-operad $\alpha:A \to \ca T_{\leq n}$ such that the monad $A$ is over $\Set$ and $\alpha$'s components are identities on objects. We denote by $\NOp {\ca T_{\leq n}}$ the category of $\ca T_{\leq n}$-operads over $\Set$. There is also a notion of $E$-multitensor for a given cartesian multitensor $E$ on a category $V$, which consists of another multitensor $F$ on $V$ equipped with a cartesian morphism of multitensors $\phi:F \to E$. We denote by $\Mult E$ the category of $E$-multitensors. Thanks to the functoriality of $\Gamma$ and its compatibility with the categorical properties participating in these definitions one has
\begin{theorem}\label{thm:Operad-Multitensor-Equiv}
The constructions $\Gamma$ and $\overline{(-)}$ give an equivalence of categories
\[ \Mult {\ca T_{\leq n}^{\times}} \catequiv \NOp {\ca T_{\leq n{+}1}}. \]
\end{theorem}
\noindent This result first appeared as corollary(7.10) of \cite{BataninWeber-EnHop}, and was subsequently generalised in \cite{Weber-MultitensorsMonadsEnrGraphs} theorem(6.1.1) with $\ca T_{\leq n}^{\times}$ replaced by a general cartesian multitensor $E$ on a lextensive category $V$.

\section{Lifting theorem}
\label{sec:lifting-theorem}

\subsection{Overview}
\label{ssec:lifting-theorem-discussion}
In \cite{Crans-ATensorForGrayCats} the so-called ``Crans tensor product'' of Gray categories was constructed by hand, and categories enriched in this tensor product were called ``4-teisi'' and were unpacked in \cite{Crans-BraidingsSyllepsesSymmetries}, these being Crans' candidate notion of semi-strict 4-category. From the perspective of these papers it is not at all clear that one can proceed in the reverse order, first constructing a $\ca T_{\leq 4}$-operad for 4-teisi, and then obtaining the tensor product from this. By ordering the discussion in this way one makes a direct formal connection between the combinatorial work of Crans and the theory of higher operads. Even given an intuition that one can begin with the operad and then obtain the tensor product, it is very far from clear from the way that things are described in \cite{Crans-ATensorForGrayCats} and \cite{Crans-BraidingsSyllepsesSymmetries}, how one would express such an intuition formally without drowning in the complexity of the combinatorial notions involved.

The lifting theorem applied to this example gives a formal expression of this intuition in a way independent of any of these combinatorial details. It does not however say anything about how the $\ca T_{\leq 4}$-operad for 4-teisi is constructed. Put simply, the $\ca T_{\leq 4}$-operad for 4-teisi is the input, and the Crans tensor product is the output for this application of the lifting theorem. 

This section is organised as follows. In section(\ref{ssec:lifting-theorem-proof}) we formulate and prove the lifting theorem as an answer to question(\ref{q:lift}). Then in section(\ref{ssec:GrayCrans-key-examples}) we discuss how this result brings the Gray and Crans tensor products, of 2-categories and Gray categories respectively, into our framework. Section(\ref{ssec:convolution-via-lifting}) exhibits Day convolution as arising via an application of the lifting theorem. Finally in section(\ref{ssec:functoriality-lifting}) we exhibit the process of lifting a multitensor as a right 2-adjoint to an appropriate inclusion of functor operads among multitensors.

\subsection{The theorem and its proof}
\label{ssec:lifting-theorem-proof}
The idea which enables us to answer question(\ref{q:lift}) is the following. Given a distributive multitensor $E$ on $V$ one can consider also the multitensor $\widetilde{E_1}$ whose unary part is also $E_1$, but whose non-unary parts are all constant at the initial object. This is clearly a sub-multitensor of $E$, also distributive, and moreover as we shall see one has $\Enrich{\widetilde{E_1}} \iso \ca G(V^{E_1})$ over $\ca GV$. Thus from the inclusion $\widetilde{E_1} \hookrightarrow E$ one induces the forgetful functor $U$ fitting in the commutative triangle
\[ \xygraph{{\ca G(V^{E_1})}="l" [r(2)] {\Enrich E}="r" [dl] {\ca GV}="b" "l":@{<-}"r"^-{U}:"b"^-{U^E}:@{<-}"l"^-{\ca G(U^{E_1})}} \]
For sufficiently nice $V$ and $E$ this forgetful functor has a left adjoint. The category of algebras of the induced monad $T$ will be $\Enrich E$ since $U$ is monadic. Thus problem is reduced to that of establishing that this monad $T$ arises from a multitensor on $V^{E_1}$. By theorem(\ref{thm:GammaE-characterisation}) this amounts to showing that $T$ is path-like and distributive.

In order to implement this strategy, we must understand something about the explicit description of the left adjoint $\ca G(V^{E_1}) \to \Enrich E$, and this understanding is a basic part of monad theory that we now briefly recall. Suppose that $(M,\eta^M,\mu^M)$ and $(S,\eta^S,\mu^S)$ are monads on a category $\ca E$, and $\phi:M{\rightarrow}S$ is a morphism of monads. Then $\phi$ induces the forgetful functor $\phi^*:{\ca E}^S \rightarrow {\ca E}^M$ and we are interested in computing the left adjoint $\phi_!$ to $\phi^*$. By the Dubuc adjoint triangle theorem \cite{Dubuc-KanExtensions}, one may compute the value of $\phi_!$ at an $M$-algebra $(X,x:MX{\rightarrow}X)$ as a reflexive coequaliser
\begin{equation}\label{eq:algebraic-left-adjoint}
\xygraph{!{0;(3,0):}
{(SMX,\mu^S_{MX})}="l" [r] {(SX,\mu^S_X)}="m" [r] {\phi_!(X,x)}="r"
"l":@<2ex>"m"^-{\mu^S_XS(\phi_X)}:"l"|-{S\eta^M_X}:@<-2ex>"m"_-{Sx}:"r"^-{q_{(X,x)}}}
\end{equation}
in ${\ca E}^S$, when this coequaliser exists. Writing $T$ for the monad induced by $\phi_! \ladj \phi^*$, the components of the unit $\eta^T$ of $T$ are given by
\[ X \xrightarrow{\eta^S_X} SX \xrightarrow{q_{(X,x)}} \phi_!(X,x). \]

The coequaliser (\ref{eq:algebraic-left-adjoint}) is taken in ${\ca E}^S$, and we will need to have some understanding of how this coequaliser is computed in terms of colimits in $\ca E$. So we suppose that $\ca E$ has filtered colimits and coequalisers, and that $S$ is $\lambda$-accessible for some regular cardinal $\lambda$. Then the underlying object $U^S\phi_!(X,x)$ in $\ca E$ of $\phi_!(X,x)$ is constructed in the following way. We shall construct morphisms
\[ \begin{array}{lccr} {v_{n,X,x} : SQ_n(X,x) \rightarrow Q_{n{+}1}(X,x)} &&& {q_{n,X,x}:Q_n(X,x) \to Q_{n{+}1}(X,x)} \end{array} \]
\[ \begin{array}{c} {q_{{<}n,X,x}:SX \to Q_n(X,x)}  \end{array} \]
starting with $Q_0(X,x)=SX$ by transfinite induction on $n$.
\\ \\
{\bf Initial step}. Define $q_{{<}0}$ to be the identity, $q_0$ to be the coequaliser of $\mu^S(S\phi)$ and $Sx$, $q_{<{1}}=q_0$ and $v_0=q_0b$. Note also that $q_0=v_0\eta^S$.
\\ \\
{\bf Inductive step}. Assuming that $v_n$, $q_n$ and $q_{<{n{+}1}}$ are given, we define $v_{n{+}1}$ to be the coequaliser of $S(q_n)(\mu^SQ_n)$ and $Sv_n$, $q_{n{+}1}=v_{n{+}1}(\eta^SQ_{n{+}1})$ and $q_{<{n{+}2}}=q_{n{+}1}q_{<{n{+}1}}$.
\\ \\
{\bf Limit step}. Define $Q_n(X,x)$ as the colimit of the sequence given by the objects $Q_m(X,x)$ and morphisms $q_m$ for $m < n$, and $q_{<{n}}$ for the component of the universal cocone at $m=0$.
\[ \xygraph{!{0;(3,0):(0,.333)::} {\colim_{m{<}n} S^2Q_m}="tl" [r] {\colim_{m{<}n} SQ_m}="tm" [r] {\colim_{m{<}n} Q_m}="tr" [d] {Q_n}="br" [l] {SQ_n}="bm" [l] {S^2Q_n}="bl" "tl":@<1ex>"tm"^-{\mu_{<{n}}}:@<1ex>"tr"^-{v_{<{n}}} "tl":@<-1ex>"tm"_-{(Sv)_{<{n}}}:@<-1ex>@{<-}"tr"_-{\eta_{<{n}}} "bl":"bm"_-{\mu}:@{<-}"br"_-{\eta} "tl":"bl"_{o_{n,2}} "tm":"bm"^{o_{n,1}} "tr":@{=}"br"} \]
We write $o_{n,1}$ and $o_{n,2}$ for the obstruction maps measuring the extent to which $S$ and $S^2$ preserve the colimit defining $Q_n(X,x)$. We write $\mu^S_{<{n}}$, $(Sv)_{<{n}}$, $v_{<{n}}$ and $\eta^S_{<{n}}$ for the maps induced by the $\mu^SQ_m$, $Sv_m$, $v_m$ and $\eta^SQ_m$ for $m < n$ respectively. Define $v_n$ as the coequaliser of $o_{n,1}\mu_{<{n}}$ and $o_{n,1}(Sv)_{<{n}}$, $q_n=v_n({\eta^S}Q_n)$ and $q_{<{n{+}1}}=q_nq_{<{n}}$.
\\ \\
Then since $S$ preserves $\lambda$-filtered colimits, this sequence stabilises in the sense that for any ordinal $n$ such that $|n| \geq \lambda$, $q_{n,X,x}$ is an isomorphism. Thus for any such $n$ one may take
\[ \begin{array}{lccr} {\phi_!(X,x)=(Q_n(X,x),q_n^{-1}v_n)} &&& {q_{<n} : (SX,\mu_X) \to (Q_n(X,x),q_n^{-1}v_n)} \end{array} \]
as an explicit definition of $\phi_!(X,x)$ and the associated coequalising map in $V^S$. The proof of this is essentially standard -- see theorem(3.9) of \cite{BarrWells-TTT} for example, and so we omit the details.

In fact all that we require of the above details is that the transfinite construction involves only \emph{connected} colimits. Moreover in the context that we shall soon consider, these will be connected colimits of diagrams of $V$-graphs which live wholly within a single fibre of $(-)_0:\ca GV \to \Set$. As was explained in section(\ref{sec:MultMndReview}) and in remark(2.2.5) of \cite{Weber-MultitensorsMonadsEnrGraphs}, such colimits in $\ca GV$ are straight forward.
\begin{lemma}\label{lem:concol-pathlike}
Let $V$ be a category with coproducts, $W$ be a cocomplete category, $J$ be a small connected category and
\[ F : J \rightarrow [\ca GV,\ca GW] \]
be a functor. Suppose that $F$ sends objects and arrows of $J$ to functors and natural transformations over $\Set$.
\begin{itemize}
\item[(1)]  Then the colimit $K:\ca GV{\rightarrow}\ca GW$ of $F$ may be chosen to be over $\Set$.\label{cpl1}
\end{itemize}
Given such a choice of $K$:
\begin{itemize}
\item[(2)]  If $Fj$ is path-like for all $j \in J$, then $K$ is also path-like.\label{cpl2}
\item[(3)]  If $Fj$ is distributive for all $j \in J$, then $K$ is also distributive.\label{clp3}
\end{itemize}
\end{lemma}
\begin{proof}
Colimits in $[\ca GV,\ca GW]$ are computed componentwise from colimits in $\ca GW$ and so for $X \in \ca GV$ we must describe a universal cocone with components
\[ \kappa_{X,j} : Fj(X) \rightarrow KX. \]
By remark(2.2.5) of \cite{Weber-MultitensorsMonadsEnrGraphs} we may demand that the $\kappa_{X,j}$ are identities on objects, and then compute the hom of the colimit between $a,b \in X_0$ by taking a colimit cocone
\[ \{\kappa_{X,j}\}_{a,b} : Fj(X)(a,b) \rightarrow KX(a,b) \]
in $W$. This establishes (1). Since the properties of path-likeness and distributivity involve only colimits at the level of the homs as does the construction of $K$ just given, (2) and (3) follow immediately since colimits commute with colimits in general.
\end{proof}
Recall the structure-semantics result of Lawvere, which says that for any category $\ca E$, the canonical functor
\[ \begin{array}{lccccc} {\tn{Mnd}(\ca E)^{\op} \to \CAT/\ca E} &&& {T} & {\mapsto} & {U:\ca E^T \to \ca E} \end{array} \]
with object map indicated is fully faithful (see \cite{Street-FTM} for a proof). An important consequence of this is that for monads $S$ and $T$ on $\ca E$, an isomorphism $\ca E^T \iso \ca E^S$ over $\ca E$ is induced by a unique isomorphism $S \iso T$ of monads. We now have all the pieces we need to implement our strategy. First, in the following lemma, we give the result we need to recognise the induced monad on $\ca G(V^{E_1})$ as arising from a multitensor.
\begin{lemma}\label{lem:mnd-lift-mult}
Let $\lambda$ be a regular cardinal. Suppose that $V$ is a cocomplete category, $R$ is a coproduct preserving monad on $V$, $S$ is a $\lambda$-accessible monad on $\ca GV$ over $\Set$, and $\phi:\ca GR{\rightarrow}S$ is a monad morphism over $\Set$. Denote by $T$ the monad on $\ca G(V^R)$ induced by $\phi_! \ladj \phi^*$.
\begin{itemize}
\item[(1)]  One may choose $\phi_!$ so that $T$ is over $\Set$.
\end{itemize}
Given such a choice of $\phi_!$:
\begin{itemize}
\item[(2)]  If $S$ is distributive and path-like then so is $T$.
\item[(3)]  If $R$ is $\lambda$-accessible then so is $T$.
\end{itemize}
\end{lemma}
\begin{proof}
Let us denote by $\rho:RU^R \to U^R$ the 2-cell datum of the Eilenberg-Moore object for $R$, and note that since $\ca G$ preserves Eilenberg Moore objects, one may identify $U^{\ca GR}=\ca G(U^R)$ and $\ca G\rho$ as the 2-cell datum for $\ca GR$'s Eilenberg-Moore object. Now $T$ is over $\Set$ iff $\ca G(U^R)T$ is. Moreover since $R$ preserves coproducts $U^R$ creates them, and so $T$ is path-like and distributive iff $\ca G(U^R)T$ is. Since $\ca G(U^R)T = U^S\phi_!$, it follows that $T$ is over $\Set$, path-like and distributive iff $U^S\phi_!$ is. Since the monads $S$ and $\ca GR$ are over $\Set$, as are $\rho$ and $\phi$, it follows by a transfinite induction using lemma(\ref{lem:concol-pathlike}) that all successive stages of this construction give functors and natural transformations over $\Set$, whence $U^S\phi_!$ is itself over $\Set$. Lemma(\ref{lem:dpl}) ensures that the functors $\ca GR$ and $\ca G(RU^R)$ are distributive and path-like, since $R$ preserves coproducts and $U^R$ creates them.
When $S$ is also distributive and path-like, then by the same sort of transfinite induction using lemmas(\ref{lem:dpl}) and (\ref{lem:concol-pathlike}), all successive stages of this construction give functors that are distributive and path-like, whence $U^S\phi_!$ is itself distributive and path-like.

Supposing $R$ to be $\lambda$-accessible, note that $\ca GR$ is also $\lambda$-accessible. One way to see this is to consider the distributive multitensor $\tilde{R}$ on $V$ whose unary part is $R$ and non-unary parts are constant at the initial object. Thus $\tilde{R}$ will be $\lambda$-accessible since $R$ is. To give an $\tilde{R}$-category structure on $X \in \ca GV$ amounts to giving $R$-algebra structures to the homs of $X$, and similarly on morphisms, whence one has a canonical isomorphism $\Enrich{\tilde{R}} \iso \ca G(V^R)$ over $\ca GV$, and thus by structure-semantics one obtains $\Gamma \tilde{R} \iso \ca GR$. Hence by lemma(\ref{lem:EAcc->GammaEAcc}), $\ca GR$ is indeed $\lambda$-accessible. But then it follows that $U^{\ca GR} = \ca G(U^R)$ creates $\lambda$-filtered colimits, and so $T$ is $\lambda$-accessible iff $\ca G(U^R)T = U^S\phi_!$ is. In the transfinite construction of $U^S\phi_!$, it is now clear that the functors involved at every stage are $\lambda$-accessible by yet another transfinite induction, and so $U^S\phi_!$ is $\lambda$-accessible as required.

To finish the proof we must check that $T$'s monad structure is over $\Set$. Since $\mu^T$ is a retraction of $\eta^TT$ it suffices to verify that $\eta^T$ is over $\Set$, which is equivalent to asking that the components of $\ca G(U^R)\eta^T$ are identities on objects. Returning to the more general setting of a morphism of monads $\phi:M \to S$ on $\ca E$ and $(X,x) \in \ca E^M$ discussed above, the outside of the diagram on the left
\[ \xygraph{{\xybox{\xygraph{!{0;(.9,0):(0,1)::} {MX}="tl" [r(4)] {X}="tr" [d(3)] {U^S\phi_!(X,x)}="br" [l(4)] {SX}="bl" "tl" [dr] {SMX}="itl" [r(2)] {SX}="itr" [l(1.5)d] {S^2X}="ib" 
"tl":"tr"^-{x}:"br"^-{U^M\eta^T_{(X,x)}}:@{<-}"bl"^-{U^Sq_{(X,x)}}:@{<-}"tl"^-{\phi_X}
"itl":"itr"^-{Sx}:@{<-}"ib"_(.55){\mu^S_X}:@{<-}"itl"^-{S\phi_X}
"tl":"itl"^-{\eta^S_{MX}} "tr":"itr"^-{\eta^S_X}:"br"_(.65){U^Sq_{(X,x)}} "bl":"ib"^-{\eta^S_{SX}} "bl":@/_{1.5pc}/"itr"_-{1}}}} [r(5)]
{\xybox{\xygraph{!{0;(2,0):(0,1)::} {\ca G(RU^R)}="tl" [r] {\ca GU^R}="tr" [d] {U^S\phi_{!}}="br" [l] {S\ca G(U^R)}="bl"
"tl":"tr"^-{\ca G\rho}:"br"^-{\ca G(U^R)\eta^T}:@{<-}"bl"^-{q}:@{<-}"tl"^-{{\phi}\ca G(U^R)}}}}} \]
clearly commutes, thus one has a commutative square as on the right in the previous display, and so the result follows.
\end{proof}
\begin{theorem}\label{thm:lift-mult}
\emph{({\bf Multitensor lifting theorem})} Let $\lambda$ be a regular cardinal and let $E$ be a $\lambda$-accessible distributive multitensor on a cocomplete category $V$. Then there is, to within isomorphism, a unique functor operad $(E',\sigma')$ on $V^{E_1}$ such that
\begin{enumerate}
\item  $(E',\sigma')$ is distributive.
\item  $\Enrich {E'} \iso \Enrich E$ over $\ca GV$.
\end{enumerate}
Moreover $E'$ is also $\lambda$-accessible.
\end{theorem}
\begin{proof}
Write $\psi:\widetilde{E_1} \hookrightarrow E$ for the multitensor inclusion of the unary part of $E$, and then apply lemma(\ref{lem:mnd-lift-mult}) with $S=\Gamma{E}$, $R=E_1$ and $\phi=\Gamma{\psi}$ to produce a $\lambda$-accessible distributive and path-like monad $T$ on $\ca G(V^{E_1})$ over $\Set$. Thus by theorem(\ref{thm:GammaE-characterisation}) $\overline{T}$ is a distributive multitensor on $V^{E_1}$ with $\Enrich {\overline{T}} \iso \Enrich E$. Moreover since $T \iso \Gamma {\overline{T}}$ it follows by lemma(\ref{lem:EAcc->GammaEAcc}) that $\overline{T}$ is $\lambda$-accessible. As for uniqueness suppose that $(E',\sigma')$ is given as in the statement. Then by theorem(\ref{thm:GammaE-characterisation}) $\Gamma(E')$ is a distributive monad on $\ca G(V^{E_1})$ and one has
\[ \ca G(V^{E_1})^{\Gamma(E')} \iso \Enrich E \]
over $\ca G(V^{E_1})$. By structure-semantics one has an isomorphism $\Gamma(E'){\iso}T$ of monads. Since $\Gamma$ is locally fully faithful, one thus has an isomorphism $E'{\iso}\overline{T}$ of multitensors as required.
\end{proof}
From the above proofs in the explicit construction of $E'$ one first obtains $\Gamma E'$. In particular one has a coequaliser
\begin{equation}\label{eq:coeq-for-lifting}
\begin{aligned}
\xygraph{*!(0,.35)=(0,.5){\xybox{\xygraph{!{0;(2.8,0):(0,.2)::} {\Gamma E(E_1X_1,...,E_1X_n)}="l" [r] {\Gamma E(X_1,...,X_n)}="m" [r(1.1)] {\Gamma E'((X_1,x_1),...,(X_n,x_n))}="r" "l":@<1.2ex>"m" "l":@<-1.2ex>"m":"r"}}}}
\end{aligned}
\end{equation}
in $\Enrich{E}$, and then one takes
\[ \opEpr\limits_{i=1}^n (X_i,x_i) = \Gamma E'((X_1,x_1),...,(X_n,x_n))(0,n). \]
The set of objects for each of the $E$-categories appearing in (\ref{eq:coeq-for-lifting}) is  $\{0,...,n\}$, and the morphisms are all identities on objects. The explicit construction of (\ref{eq:coeq-for-lifting}) at the level of $V$-graphs proceeds, as we have seen, by a transfinite construction. From this and the definition of $V$-graphs of the form $\Gamma E(X_1,...,X_n)$, it is clear that for all the $E$-categories appearing in (\ref{eq:coeq-for-lifting}), the hom between $a$ and $b$ for $a > b$ is initial. Thus to understand (\ref{eq:coeq-for-lifting}) completely it suffices to understand the homs between $a$ and $b$ for $a \leq b$.

By the explicit description of the monad $\Gamma E$, the hom of each stage of the transfinite construction in $\ca GV$ of the coequaliser (\ref{eq:coeq-for-lifting}), depends on the homs between $c$ and $d$ -- where $a \leq c \leq d \leq b$ -- of the earlier stages of the construction. Thus the hom between $a$ and $b$ of (\ref{eq:coeq-for-lifting}) is
\begin{equation}\label{eq:homs-of-coeq-for-lifting}
\begin{aligned}
\xygraph{*!(0,.75)=(0,1){\xybox{\xygraph{!{0;(2,0):(0,.5)::} {\smash{\opE\limits_{a{<}i{\leq}b}} E_1X_i}="l" [r] {\smash{\opE\limits_{a{<}i{\leq}b}} X_i}="m" [r] {\smash{\opEpr\limits_{a{<}i{\leq}b}} (X_i,x_i)}="r"
"l":@<1.2ex>"m"^-{\sigma} "l":@<-1.2ex>"m"_-{\opE\limits_i x_i}:"r"}}}}
\end{aligned}
\end{equation}
and moreover, by virtue of its dependence on the intermediate homs (ie between $c$ and $d$ as above), this will not simply be the process of taking (\ref{eq:homs-of-coeq-for-lifting}) to be the coequaliser in $V^{E_1}$. For instance when $E'$ is the Gray tensor product as in example(\ref{ex:Gray}) below, what we have here is a description of the Gray tensor product of 2-categories in terms of certain coequalisers in $\Enrich{\tn{Gray}}$.

However note that when one applies $E'$ to sequences of free $E_1$-algebras, as in
\[ \begin{array}{lll} {\opEpr\limits_{i=1}^n (E_1X_i,\sigma_{(X_i)})} & = & {\Gamma E'((E_1X_1,\sigma_{(X_1)}),...,(E_1X_n,\sigma_{(X_n)}))(0,n)}
\\ & {=} & {\Gamma E' \ca GE_1(X_1,...,X_n)(0,n)} \\ & = & {\Gamma E(X_1,...,X_n)(0,n) \, = \, \opE\limits_{i=1}^n X_i} \end{array} \]
one simply recovers $E$.

\subsection{Gray and Crans tensor products}
\label{ssec:GrayCrans-key-examples}
In the examples that we present in this section we shall use the following notation. We denote by $A$ the appropriate $\ca T_{\leq{n{+}1}}$-operad over $\Set$ and by $E$ the $\ca T_{\leq{n}}^{\times}$-multitensor associated to it by theorem(\ref{thm:Operad-Multitensor-Equiv}), so that one has
\[ \begin{array}{lccr} {A = \Gamma E} &&& {E = \overline{A}} \end{array} \]
and $\ca G^{n+1}(\Set)^A \iso \Enrich E$ over $\ca G^{n+1}(\Set)$. The monad $E_1$ on $\ca G^n(\Set)$ has as algebras the structure borne by the homs of an $A$-algebra. Theorem(\ref{thm:lift-mult}) produces the functor operad $E'$ on $\ca G^n(\Set)^{E_1}$ such that
\[ \ca G^{n{+}1}(\Set)^A \iso \Enrich {E'} \iso \Enrich {E} \]
over $\ca G^n(\Set)^{E_1}$. Moreover $E'$ is the unique such functor operad which is distributive. The first of our examples is the most basic.
\begin{example}\label{ex:strict-n-cat}
When $A$ is the terminal $\ca T_{\leq{n{+}1}}$-operad, $E$ is the terminal $\ca T_{\leq{n}}^{\times}$-multitensor, and so $E_1 = \ca T_{\leq{n}}$. Since strict $(n{+}1)$-categories are categories enriched in $\Enrich n$ using cartesian products, and these commute with coproducts (in fact all colimits), it follows by the uniqueness part of theorem(\ref{thm:lift-mult}) that $E'$ is just the cartesian product of $n$-categories. 
\end{example}
The general context in which this example can be generalised is that of a distributive law of a multitensor over a monad, as described in section(5.2) of \cite{Weber-MultitensorsMonadsEnrGraphs}. Recall that in theorem(5.2.1) of \cite{Weber-MultitensorsMonadsEnrGraphs} such a distributive law was identified with structure on the monad making it opmonoidal with respect to the multitensor, and that analogously to the usual theory of distributive laws between monads, one has a lifting of the multitensor to the category of algebras of the monad. As the following example explains, the two senses of the word ``lifting'' -- coming from the theory of distributive laws, and from the lifting theorem -- are in fact compatible.
\begin{example}\label{ex:opmonoidal}
Let $E$ be a multitensor on $V$ and $T$ be an opmonoidal monad on $(V,E)$. Then one has by theorem(5.2.1) of \cite{Weber-MultitensorsMonadsEnrGraphs} a lifted multitensor $E'$ on $V^T$. On the other hand if moreover $V$ is cocomplete, $E$ is a distributive and accessible functor operad, and $T$ is coproduct preserving and accessible, then $E'$ may also be obtained by applying theorem(\ref{thm:lift-mult}) to the composite multitensor $EM(T)$. When $E$ is given by cartesian product $EM(T)$ is just another name for the multitensor $T^{\times}$, making $E'$ the cartesian product of $T$-algebras by the uniqueness part of theorem(\ref{thm:lift-mult}) and proposition(2.8) of \cite{BataninWeber-EnHop}. Specialising further to the case $T = \ca T_{{\leq}n}$, we recover example(\ref{ex:strict-n-cat}).
\end{example}
In the above examples we used the uniqueness part of theorem(\ref{thm:lift-mult}) to enable us to identify the lifted multitensor $E'$ as the cartesian product. In each case we had the cartesian product on the appropriate category of algebras as a candidate, and the aforementioned uniqueness told us that this candidate was indeed our $E'$ because the resulting enriched categories matched up. In the absence of this uniqueness, in order to identify $E'$ one would have to unpack its construction, and as we saw in the proof of lemma(\ref{lem:mnd-lift-mult}), this involves a transfinite colimit construction in the appropriate category of enriched graphs. The importance of this observation becomes greater as the operads we are considering become more complex. We now come to our leading example.
\begin{example}\label{ex:Gray}
Take $A$ to be the $\ca T_{\leq{3}}$-operad for Gray categories constructed in \cite{Batanin-MonGlobCats} (example(4) after corollary(8.1.1)). Since $E_1$ is the monad on $\ca G^2(\Set)$ for 2-categories, in this case $E'$ is a functor operad for 2-categories. However the Gray tensor product of 2-categories \cite{Gray-FormalCatThy} is part of a symmetric monoidal closed structure. Thus it is distributive as a functor operad, and since Gray categories are categories enriched in the Gray tensor product by definition, it follows that $E'$ is the Gray tensor product.
\end{example}
Lemma(2.5) of \cite{Crans-BraidingsSyllepsesSymmetries} unpacks the notion of 4-tas (``tas'' being the singular form, ``teisi'' being the plural) in detail. This explicit description can be interpretted as an explicit description of the $\ca T_{\leq 4}$-operad for 4-teisi. On the other hand from \cite{Crans-ATensorForGrayCats}, one can verify that $\tensor_{\tn{Crans}}$ is distributive in the following way. First we note that to say that $\tensor_{\tn{Crans}}$ is distributive is to say that
\[ \id \tensor c_i : A \tensor_{\tn{Crans}} B_i \longrightarrow A \tensor_{\tn{Crans}} B \]
is a coproduct cocone, for all Gray categories $A$ and coproduct cocones $(B_i \xrightarrow{c_i} B \,\, : \,\, i \in I)$ of Gray categories, since $\tensor_{\tn{Crans}}$ is symmetric. Since the forgetful functor $\Enrich{\tn{Gray}} \to \ca G(\Enrich{2})$ creates coproducts, and using the explicit description of coproducts of enriched graphs, a discrete cocone $(C_i \xrightarrow{k_i} C \,\, : \,\, i \in I)$ in $\Enrich{\tn{Gray}}$ is universal iff it is universal at the level of objects and each of the $k_i$'s is fully faithful (in the sense that the hom maps are isomorphisms of 2-categories). From the explicit description of $\tensor_{\tn{Crans}}$ given in section(4) of \cite{Crans-ATensorForGrayCats}, one may witness that
\[ A \tensor_{\tn{Crans}} (-) : \Enrich{\tn{Gray}} \longrightarrow \Enrich{\tn{Gray}} \]
preserves fully faithful Gray-functors. Thus the distributivity of $\tensor_{\tn{Crans}}$ follows since at the level of objects $\tensor_{\tn{Crans}}$ is the cartesian product.
\begin{example}\label{ex:Crans}
Take $A$ to be the $\ca T_{\leq 4}$-operad for 4-teisi. The associated multitensor $E$ has $E_1$ equal to the $\ca T_{{\leq}3}$-operad for Gray categories. Thus theorem(\ref{thm:lift-mult}) constructs a functor operad $E'$ of Gray categories whose enriched categories are 4-teisi. As we explained above $\tensor_{\tn{Crans}}$ is distributive, and so the uniqueness part of theorem(\ref{thm:lift-mult}) ensures that $E'=\tensor_{\tn{Crans}}$, since teisi are categories enriched in the Crans tensor product by definition.
\end{example}

\subsection{Day convolution}
\label{ssec:convolution-via-lifting}
While this article is directed primarily at an improved understanding of the examples discussed above, within a framework that one could hope will lead to an understanding of the higher dimensional analogues of the Gray tensor product, it is interesting to note that Day convolution can be seen as an instance of the multitensor lifting theorem.

The set of multimaps $(X_1,...,X_n) \to Y$ in a given multicategory $\C$ shall be denoted as $\C(X_1,...,X_n;Y)$. Recall that a linear map in $\C$ is a multimap whose domain is a sequence of length $1$. The objects of $\C$ and linear maps between them form a category, which we denote as $\C_l$, and we call this the linear part of $\C$. The set of objects of $\C$ is denoted as $\C_0$. Given objects
\[ \begin{array}{lcccr} {A_{11}, ..., A_{1n_1}, ... ..., A_{k1},...,A_{kn_k}} && {B_1,...,B_k} && {C} \end{array} \]
of $\C$, we denote by
\[ {\sigma_{A,B,C} : \C(B_1,...,B_k;C) \times \prod\limits_i \C(A_{i1},...,A_{in_i};B_i) \rightarrow \C(A_{11},...,A_{kn_k};C)} \]
the substitution functions of the multicategory $\C$. One thus induces a function
\[ \sigma_{A,C} : \int^{B_1,...,B_k} \C(B_1,...,B_k;C) \times \prod\limits_i \C(A_{i1},...,A_{in_i};B_i) \rightarrow \C(A_{11},...,A_{kn_k};C) \]
in which for the purposes of making sense of this coend, the objects $B_1,...,B_k$ are regarded as objects of the category $\C_{l}$. A \emph{promonoidal category} in the sense of Day \cite{Day-Convolution}, in the unenriched context, can be defined as a multicategory $\C$ such that these induced functions $\sigma_{A,C}$ are all bijective. A \emph{promonoidal structure} on a category $\D$ is a promonoidal category $\C$ such that $\C_l = \D^{\op}$.

A lax monoidal category $(V,E)$ is \emph{cocomplete} when $V$ is cocomplete as a category and $E_n:V^n \to V$ preserves colimits in each variable for all $n \in \N$. In this situation the multitensor $E$ is also said to be cocomplete. When $\C$ is small it defines a functor operad on the functor category $[\C_l,\Set]$ whose tensor product $F$ is given by the coend
\[ \opF\limits_i X_i = \int^{C_1,...,C_n} \C(C_1,...,C_n;-) \times \prod\limits_i X_iC_i \]
and substitution is defined in the evident way from that of $\C$. By proposition(2.1) of \cite{DayStreet-Substitudes} $F$ is a cocomplete functor operad and is called the \emph{standard convolution} structure of $\C$ on $[\C_l,\Set]$. By proposition(2.2) of \cite{DayStreet-Substitudes}, for each fixed category $\D$, standard convolution gives an equivalence between multicategories on $\C$ such that $\C_l = \D$ and cocomplete functor operads on $[\D,\Set]$, which restricts to the well-known \cite{Day-Convolution} equivalence between promonoidal structures on $\D^{\op}$ and closed monoidal structures on $[\D,\Set]$.

We have recalled these facts in a very special case compared with the generality at which this theory is developed in \cite{DayStreet-Substitudes}. In that work all structures are considered as enriched over some nice symmetric monoidal closed base $\ca V$, and moreover rather than $\D = \C_l$ as above, one has instead an identity on objects functor $\D \to \C_l$. The resulting combined setting is then what are called $\ca V$-substitudes in \cite{DayStreet-Substitudes}, and in the $\ca V = \Set$ case the extra generality of the functor $\D \to \C_l$, corresponds at the level of multitensors, to the consideration of general closed multitensors on $[\D,\Set]$ instead of mere functor operads. We shall now recover standard convolution, for the special case that we have described above, from the lifting theorem.

Given a multicategory $\C$ we define the multitensor $E$ on $[\C_0,\Set]$ via the formula
\[ \left(\opE\limits_{1{\leq}i{\leq}n} X_i\right)(C) = \coprod\limits_{C_1,...,C_n} \left(\C(C_1,...,C_n;C) \times \prod\limits_{1{\leq}i{\leq}n} X_i(C_i)\right) \]
using the unit and compositions for $\C$ in the evident way to give the unit $u$ and substitution $\sigma$ for $E$. When $\C_0$ has only one element, this is the multitensor on $\Set$ coming from the operad $P$ described in \cite{BataninWeber-EnHop} and \cite{Weber-MultitensorsMonadsEnrGraphs}, whose tensor product is given by the formula
\[ \opE\limits_{1{\leq}i{\leq}n} X_i = P_n \times X_1 \times ... \times X_n.  \]
An $E$-category with one object is exactly an algebra of the coloured operad $P$ in the usual sense. A general $E$-category amounts to a set $X_0$, sets $X(x_1,x_2)(C)$ for all $x_1,x_2 \in X_0$ and $C \in \C_0$, and functions
\begin{equation}\label{eq:sc-cat} \C(C_1,...,C_n;C) \times \prod\limits_i X(x_{i-1},x_i)(C_i) \to X(x_0,x_n)(C) \end{equation}
compatible in the evident way with the multicategory structure of $\C$. On the other hand an $F$-category amounts to a set $X_0$, sets $X(x_1,x_2)(C)$ natural in $C$, and maps as in (\ref{eq:sc-cat}) but which are natural in $C_1,...,C_n,C$, and compatible with $\C$'s multicategory structure. However this added naturality enjoyed by an $F$-category isn't really an additional condition, because it follows from the compatibility with the linear maps of $\C$. Thus $E$ and $F$-categories coincide, and one may easily extend this to functors and so give $\Enrich E \iso \Enrich F$ over $\ca G[\C_0,\Set]$.

The unary part of $E$ is given on objects by
\[ E_1(X)(C) = \coprod\limits_{D} \C_l(D,C) \times X(D) \]
which should be familiar -- $E_1$ is the monad on $[\C_0,\Set]$ whose algebras are functors $\C_l \to \Set$, and may be recovered from left Kan extension and restriction along the inclusion of objects $\C_0 \hookrightarrow \C_l$. Thus the category of algebras of $E_1$ may be identified with the functor category $[\C_l,\Set]$. Since the multitensor $E$ is clearly cocomplete, it satisfies the hypotheses of theorem(\ref{thm:lift-mult}), and so one has a unique finitary distributive multitensor $E'$ on $[\C_l,\Set]$ such that $\Enrich E \iso \Enrich {E'}$ over $\ca G[\C_0,\Set]$. By uniqueness we have
\begin{proposition}\label{prop:convolution-via-lifting}
Let $\C$ be a multicategory, $F$ be the standard convolution structure on $[\C_l,\Set]$ and $E$ be the multitensor on $[\C_0,\Set]$ defined above. Then one has an isomorphism $F \iso E'$ of multitensors.
\end{proposition}
In particular when $\C$ is a promonoidal category proposition(\ref{prop:convolution-via-lifting}) expresses classical unenriched Day convolution as a lift in the sense of theorem(\ref{thm:lift-mult}).

\subsection{The 2-functoriality of multitensor lifting}
\label{ssec:functoriality-lifting}
We now express the lifting theorem as a coreflection to the inclusion of functor operads within a 2-category of multitensors which are sufficiently nice that theorem(\ref{thm:lift-mult}) can be applied to them. 

Recall \cite{Street-FTM} that when a 2-category $\ca K$ has Eilenberg-Moore objects, one has a 2-functor
\[ \begin{array}{lccr} {\tn{sem}_{\ca K} : \MND(\ca K) \longrightarrow \ca K^{[1]}} &&& {(V,T) \mapsto U^T:V^T \to V} \end{array} \]
which on objects sends a monad $T$ to the forgetful arrow $U^T$ which forms part of the Eilenberg-Moore object of $T$, and that a straight forward consequence of the universal property of Eilenberg-Moore objects is that $\tn{sem}_{\ca K}$ (``sem'' being short for ``semantics'') is 2-fully-faithful. In the case $\ca K = \CAT$ if one restricts attention to the sub-2-category of $\MND(\CAT)$ consisting of the 1-cells of the form $(1_{\ca E},-)$, then one refinds the structure-semantics result of Lawvere referred to earlier.

The 2-fully-faithfulness of $\tn{sem}_{\ca K}$ says that the one and 2-cells of the 2-category $\MND(\ca K)$ admit an alternative ``semantic'' description. Given monads $(V,T)$ and $(W,S)$ in $\ca K$, to give a monad functor $(H,\psi):(V,T) \rightarrow (W,S)$, is to give $\tilde{H}:V^T \rightarrow W^S$ such that $U^S\tilde{H}=HU^T$; and to give a monad 2-cell $\phi:(H_1,\psi_1) \rightarrow (H_2,\psi_2)$ is to give $\phi:H_1{\rightarrow}H_2$ and $\tilde{\phi}:\tilde{H_1}{\rightarrow}\tilde{H_2}$ commuting with $U^T$ and $U^S$. Note that Eilenberg-Moore objects in $\CAT/\Set$ are computed as in $\CAT$, and we shall soon apply these observations to the case $\ca K = \CAT/\Set$.

In view of theorem(\ref{thm:GammaE-characterisation}) of this article and proposition(4.4.2) of \cite{Weber-MultitensorsMonadsEnrGraphs}, the 2-functor $\Gamma$ restricts to a 2-equivalence
\begin{equation}\label{eq:2eq-from-Gamma}
\ADM \catequiv \AGMnd.
\end{equation}
Here $\ADM$ is the full sub-2-category of $\DISTMULT$ consisting of the $(V,E)$ such that $V$ is cocomplete and the multitensor $E$ is accessible. The 2-category $\AGMnd$ is defined as follows:
\begin{itemize}
\item objects are pairs $(V,T)$, where $V$ is a cocomplete category and $T$ is a monad on $\ca GV$ which is distributive, path-like and accessible.
\item an arrow $(V,T) \to (W,S)$ is a pair $(H,\psi)$, where $H:V \to W$ is a functor, and $\psi$ is a natural transformation making $(\ca GH,\psi):(\ca GV,T) \to (\ca GW,S)$ a morphism of $\MND(\CAT/\Set)$.
\item a 2-cell $(H,\psi) \to (K,\kappa)$ is a monad 2-cell $(\ca GH,\psi) \to (\ca GK,\kappa)$. 
\end{itemize}
Since by \cite{Weber-MultitensorsMonadsEnrGraphs} proposition(2.1.6), the effect $\ca G_1 : \CAT \longrightarrow \CAT/\Set$ of $\ca G$ on arrows into $1$ is locally fully faithful, the data of a 2-cell of $\AGMnd$ will be of the form $\ca G\phi$ for a unique natural transformation $\phi:H \to K$. Moreover from the recollections of the previous paragraph, the one and 2-cells of $\AGMnd$ can be reinterpretted in semantic terms.

Let us denote by $\AFO$ the full sub-2-category of $\ADM$ consisting of the $(V,E)$ such that $E$ is a functor-operad, and by
\[ I : \AFO \longrightarrow \ADM \]
the inclusion.
\begin{proposition}\label{prop:2-fun-lifting-thm}
The assignment $(V,E) \mapsto (V^{E_1},E')$, with $E'$ defined by theorem(\ref{thm:lift-mult}), is the object map of a right adjoint
\[ C : \ADM \longrightarrow \AFO \]
to the inclusion $I$.
\end{proposition}
\begin{proof}
We shall define the coreflection $C$ and a 2-natural transformation $\varepsilon:IC \to 1$ such that $C\varepsilon = \id$ and $\varepsilon I = \id$, so that $\varepsilon$ is the counit of an adjunction $I \ladj C$ whose unit is an identity. Let $(H,\psi):(W,F) \to (V,E)$ be a lax monoidal functor. Then we have the following serially commutative diagram of forgetful functors
\[ \xygraph{!{0;(2,0):(0,.5)::} {\Enrich{F}}="p1" [r] {\ca G(W^{F_1})}="p2" [r] {\ca GW}="p3" [d] {\ca GV}="p4" [l] {\ca G(V^{E_1})}="p5" [l] {\Enrich{E}}="p6"
"p1":"p2"^-{U^{F'}}:"p3"^-{\ca G(U^{F_1})}:"p4"^-{\ca GH}:@{<-}"p5"^-{\ca G(U^{E_1})}:@{<-}"p6"^-{U^{E'}}:@{<-}"p1"^-{\psi^*} "p2":"p5"^{\ca G(\psi_1^*)}
"p1":@/^{2.2pc}/"p3"^-{U^F} "p6":@/_{2.2pc}/"p4"_-{U^E}} \]
The left-most square is the semantic interpretation of a morphism of monads
\[ (\ca G(\psi_1^*),\tilde{\psi}):(\ca G(W^{F_1}),\Gamma F') \longrightarrow (\ca G(V^{E_1}),\Gamma E') \]
and so by (\ref{eq:2eq-from-Gamma}), there is a unique natural transformation $\psi'$ giving the coherence data of a lax monoidal functor
\[ (\psi^*_1,\psi'):(W^{F_1},F') \longrightarrow (V^{E_1},E') \]
such that $\Gamma(\psi^*_1,\psi')=(\ca G(\psi_1^*),\tilde{\psi})$. This enables us to define the 1-cell map of $C$ as $C(H,\psi)=(\psi^*_1,\psi')$, and similarly using the semantic interpretation of the 2-cells of $\AGMnd$ and (\ref{eq:2eq-from-Gamma}), one defines the 2-cell map of $C$. The 2-functoriality of $C$ follows from that of $\Gamma$ and $\tn{sem}_{\CAT/\Set}$.

For $(V,E) \in \ADM$ one has
\[ \xygraph{!{0;(2,0):(0,.5)::} {\Enrich{E}}="tl" [r] {\ca G(V^{E_1})}="tr" [d] {\ca GV}="br" [l] {\Enrich{E}}="bl" "tl":"tr"^-{U^{E'}}:"br"^-{\ca G(U^{E_1})}:@{<-}"bl"^-{U^{E}}:@{<-}"tl"^-{1}} \]
which is the semantic interpretation of
\[ (U^{E_1},\overline{\varepsilon_E}) : (V^{E_1},E') \longrightarrow (V,E) \, \in \, \ADM \]
and we take this to be the component $\varepsilon_{(V,E)}$ of $\varepsilon$. To verify naturality with respect to $(H,\psi)$ as above, note that $\tn{sem}_{\CAT/\Set}(\varepsilon_{(V,E)}(\psi_1^*,\psi'))$ is the composite morphism $U^{F'} \to U^{E_1}$ depicted on the left in
\[ \xygraph{{\xybox{\xygraph{!{0;(2,0):(0,.5)::} {\Enrich{F}}="p1" [r] {\ca G(W^{F_1})}="p2" [d] {\ca G(V^{E_1})}="p3" [d] {\ca GV}="p4" [l] {\Enrich{E}}="p5" [u] {\Enrich{E}}="p6"
"p1":"p2"^-{U^{F'}}:"p3"^-{\ca G(\psi_1^*)}:"p4"^-{\ca G(U^{E_1})}:@{<-}"p5"^-{U^E}:@{<-}"p6"^-{1}:@{<-}"p1"^-{\psi^*} "p6":"p3"^-{U^{E'}}}}} [r(4)]
{\xybox{\xygraph{!{0;(2,0):(0,.5)::} {\Enrich{F}}="p1" [r] {\ca G(W^{F_1})}="p2" [d] {\ca GW}="p3" [d] {\ca GV}="p4" [l] {\Enrich{E}}="p5" [u] {\Enrich{F}}="p6"
"p1":"p2"^-{U^{F'}}:"p3"^-{\ca G(U^{F_1})}:"p4"^-{\ca GH}:@{<-}"p5"^-{U^E}:@{<-}"p6"^-{\psi^*}:@{<-}"p1"^-{1} "p6":"p3"^-{U^F}}}}} \]
whereas $\tn{sem}_{\CAT/\Set}((H,\psi)\varepsilon_{(W,F)})$ is the composite depicted on the right. Since $U^{E_1}\psi_1^* = HU^{F_1}$ these are equal and so $\varepsilon$ is natural with respect to $(H,\psi)$ by the fully faithfulness of $\tn{sem}_{\CAT/\Set}$. The 2-naturality of $\varepsilon$ follows similarly using the 2-fully faithfulness of $\tn{sem}_{\CAT/\Set}$. When $E$ is itself a functor operad one has that $E_1$ is the identity monad, and so $U^{E_1}$ and thus $\varepsilon_{(V,E)}$ are identities, whence $\varepsilon I = \id$. On the other hand for any $(V,E) \in \ADM$, $\overline{\varepsilon_E}^*_1 = \id$ by definition, and so $C\varepsilon_{(V,E)} = \id$.
\end{proof}
As pointed out in section(6.1) of \cite{Weber-MultitensorsMonadsEnrGraphs}, the data $\phi:E \to F$ of a morphism of multitensors on some category $V$ can also be regarded as the coherences making the identity functor $1_V$ into a lax monoidal functor $(1_V,\phi):(V,F) \to (V,E)$, or alternatively as the coherences making $1_V$ into an oplax monoidal functor $(1_V,\phi):(V,E) \to (V,F)$.
\begin{remark}\label{rem:product-comparison-unreduced-context}
Suppose that $\psi:E \rightarrow \ca T^{\times}_{{\leq}n}$ is a $\ca T^{\times}_{{\leq}n}$-multitensor. Then the datum $\psi$ can also be regarded as the coherence data for a lax monoidal functor
\[ (1_{\ca G^n(\Set)},\psi) : (\ca G^n(\Set),\ca T^{\times}_{{\leq}n}) \longrightarrow (\ca G^n(\Set),E), \]
applying $C$ to it gives a lax monoidal functor
\[ \begin{array}{c} {(\psi_1^*,\psi') : (\ca G^n(\Set)^{\ca T_{\leq n}},\prod) \longrightarrow (\ca G^n(\Set)^{E_1},E')} \end{array} \]
and the components of $\psi'$ are morphisms
\[ \begin{array}{c} {\psi'_{(X_1,x_1),...,(X_n,x_n)} : {\opE\limits_i}' \psi_1^*(X_i,x_i) \longrightarrow \prod\limits_{i=1}^n \psi_1^*(X_i,x_i)} \end{array} \]
of $E_1$-algebras defined for each sequence $((X_1,x_1),...,(X_n,x_n))$ of strict $n$-categories. The codomain of these morphisms can be written as above since $\psi_1^*$ as a right adjoint preserves products. This gives a general comparison map between the functor operad $E'$ produced by theorem(\ref{thm:lift-mult}) and cartesian products, defined for sequences of $E_1$-algebras that underlie strict $n$-categories. In particular when $E$ is as in example(\ref{ex:Gray}), then $\psi_1^*$ is the identity, and one has the comparison maps between the Gray tensor product and the cartesian product of 2-categories.
\end{remark}
One of the expected features of the higher dimensional analogues of the Gray tensor product are such comparisons with cartesian product, which moreover are expected to be equivalences in the appropriate higher categorical sense. We shall see in remark(\ref{rem:product-comparisons-reduced-context}) below, that in the context of reduced $\ca T_{\leq n}$-operads, one always gets such comparison maps with the cartesian product, and these are defined for \emph{all} $E_1$-algebras.

\section{Weak $n$-categories via algebraic weak factorisation systems}
\label{sec:weak-n-cat-review}

\subsection{Overview}
\label{ssec:ncatopreview-overview}
In this section we review the operadic definition of weak $n$-category (more precisely Leinster's variant \cite{Leinster-HDA-book} of Batanin's \cite{Batanin-MonGlobCats} original definition) in a way that allows adaptation to the strictly unital case in section(\ref{sec:strictly-unital-weak-n-cats}). The overall scheme of this definition is as follows. One has the category $\NColl {\ca T_{\leq n}}$ of $\ca T_{\leq n}$-collections over $\Set$, and this is a presheaf category. There are two finitary monads on $\NColl {\ca T_{\leq n}}$, $\tn{Opd}_{\leq n}$ and $\tn{Cont}_{\leq n}$, whose algebras are $\ca T_{\leq n}$-operads over $\Set$ and $\ca T_{\leq n}$-collections over $\Set$ with chosen contractions (in the sense to be recalled below), respectively. The coproduct of these monads $\tn{Opd}_{\leq n} \coprod \tn{Cont}_{\leq n}$ exists and is finitary, and its algebras are $\ca T_{\leq n}$-operads over $\Set$ equipped with chosen contractions. As the category of algebras of a finitary monad on a presheaf 
category, the category of $\ca T_{\leq 
n}$-operads over $\Set$ with chosen contractions is locally finitely presentable, and so has an initial object $\ca K_{\leq n}$. A weak $n$-category is by definition an algebra for $\ca K_{\leq n}$.

The monad $\tn{Opd}_{\leq n}$ can be thought about in a few ways. Most generally, one has a finite limit sketch $S_1$ whose $\Set$-valued models are $\ca T_{\leq n}$-operads over $\Set$, another such sketch $S_2$ (this time without any distinguished limit cones) for $\ca T_{\leq n}$-collections over $\Set$, and an inclusion $S_2 \hookrightarrow S_1$ of finite limit sketches which induces the finitary monad $\tn{Opd}_{\leq n}$. Another viewpoint is that $\NOp{\ca T_{\leq n}}$ is the category of monoids for the substitution tensor product on $\NColl{\ca T_{\leq n}}$, this tensor product being filtered colimit preserving in one variable and cocontinuous in the other, and so by general well-known results such as those of \cite{Kelly-Transfinite}, the monoid monad exists and is finitary. Most particularly, an explicit description of $\tn{Opd}_{\leq n}$ is given in \cite{Leinster-HDA-book} Appendix D, from which its finitariness and many other desirable properties may be witnessed.

As for $\tn{Cont}_{\leq n}$, our preferred approach will be that of \cite{Garner-LeinsterDefinition}, which is to use the theory of cofibrantly-generated algebraic weak factorisation systems from \cite{Garner-SmallObjectArgument} to give a conceptual understanding of this monad. We shall spend some time reviewing the relevant aspects of \cite{Garner-SmallObjectArgument} as we shall use this machinery in section(\ref{sec:strictly-unital-weak-n-cats}).

\subsection{Algebraic weak factorisation systems}
\label{ssec:ncatopreview-awfs}
Composition in any category $\ca E$ can be regarded as a morphism $\tn{comp}_{\ca E}$ of spans of categories
\[ \xygraph{!{0;(2,0):(0,.3)::} {\ca E}="p1" [ur] {\ca E^{[2]}}="p2" [dr] {\ca E}="p3" [dl] {\ca E^{[1]}}="p4" "p1":@{<-}"p2"^-{}:"p3"^-{}:@{<-}"p4"^-{t_{\ca E}}:"p1"^-{s_{\ca E}} "p2":"p4"^{\tn{comp}_{\ca E}}} \]
where $s_{\ca E}$ (resp. $t_{\ca E}$) takes the source (resp. target) of a given morphism of $\ca E$. A \emph{functorial factorisation} is a morphism of spans that is a section of $\tn{comp}_{\ca E}$. On objects a given functorial factorisation $F:\ca E^{[1]} \to \ca E^{[2]}$ takes an arrow $f:X \to Y$ of $\ca E$ to a composable pair
\[ X \xrightarrow{Lf} Kf \xrightarrow{Rf} Y \]
of morphisms whose composite is $f$. With the evident morphisms one has a category $\tn{FF}(\ca E)$ of functorial factorisations for $\ca E$. For a given $F$ the assignment $f \mapsto Lf$ defines a copointed endofunctor $\varepsilon:L \to 1$ on $\ca E^{[1]}$ such that $s_{\ca E}\varepsilon=\id_{s_{\ca E}}$, where the $f$-component of $\varepsilon$ is given as on the left in
\[ \xygraph{{\xybox{\xygraph{{X}="tl" [r] {X}="tr" [d] {Y}="br" [l] {Kf}="bl" "tl":"tr"^-{1_X}:"br"^-{f}:@{<-}"bl"^-{Rf}:@{<-}"tl"^-{Lf}}}} [r(3)]
{\xybox{\xygraph{{X}="tl" [r] {Kf}="tr" [d] {Y}="br" [l] {Y}="bl" "tl":"tr"^-{Lf}:"br"^-{Rf}:@{<-}"bl"^-{1_Y}:@{<-}"tl"^-{f}}}}} \]
and similarly, the assignment $f \mapsto Rf$ defines a pointed endofunctor $\eta:1 \to R$ on $\ca E^{[1]}$ such that $t_{\ca E}\eta=\id_{t_{\ca E}}$, with the $f$-component of $\eta$ given as on the right in the previous display. Clearly these give three equivalent viewpoints on the notion, that is to say, to give a functorial factorisation $F$, is to give $(L,\varepsilon)$ such that $s_{\ca E}\varepsilon=\id_{s_{\ca E}}$, which in turn is equivalent to giving $(R,\eta)$ such that $t_{\ca E}\eta=\id_{t_{\ca E}}$. Thus (following \cite{Garner-SmallObjectArgument}) we shall also denote $F$ as the pair $(L,R)$. The assignment $f \mapsto Kf$ is the object map of a functor $K:\ca E^{[1]} \to \ca E$ that we call the \emph{image} of a given functorial factorisation.

Composition of the associated pointed endofunctors gives $\tn{FF}(\ca E)$ a strict monoidal structure, with tensor product denoted as $\comp_{r}$, and whose unit is initial. Composition of the associated copointed endofunctors gives $\tn{FF}(\ca E)$ another strict monoidal structure, with terminal unit and tensor product denoted as $\comp_l$. A key observation of \cite{Garner-SmallObjectArgument} is that one has natural maps
\[ \lambda_{F_1,F_2,F_3,F_4} : (F_1 \comp_l F_2) \comp_r (F_3 \comp_l F_4) \longrightarrow (F_1 \comp_r F_3) \comp_l (F_2 \comp_r F_4) \]
providing the interchange maps of a duoidal structure. Recall that a \emph{duoidal} structure \cite{AguiarMahajan-MonoidalFunctorsSpeciesHopf} \cite{BataninMarkl-HomotopyCentres} on a category $\ca V$ is a pair of monoidal structures $(I_1,\tensor_1,\alpha_1,\lambda_1,\rho_1)$ and $(I_2,\tensor_2,\alpha_2,\lambda_2,\rho_2)$, together with morphisms $\eta:I_1 \to I_2$, $\mu:I_2 \tensor I_2 \to I_2$ and $\delta:I_1 \to I_1 \tensor_2 I_1$, and natural morphisms
\[ \lambda_{A,B,C,D} : (A \tensor_2 B) \tensor_1 (C \tensor_2 D) \longrightarrow (A \tensor_1 C) \tensor_2 (B \tensor_1 D) \]
called interchange maps, such that
\begin{itemize}
\item $(\eta,\mu)$ makes $I_2$ a $\tensor_1$-monoid.
\item $(\eta,\delta)$ makes $I_1$ a $\tensor_2$-comonoid.
\item $(\iota,\delta)$ are the coherences making $\tensor_2 : \ca V \times \ca V \to \ca V$ lax monoidal for $\tensor_1$.
\item $(\iota,\mu)$ are the coherences making $\tensor_1 : \ca V \times \ca V \to \ca V$ oplax monoidal for $\tensor_2$.
\end{itemize}
The key observation alluded to above is precisely that one has a duoidal structure on $\tn{FF}(\ca E)$ for which $I_1 = 0$, $I_2 = 1$, $\tensor_1 = \comp_r$ and $\tensor_2 = \comp_l$.

A duoidal category is a natural environment for a notion of bialgebra. A \emph{bialgebra} structure on $X$ in a general duoidal category $\ca V$ consists of morphisms $i:I_1 \to X$, $m:X \tensor_1 X \to X$, $c:X \to I_2$ and $d:X \to X \tensor_2 X$, such that $(i,m)$ is a $\tensor_1$-monoid structure, $(c,d)$ is a $\tensor_2$-comonoid structure, and these are compatible in the following equivalent ways:
\begin{itemize}
\item $i$ and $m$ are morphisms of $\tensor_2$-comonoids.
\item $c$ and $d$ are morphisms of $\tensor_1$-monoids.
\end{itemize}
An \emph{algebraic weak factorisation system} on $\ca E$ (originally called a \emph{natural} weak factorisation system in \cite{Garner-SmallObjectArgument}) is by definition a bialgebra in the duoidal category $\tn{FF}(\ca E)$.

In particular for an algebraic weak factorisation system $F$, the copointed endofunctor $(L,\varepsilon)$ has in addition a comultiplication making it into a comonad, and the pointed endofunctor $(R,\eta)$ has a multiplication making it into a monad. Thus the given factorisation of $f:X \to Y$ as
\[ X \xrightarrow{Lf} Kf \xrightarrow{Rf} Y \]
is into a map $Lf$ which has the structure of an $L$-coalgebra, followed by the map which has the structure of an $R$-algebra. In general $L$-coalgebras and $R$-algebras admit lifting conditions with respect to each other, that is to say, given a commutative square as on the left in
\[ \xygraph{{\xybox{\xygraph{!{0;(1.5,0):(0,1.3333)::} {A}="tl" [r] {C}="tr" [d] {D}="br" [l] {B}="bl" "tl":"tr"^-{u}:"br"^-{r}:@{<-}"bl"^-{v}:@{<-}"tl"^-{l}}}} [r(4)]
{\xybox{\xygraph{!{0;(1.5,0):(0,.6667)::} {A}="p1" [r] {C}="p2" [d] {Kr}="p3" [d] {D}="p4" [l] {B}="p5" [u] {Kl}="p6"
"p1":"p2"^-{u}:"p3"^-{Lr}:"p4"^-{Rr}:@{<-}"p5"^-{v}:@{<-}"p6"^-{Rl}:@{<-}"p1"^-{Ll}
"p5":@{.>}@<-1.5ex>"p6"_{\lambda}:@{.>}"p3"^-{K(u,v)}:@{.>}@<1.5ex>"p2"^{\rho}}}}} \]
together with an $L$-coalgebra structure $\lambda$ on $l$ and an $R$-algebra structure $\rho$ on $r$, one constructs the diagonal filler as on the right in the previous display. Thus everything works as with the usual theory of weak factorisation systems, except that here instead of classes of left and right maps one has the categories of $L$-coalgebras and $R$-algebras, and the liftings are constructed as shown above using the algebraic data. In fact as explained in \cite{Garner-SmallObjectArgument}, every algebraic weak factorisation system has an underlying weak factorisation system.

For any $X \in \ca E$, $L$ restricts to a comonad $L_X$ on the coslice $X/\ca E$ whose coalgebras are $L$-coalgebras with domain $X$. Similarly $R$ restricts to a monad $R_X$ on the slice $\ca E/X$ whose algebras are $R$-algebras with codomain $X$. In particular if $\ca E$ has an initial object $0$, then $L_0$ is a comonad on $\ca E$ whose coalgebras are objects $X$ of $\ca E$ together with the structure of an $L$-coalgebra on the unique map $0 \to X$. This is the algebraic analogue of cofibrant replacement, when the underlying weak factorisation system is the factorisation of morphisms of a Quillen model category into cofibrations followed by trivial fibrations. Dually when $\ca E$ has a terminal object $1$, the monad $R_1$ on $\ca E$ is an algebraic analogue of fibrant replacement, when the underlying weak factorisation system is the factorisation of morphisms of a Quillen model category into trivial cofibrations followed by fibrations.

Given a functorial factorisation $F = (L,R)$ with image denoted as $K$, since colimits in $\ca E^{[1]}$ and $\ca E^{[2]}$ are computed componentwise as in $\ca E$, the following statements are clearly equivalent for a given category $\ca C$:
\begin{itemize}
\item $F:\ca E^{[1]} \to \ca E^{[2]}$ preserves limits (resp. colimits) of functors from $\ca C$.
\item $L:\ca E^{[1]} \to \ca E^{[1]}$ preserves limits (resp. colimits) of functors from $\ca C$.
\item $R:\ca E^{[1]} \to \ca E^{[1]}$ preserves limits (resp. colimits) of functors from $\ca C$.
\item $K:\ca E^{[1]} \to \ca E$ preserves limits (resp. colimits) of functors from $\ca C$.
\end{itemize}
When any of these conditions is satisfied, we say that the functorial factorisation preserves the given limit or colimit. In particular given a regular cardinal $\lambda$ an algebraic weak factorisation system is said to be \emph{$\lambda$-ary} when its underlying functorial factorisation preserves $\lambda$-filtered colimits, and \emph{finitary} in the case $\lambda = \aleph_0$ of filtered colimits. Note that for a finitary algebraic weak factorisation system $(L,R)$, the comonad $L_0$ and the right monad $R_1$ are both finitary. Below we shall exhibit $\tn{Cont}_{\leq n}$ as the monad $R_1$ for a finitary algebraic weak factorisation system $(L,R)$ on $\NColl{\ca T_{\leq n}}$.

Since $0 \in \tn{FF}(\ca E)$ has underlying monad the identity, it preserves all limits and colimits as a functorial factorisation, and similarly for $1 \in \tn{FF}(\ca E)$ since its underlying comonad is the identity. Thus if we write $\tn{FF}(\ca E)'$ for the full subcategory of $\tn{FF}(\ca E)$ consisting of the functorial factorisations preserving some class of limits and colimits, then the duoidal structure of $\tn{FF}(\ca E)$ restricts to $\tn{FF}(\ca E)'$. So the algebraic weak factorisation system's which preserve the limits and colimits in the given class could equally well be regarded as bialgebras in $\tn{FF}(\ca E)'$.

We conclude this section with an observation that will be of use in section(\ref{ssec:ncatopreview-smallobj}).
\begin{lemma}\label{lem:pres-FF-from-opfib}
Let $P:\ca A \to \ca B$ be an opfibration, and let $F = (L,R)$ be the functorial factorisation on $\ca A$ of a map into a cocartesian arrow for $P$ followed by a vertical arrow (ie one sent to an identity by $P$). Then $F$ preserves any colimit that $P$ preserves.
\end{lemma}
\begin{proof}
For any functor $P:\ca A \to \ca B$ one has the functor $i_P:\ca A \to P/\ca B$ given on objects by $PA = (A,1_{PA},PA)$, and the structure of an opfibration amounts to a left adjoint $\pi:P/\ca B \to \ca A$ to $i_P$ over $\ca B$ (see \cite{Street-FibrationIn2cats} for instance). One also has the canonical functor
\[ \begin{array}{lccr} {P' : \ca A^{[1]} \to P/\ca B} &&& {A \xrightarrow{\alpha} A' \,\, \mapsto \,\, (A,P\alpha,PA')} \end{array} \]
and the image for the cocartesian-vertical factorisation for $P$ is the composite $\pi P'$. Now the forgetful functor $P/\ca B \to \ca A \times \ca B$ creates all colimits that are preserved by $P$. Thus $P'$ preserves all colimits that $P$ does since colimits in $\ca E^{[1]}$ are formed componentwise as in $\ca E$, and so the result follows since $\pi$ as a left adjoint preserves all colimits.
\end{proof}

\subsection{Garner's small object argument}
\label{ssec:ncatopreview-smallobj}
The algebraic weak factorisation system's relevant for us will all be cofibrantly generated, in the sense that they all arise from Garner's small object argument, which we now recall in a special case of interest to us (for the general theory see \cite{Garner-SmallObjectArgument}). Suppose that $\ca E$ is locally finitely presentable and that a functor $J : \ca I \to \ca E^{[1]}$ is given, where $\ca I$ is a small category. Then Garner's small object argument produces an algebraic weak factorisation system on $\ca E$ in three steps. Before we begin to recall these, we invite the reader to keep in mind the more familiar special case when $\ca I$ is discrete, and so $J$ may be regarded as a set of maps
\[ J = \{s_i \xrightarrow{\phi_i} d_i \,\, : \,\, i \in \ca I\} \]
that one might call the ``generating cofibrations'' of the resulting algebraic weak factorisation system.

The first step is to take the left Kan extension of $J$ along itself to produce $L_1:\ca E^{[1]} \to \ca E^{[1]}$, which comes with a canonical comonad structure (comonads arising this way are commonly known as ``density'' comonads). The underlying copointed endofunctor $(L_1,\varepsilon_1)$ of this comonad does not satisfy the condition $s_{\ca E}\varepsilon_1 = \id_{s_{\ca E}}$ that would make it part of a functorial factorisation, and the point of the second step is to force this condition. Observe that $s_{\ca E} : \ca E^{[1]} \to \ca E$ is an opfibration, in which the cocartesian arrows are those whose underlying square in $\ca E$ is a pushout{\footnotemark{\footnotetext{Note that for any $\ca E$, $s_{\ca E}$ is a fibration, but what is important for us is that pushouts in $\ca E$ make $s_{\ca E}$ into an \emph{op}fibration.}}}. Then the second step is carried out by factoring the counit components $\varepsilon_{1,f} : L_1f \to f$ into
\[ L_1f \xrightarrow{c_f} L_2f \xrightarrow{\varepsilon_{2,f}} f \]
a cocartesian arrow followed by a vertical arrow for $s_{\ca E}$. In this way $(L_2,\varepsilon_2)$ satisfies the condition $s_{\ca E}\varepsilon_2 = \id_{s_{\ca E}}$ and so we have a functorial factorisation. Moreover $L_2$ inherits a comultiplication from that of $L_1$, making it into a comonad. What is still missing is the compatible monad structure on $R$, and the third step of the construction will supply this.

Before proceeding to describe this third step, let us note that the factorisation provided by the first two steps yields something very familiar in the case where $\ca I$ is discrete. For suppose that $f:X \to Y$ is in $\ca E$, then $L_1f$ and $L_2f$ are as indicated on the left in
\[ \xygraph{{\xybox{\xygraph{!{0;(4,0):(0,.4)::} {\smash{\coprod\limits_{i,\alpha,\beta}} s_i}="tl" [r] {X}="tr" [d] {Y}="br" [l] {\smash{\coprod\limits_{i,\alpha,\beta}} d_i}="bl" "tl"(:"tr"^-{(\alpha)}:"br"^-{f},[d(.25)] :"bl"_-{L_1f}:"br"_-{(\beta)}) "br" [u(.4)l(.4)] {P_f}="po" "bl":"po" "tr":"po"_-{L_2 f}:"br"^-{R_2f} "po" [u(.2)l(.1)] (-[r(.05)],-[d(.1)])}}} [r(5)]
{\xybox{\xygraph{{s_i}="tl" [r] {X}="tr" [d] {Y}="br" [l] {d_i}="bl" "tl":"tr"^-{\alpha}:"br"^-{f}:@{<-}"bl"^-{\beta}:@{<-}"tl"^-{\phi_i}:@{}"br"|-{=}}}}} \]
the summands are indexed over the set of triples $(i,\alpha,\beta)$ giving rise to a commutative square as indicated on the right in the previous display. So we have here what is usually regarded as the first step of Quillen's small object argument. In general, that is for $\ca I$ not necessarily discrete, the above coproducts will be replaced by more general colimits.

On completing the first two steps we have produced a functorial factorisation $(L_2,R_2)$ in which $L_2$ is a comonad, in other words we have a $\comp_l$-comonoid in $\tn{FF}(\ca E)$. The third step of Garner's small object argument is to take the free bialgebra on this comonoid, and this encapsulates the algebraic analogue of the transfinite part of Quillen's small object argument.

Let us see why this free bialgebra exists in our setting with $\ca E$ locally finitely presentable. First we note that since $\ca I$ is small and every object of $\ca E$ has rank, the functor $J$ factors through the inclusion $\ca E_{\lambda} \hookrightarrow \ca E$ of the $\lambda$-presentable objects where $\lambda$ is some regular cardinal. Thus $L_1$ is the left extension of its restriction to $\ca E_{\lambda}$, and so it preserves $\lambda$-filtered colimits. Since the opfibration $s_{\ca E} : \ca E^{[1]} \to \ca E$ preserves all colimits and thus $\lambda$-filtered ones in particular, by lemma(\ref{lem:pres-FF-from-opfib}) $L_2$ preserves $\lambda$-filtered colimits. So in fact $(L_2,R_2)$ is a $\comp_l$-comonoid in $\tn{FF}_{\lambda}(\ca E)$, the duoidal category of $\lambda$-ary functorial factorisations on $\ca E$.

In $\tn{FF}_{\lambda}(\ca E)$ the tensor products $\comp_l$ and $\comp_r$ are cocontinuous in one variable and $\lambda$-filtered colimit preserving in the other, since in each case these are just obtained by the composition of some kind of $\lambda$-ary functors. Because of the compatibility of these tensor products within the duoidal structure, the tensor product $\comp_r$ lifts to a tensor product on the category $\comp_l{\tn{-CoMon}}_{\lambda}(\ca E)$ of $\lambda$-ary $\comp_l$-comonoids, and since the forgetful functor $U_l : \comp_l{\tn{-CoMon}}_{\lambda}(\ca E) \to \tn{FF}_{\lambda}(\ca E)$ is strict monoidal and creates colimits, $\comp_r$ as a tensor product on $\comp_l{\tn{-CoMon}}_{\lambda}(\ca E)$ is also cocontinuous in one variable and $\lambda$-filtered colimit preserving in the other. Since $\tn{FF}_{\lambda}(\ca E)$ and thus $\comp_l{\tn{-CoMon}}_{\lambda}(\ca E)$ is evidently cocomplete, one can apply the construction of the free monoid on a pointed object from \cite{Kelly-Transfinite} to $(
L_2,R_2) \in \comp_l{\tn{-CoMon}}_{\lambda}(\ca E)$, to obtain the free bialgebra.

One reason we have discussed these details is to make the following observation.
\begin{lemma}\label{lem:fin-cofgen-algebraic weak factorisation system}
Suppose that $J:\ca I \to \ca E^{[1]}$ where $\ca E$ is locally finitely presentable, $\ca I$ is small and for all $i \in \ca I$, the domain and codomain of $J(i)$ is finitely presentable. Then the algebraic weak factorisation system obtained from $J$ via Garner's small object argument is finitary.
\end{lemma}
\begin{proof}
In this case the $\lambda$ of the above discussion is $\aleph_0$ and so the third part of the construction stays within the category of finitary functorial factorisations.
\end{proof}
Note that in the general construction of the free bialgebra, one can think of it as taking place just within $\tn{FF}_{\lambda}(\ca E)$, ie one can forget about the $\comp_l$-comonoid structures, since $U_l$ is strict monoidal and creates colimits, so this extra structure is \emph{automatically} compatible with everything that is going on. It is also worth noting that the passage from $(L,R) \in \tn{FF}(\ca E)$ to $R_X$, for all $X \in \ca E$, preserves colimits and sends $\comp_r$ to composition of pointed endofunctors of $\ca E/X$, and so for the $(L,R)$ obtained from Garner's small object argument, $R_X$ is the free monad on the pointed endofunctor $(R_2)_X$.

The algebraic weak factorisation system obtained from $J:\ca I \to \ca E^{[1]}$ via Garner's small object argument is said to be \emph{cofibrantly generated} by $J$. For such algebraic weak factorisation systems one has a direct description of what an $R$-algebra structure on a given $f:X \to Y$ amounts to. In order to describe this let us denote by $i \mapsto s(i) \xrightarrow{\phi_i} d(i)$ the object map of $J$. Then to give an $R$-algebra structure on $f$ is to give a choice of diagonal fill
\[ \xygraph{!{0;(1.5,0):(0,.6667)::} {si}="tl" [r] {X}="tr" [d] {Y}="br" [l] {di}="bl" "tl":"tr"^-{\alpha}:"br"^-{f}:@{<-}"bl"^-{\beta}:@{<-}"tl"^-{\phi_i} "bl":@{.>}"tr"|-{\gamma(i,\alpha,\beta)}} \]
for every $(i,\alpha,\beta)$ such that $f\alpha = \beta\phi_i$; and these are compatible in the sense that given $\delta:j \to i$ in $\ca I$ and $\alpha$ and $\beta$ as above, one has
\[ \xygraph{!{0;(2.5,0):(0,.5)::} {sj}="p1" [r] {si}="p2" [r] {X}="p3" [d] {Y}="p4" [l] {di}="p5" [l] {dj}="p6"
"p1":"p2"^-{s\delta}:"p3"^-{\alpha}:"p4"^-{f}:@{<-}"p5"^-{\beta}:@{<-}"p6"^-{d\delta}:@{<-}"p1"^-{\phi_j} "p2":"p5"_(.35){\phi_i}
"p5":@{.>}"p3"|-{\gamma(i,\alpha,\beta)} "p6":@{.>}"p3"|(.3){\gamma(j,\alpha s(\delta),\beta d(\delta))}} \]
$\gamma(i,\alpha,\beta)d(\delta) = \gamma(j,\alpha s(\delta),\beta d(\delta))$. This compatibility condition is one of the novel features of the theory of algebraic weak factorisation systems, which owes its existence to the fact that one has a \emph{category} of ``generating cofibrations'' as opposed to a mere set. We shall use this feature below in section(\ref{sec:strictly-unital-weak-n-cats}) to express what it means for unital operations to be ``strict'' in our operadic definition of ``weak $n$-category with strict units''.

\subsection{$\ca T_{\leq n}$-collections}
\label{ssec:n-collection-review}
Up to this point in this article we have been using exclusively the general definitions of $\ca T_{\leq n}$-collections (resp. operads) as cartesian natural transformations (resp. cartesian monad morphisms) into $\ca T_{\leq n}$. The advantage of this is that all the combinatorial aspects of pasting diagrams are neatly packaged into a single very well-behaved object -- the monad $\ca T_{\leq n}$ on $\ca G^n(\Set)$ for strict $n$-categories. We now, for both the convenience of the reader and to set our notation and terminology, recall some of these combinatorial aspects.

We denote also by $\tn{Tr}_{\leq n}$ the $n$-globular set $\ca T_{\leq n}(1)$. This is the free strict $n$-category on the terminal $n$-globular set, the set of $k$-cells of which are denoted as $\tn{Tr}_k$, whose elements can be regarded as $k$-stage trees in the sense of \cite{Batanin-MonGlobCats}, a $k$-stage tree $p$ being a sequence
\[ \xygraph{!{0;(1.5,0):} {p^{(k)}}="p1" [r] {p^{(k-1)}}="p2" [r] {...}="p3" [r] {p^{(1)}}="p4"  [r] {p^{(0)}}="p5"
"p1":"p2"^-{\xi_k}:"p3"^-{}:"p4"^-{}:"p5"^-{\xi_1}} \]
in $\Delta_+$ such that $p^{(0)}=1$. An element $x \in p^{(r)}$ is a \emph{node} of $p$ whose \emph{height} is $r$, the element $y=\xi_r...\xi_{s+1}(x)$ is $x$'s \emph{ancestor} of height $s$, $x$ is said to be a \emph{descendant} of $y$, and a node $x$ is said to be a \emph{leaf} when it doesn't have any descendants. The leaves of a given tree $p$ have a linear order determined by the order on the $p^{(r)}$'s and the $\xi_r$'s. For $k > 0$ the source and target functions $s,t:\tn{Tr}_{k+1} \to \tn{Tr}_k$ are equal, they correspond to forgetting the nodes at height $(k+1)$, that is to say ``truncation to height $k$'', and so we write $s=\tn{tr}=t$. For $k < n$ any $k$-stage tree $p$ can be regarded as a $(k{+}1)$-stage tree $zp$ such that $\tn{tr}(zp) = p$ and $(zp)^{(k+1)}=0$, and one has an inclusion $z : \tn{Tr}_{k} \hookrightarrow \tn{Tr}_{k{+}1}$. The $k$-stage tree with exactly one leaf at height $k$ is denoted as $U_k$, and so the $k$-stage trees with exactly one leaf are those of the 
form $z^sU_{k{-}s}$. Such trees are said to be \emph{linear}.

Trees in the sense just described define the $n$-globular sets to be regarded as globular pasting schemes. In \cite{Batanin-MonGlobCats} the notation $p^*$ was used for the globular set determined by the tree $p$. From a general perspective the assignation of a globular set from a tree is the object map of a functor $\tn{el}(\tn{Tr}_{\leq n}) \to \ca G^n(\Set)$ ($\tn{el}(\tn{Tr}_{\leq n})$ is the category of elements of the $n$-globular set $\tn{Tr}_{\leq n}$) which comes from the fact that $\ca T_{\leq n}$ is a local right adjoint monad on a presheaf category -- so the trees/globular pasting diagrams are the canonical arities for $\ca T_{\leq n}$, see \cite{Weber-Fam2fun} \cite{BergMellWeber-MonadsArities}. For an explicit description of this operation, see \cite{Batanin-MonGlobCats} or \cite{Weber-Generic}. We shall omit the $(-)^*$ and just refer to the $n$-globular set $p$. In particular the $n$-globular sets $p$ and $zp$ coincide.

The category $\Coll {\ca T_{\leq n}}$ of $\ca T_{\leq n}$-collections admits two useful reformulations because of the equivalences of categories
\[ \Coll {\ca T_{\leq n}} \catequiv \ca G^n(\Set)/\tn{Tr}_{\leq n} \catequiv [\tn{el}(\tn{Tr}_{\leq n})^{\op},\Set] \]
the first of which is given by evaluating at $1$, and the second is the standard description of a slice of a presheaf category as a presheaf category. Given $p \in \tn{Tr}_k$ where $k \leq n$, the set $A_p$ is just the fibre over $p \in \tn{Tr}_k$ of $(\alpha_1)_k : A(1)_k \to \tn{Tr}_k$. When $A$ underlies an operad, $a \in A_p$ is an operation which in any $A$-algebra composes such a pasting diagram to a $k$-cell. By the Yoneda lemma a $k$-stage tree $p$ may also be regarded as a morphism $p:k \to \tn{Tr}_{\leq n}$ of $n$-globular sets, and these are exactly the representables when regarding $\Coll{\ca T_{\leq n}}$ as a presheaf category. For $k > 0$ one has cosource and cotarget morphisms $\sigma,\tau:(k{-}1) \to k$ of representable globular sets, giving cosource and cotarget maps $\sigma_p,\tau_p:\tn{tr}(p) \to p$ in $\Coll{\ca T_{\leq n}}$, since in $\ca G^n(\Set)/\tn{Tr}_{\leq n}$ one has $\tn{tr}(p)=p\sigma=p\tau$.

To say that a $\ca T_{\leq n}$-collection is over $\Set$ is to say that for the corresponding $\alpha:A \to \tn{Tr}_{\leq n}$, the fibre over $U_0$ is singleton, that is that $A_{U_0}$ is singleton. The category $\NColl{\ca T_{\leq n}}$ is also a presheaf category
\[ \NColl{\ca T_{\leq n}} \catequiv [\ca N_{\leq n}^{\op},\Set] \]
where $\ca N_{\leq n}$ is the full subcategory of $\tn{el}(\tn{Tr}_{\leq n})$ consisting of those trees $p \neq U_0$. Right Kan extension along the inclusion $i:\ca N_{\leq n} \hookrightarrow \tn{el}(\tn{Tr}_{\leq n})$ corresponds to the inclusion $\NColl{\ca T_{\leq n}} \hookrightarrow \Coll{\ca T_{\leq n}}$, and restriction along $i$ applied to $A$ amounts to replacing $A_{U_0}$ with a singleton. So the trees $p \neq U_0$ are also representables in $\NColl{\ca T_{\leq n}}$, but viewed in $\ca G^n(\Set)/\tn{Tr}_{\leq n}$ this will be different from $p:k \to \tn{Tr}_{\leq n}$ in that the two 0-cells of the $n$-globular set $k$ will have been identified. Despite this difference we shall, when there is no risk of confusion, denote by $p$ this representable object of $\NColl{\ca T_{\leq n}}$. Thus for a $\ca T_{\leq n}$-operad $A$ over $\Set$, one may may regard an element of $a \in A_p$, that is to say an operation $a$ of $A$ of arity $p$, as a morphism $a:p \to A$ in $\NColl{\ca T_{\leq n}}$ by the Yoneda 
lemma. Similarly, for $k$-stage trees $p$ where $k>1$, one has $\sigma_p,\tau_p:\tn{tr}(p) \to p$ in $\NColl{\ca T_{\leq n}}$.

\subsection{$\ca T_{\leq n}$-operads with chosen contractions}
\label{ssec:operads-with-chosen-contractions}
The (Leinster-)contractibility of a $\ca T_{\leq n}$-operad is a condition on the underlying collection. A $\ca T_{\leq n}$-collection is contractible iff the corresponding morphism $\alpha:A \to \tn{Tr}_{\leq n}$ of $n$-globular sets satisfies the right lifting property (right lifting property) with respect to the set
\[ \ca I_{\leq n} = \{\partial(k) \xrightarrow{\phi_k} k \,\, : \,\, 0 \leq k \leq n\} \cup \{\partial(n{+}1) \xrightarrow{\phi_n'} n\}. \]
To equip the $\ca T_{\leq n}$-collection with \emph{chosen contractions} is by definition to give choices of liftings that witness $\alpha$'s right lifting property.

As such, possessing chosen contractions is exactly the structure of an $R$-algebra on $\alpha$, where $(L,R)$ is the algebraic weak factorisation system on $\ca G^n(\Set)$ cofibrantly generated by $\ca I_{\leq n}$. As observed in \cite{Garner-LeinsterDefinition}, possessing chosen contractions may be similarly identified from a cofibrantly generated algebraic weak factorisation system on $\NColl{\ca T_{\leq n}}$. One has
\[ \xygraph{!{0;(3,0):(0,1)::} {\NColl{\ca T_{\leq n}}}="l" [r] {\Coll{\ca T_{\leq n}}}="m" [r] {\ca G^n(\Set),}="r"
"l":@<-1ex>"m"_-{\tn{ran}_i}|-{}="bot":"r"^-{\tn{dom}} "l":@<1ex>@{<-}"m"^-{\tn{res}_i}|-{}="top" "top":@{}"bot"|-{\perp}} \]
one defines $\ca I_{\leq n}'$ to be the set of morphisms $\phi$ of $\Coll{\ca T_{\leq n}}$ such that $\tn{dom}(\phi) \in \ca I_{\leq n}$, and then $\ca I_{\leq n}''$ is the set of morphisms $\tn{res}_i(\phi)$ of $\NColl{\ca T_{\leq n}}$ such that $\phi \in \ca I_{\leq n}'$. Denote by $(L',R')$ and $(L'',R'')$ the algebraic weak factorisation systems cofibrantly generated by $\ca I_{\leq n}$ and $\ca I_{\leq n}''$ respectively. Let $B \in \NColl{\ca T_{\leq n}}$ such that $\tn{ran}_i(B)$ corresponds to $\alpha$. Then by the definition of $\ca I_{\leq n}'$, an $R$-algebra structure on $\alpha$ amounts to an $R'$-algebra structure on the unique map $\tn{ran}_i(B) \to 1$, which in turn amounts to an $R''$-algebra structure on $B \to 1$ by the definition of $\ca I_{\leq n}''$.

Moreover note that the boundaries $\partial k$ of representables are clearly finite colimits of representables and so are finitely presentable objects of $\ca G^n(\Set)$. Thus $(L,R)$ is a finitary algebraic weak factorisation system. The objects of a slice $\ca E/X$ whose domains are finitely presentable are finitely presentable in $\ca E/X$, thus the domains and codomains of $\phi \in \ca I_{\leq n}'$ are finitely presentable, whence $(L',R')$ is also finitary. The functor $\tn{ran}_i$ is clearly finitary -- this is most easily seen by regarding it as the inclusion, thus its left adjoint preserves finitely presentable objects, whence $(L'',R'')$ is also a finitary algebraic weak factorisation system. We denote by $\tn{Cont}_{\leq n}$ the monad $R''_1$, and we have proved
\begin{proposition}\label{prop:finitary-monad-for-chosen-contractions}
The monad $\tn{Cont}_{\leq n}$ on $\NColl{\ca T_{\leq n}}$ just described is finitary and its algebras are $\ca T_{\leq n}$-collections over $\Set$ equipped with chosen contractions.
\end{proposition}
\noindent In view of section(\ref{ssec:ncatopreview-overview}), this concludes our recollection of the definition of weak $n$-category.
\begin{remark}\label{rem:Bat-defn}
We have in this work used the definition of Leinster \cite{Leinster-HDA-book} for ease of exposition, and because its compatibility with the technology of algebraic weak factorisation systems has been well discussed in the literature \cite{Garner-LeinsterDefinition}. However one can, in a similar way, recast the original definition of \cite{Batanin-MonGlobCats} in similar terms, though working with strictly rather than weakly initial $\ca T_{\leq n}$-operads of the appropriate type. In this variant instead of $\NColl{\ca T_{\leq n}}$ one works with pointed collections, that is, with the coslice $1_{\ca G^n(\Set)}/\NColl{\ca T_{\leq n}}$. On this coslice there are three finitary monads whose algebras are $\ca T_{\leq n}$-operads, pointed $\ca T_{\leq n}$-collections equipped with a system of compositions, and pointed $\ca T_{\leq n}$-collections that are contractible in the sense of \cite{Batanin-MonGlobCats}, with the last of these obtained via a finitary algebraic weak factorisation system. The algebras of the coproduct of these monads are, in the language 
of \cite{Batanin-MonGlobCats}, exactly ``contractible normalised $n$-operads equipped with a system of compositions'', and the algebras of the initial such are weak $n$-categories in the sense of \cite{Batanin-MonGlobCats}.
\end{remark}

\subsection{Further remarks on generating cofibrations}
\label{ssec:overview-weak-units}
We analyse the sets $\ca I_{\leq n}$, $\ca I_{\leq n}'$ and $\ca I_{\leq n}''$ of generating cofibrations a little more, and develop some further notation and terminology for the sake of section(\ref{sec:strictly-unital-weak-n-cats}). The role of the ``extra'' arrow $\phi_n'$ in the definition of $\ca I_{\leq n}$ is that it ensures that for the resulting algebraic weak factorisation system $(L,R)$, the structure of $R$-algebra involves unique lifting against the map $\phi_n: \partial(n) \to n$. This is part of a general phenomenon first observed by Bousfield in section(4) of \cite{Bousfield-FactSystems}. We are grateful to Richard Garner for pointing us to this reference.
\begin{lemma}\label{lem:weak factorisation system->sfs}
(Bousfield \cite{Bousfield-FactSystems})
Let $\ca E$ be a category with pushouts, $\phi:S \to D$ in $\ca E$, and form $\phi'$ as follows:
\[ \xygraph{!{0;(2,0):(0,.5)::} {S}="tl" [r] {D}="tr" [d] {S'}="br" [l] {D}="bl" "tl":"tr"^-{\phi}:"br"^-{\tau}:@{<-}"bl"^-{\sigma}:@{<-}"tl"^-{\phi}:@{}"br"|-{\tn{po}}
"br" [dr] {D}="bbrr" "bl":@/_{1pc}/"bbrr"_{1_D} "tr":@/^{1pc}/"bbrr"^{1_D} "br":@{.>}"bbrr"|-{\phi'}} \]
Then $f:X \to Y$ satisfies the right lifting property with respect to $\phi$ and $\phi'$ iff it satisfies the unique right lifting property with respect to $\phi$.
\end{lemma}
Given a cocomplete category $\ca E$ satisfying either $(*)$ or $(\dagger)$ of \cite{Garner-SmallObjectArgument} section(4){\footnotemark{\footnotetext{These are general conditions under which Garner's small object argument works, and in particular are satisfied when $\ca E$ is locally finitely presentable.}}}, and that $J:\ca I \to \ca E^{[1]}$ with $\ca I$ small, and denote by $(L,R)$ the algebraic weak factorisation system cofibrantly generated by $J$. Suppose that a subset $\ca S$ of the objects of $\ca I$ are given, and regard $\ca S$ as a discrete subcategory of $\ca I$. Define $\ca I_{\ca S} = \ca I \coprod \ca S$ and $J_{\ca S}:{\ca I}_{\ca S} \to \ca E^{[1]}$ for the functor which agrees with $J$ on $\ca I$, and for $s \in \ca S$ is given by $J_{\ca S}(s) = J(s)'$ defined as in lemma(\ref{lem:weak factorisation system->sfs}). Denote by $(L_{\ca S},R_{\ca S})$ the algebraic weak factorisation system cofibrantly generated by $J_{\ca S}$. 
\begin{corollary}\label{cor:semi-strict-cofibgen-algebraic weak factorisation system}
For $\ca E$, $J:\ca I \to \ca E^{[1]}$ and $\ca S$ as just defined, an $R_{\ca S}$-algebra structure on $f:X \to Y$ in $\ca E$ is exactly an $R$-algebra structure together with the property that $f$ has the unique right lifting property with respect to the morphisms of $\ca S$.
\end{corollary}
\begin{remark}\label{rem:discrete-strict-cofibrations}
In corollary(\ref{cor:semi-strict-cofibgen-algebraic weak factorisation system}) $\ca S$ was just a subset of $\ca I$ regarded as a discrete subcategory. There is an evident way to extend $(-)'$ of lemma(\ref{lem:weak factorisation system->sfs}) to morphisms of generating cofibrations, and then one could take $\ca S$ to be a full subcategory of $\ca I$. However this gives nothing different, because the uniqueness of the lifting conditions would imply that the compatibilities expressed by these extra morphisms of generating cofibrations would be automatically satisfied by any $\ca R_{\ca S}$-algebra. In fact, if one wished to be completely minimalistic, one could also remove any morphisms of $\ca I \subseteq \ca I_{\ca S}$ whose domains are in $\ca S$, because for the same reason the resulting cofibrantly generated algebraic weak factorisation system would be the same.
\end{remark}
Returning to $\ca I_{\leq n}$, the above discussion shows that instead of defining $\ca I_{\leq n}$ one can instead define
\[ \tilde{\ca I}_{\leq n} = \{\partial(k) \xrightarrow{\phi_k} k \,\, : \,\, 0 \leq k \leq n\} \]
and distinguish $\ca S = \{\phi_n\}$ so that $(\tilde{\ca I}_{\leq n})_{\ca S} = \ca I_{\leq n}$. Thus we introduce
\begin{terminology}
Let $J:\ca I \to \ca E^{[1]}$ and $\ca S$ be as in corollary(\ref{cor:semi-strict-cofibgen-algebraic weak factorisation system}). Then the elements of $\ca S$ are said to be strict, and the morphisms $J(s)$ for $s \in \ca S$ are said to be generating strict cofibrations. The algebraic weak factorisation system cofibrantly generated by $(\ca S,J)$ is defined to be the algebraic weak factorisation system cofibrantly generated by $J_{\ca S}$.
\end{terminology}
\noindent So as an alternative to defining $\ca I_{\leq n}$, we define instead $\tilde{\ca I}_{\leq n}$ in which $\phi_n$ is distinguished as being strict, and consider the algebraic weak factorisation system cofibrantly generated therefrom.

Similarly instead of considering $\ca I_{\leq n}'$, we define instead the set
\[ \tilde{\ca I}_{\leq n}' = \{\partial(p) \xrightarrow{\phi_p} p \,\, : \,\, p \in \tn{Tr}_k, \, 0 \leq k \leq n\} \]
of morphisms of $\Coll{\ca T_{\leq n}}$, where $p$ is regarded as a morphism $p:k \to \tn{Tr}_{\leq n}$ of $n$-globular sets by Yoneda, and thus as a collection, in fact the collections of this form are exactly the representables -- remembering that $\Coll{\ca T_{\leq n}}$ is a presheaf category. The $\ca T_{\leq n}$-collection $\partial(p)$ is by definition the composite
\[ \partial(k) \xrightarrow{\phi_k} k \xrightarrow{p} \tn{Tr}_{\leq n} \]
and so $\phi_k$ underlies the morphism $\phi_p:\partial(p) \to p$ of collections. Then we distinguish the $\phi_p$ for $p \in \tn{Tr}_n$ as being strict. In the same way, as an alternative to $\ca I_{\leq n}''$ we define
\[ \tilde{\ca I}_{\leq n}'' = \{\partial(p) \xrightarrow{\phi_p} p \,\, : \,\, p \in \tn{Tr}_k, \, 0 < k \leq n\} \]
of morphisms of $\NColl{\ca T_{\leq n}}$, again distinguishing $\phi_p$ for $p \in \tn{Tr}_n$ as strict. As before the $p$ here are representables in the presheaf category $\NColl{\ca T_{\leq n}}$, noting that $\tn{res}_i$ sends $U_0$ to the initial object and all the other representables to the representables of the same name in $\NColl{\ca T_{\leq n}}$, and so we define $\phi_p$ in $\NColl{\ca T_{\leq n}}$ for $p \neq U_0$ to be $\tn{res}_i(\phi_p)$. There is no harm in excluding the case $\tn{res}_i(\phi_{U_0})$ since this is the identity on the initial object, against which every map satisfies the unique right lifting property, and so its inclusion or exclusion has no effect on the resulting algebraic weak factorisation system.

Finally we note that $\tilde{\ca I}_{\leq n}$, $\tilde{\ca I}_{\leq n}'$ and $\tilde{\ca I}_{\leq n}''$ admit alternative inductive descriptions, intrinsic to $\ca G^n(\Set)$, $\Coll{\ca T_{\leq n}}$ and $\NColl{\ca T_{\leq n}}$, an analogue of which will be important for section(\ref{ssec:define-strictly-unital-weak-n-cat}). For $\tilde{\ca I}_{\leq n}$ the initial step is that $\partial(0)$ is initial and $\phi_0$ is the unique map $\partial(0) \to 0$. For $0 \leq k < n$, $\partial (k+1)$ and $\phi_{k+1}$ are defined by the pushout
\[ \xygraph{!{0;(2,0):(0,.5)::} {\partial k}="tl" [r] {k}="tr" [d] {\partial(k+1)}="br" [l] {k}="bl"
"tl":"tr"^-{\phi_k}:"br"^-{}:@{<-}"bl"^-{}:@{<-}"tl"^-{\phi_k}:@{}"br"|-{\tn{po}}
"br" [dr] {k+1}="bbrr" "bl":@/_{1pc}/"bbrr"_{\sigma} "tr":@/^{1pc}/"bbrr"^{\tau} "br":@{.>}"bbrr"|-{\phi_{k+1}}} \]
where $\sigma$ and $\tau$ are the cosource and cotarget maps from $\mathbb{G}_{\leq n}$ regarded as morphisms between representable $n$-globular sets. For $\tilde{\ca I}_{\leq n}'$ one defines $\phi_p:\partial(p) \to p$ for $p \in \tn{Tr}_k$ (regarded as a representable) by induction on $k$. The base case is that $\partial(U_0)$ is initial and $\phi_{U_0}$ is determined uniquely and the inductive step is indicated in
\[ \xygraph{!{0;(2,0):(0,.5)::} {\partial\tn{tr}(p)}="tl" [r] {\tn{tr}(p)}="tr" [d] {\partial(p)}="br" [l] {\tn{tr}(p)}="bl"
"tl":"tr"^-{\phi_{\tn{tr}(p)}}:"br"^-{\tau_p'}:@{<-}"bl"^-{\sigma_p'}:@{<-}"tl"^-{\phi_{\tn{tr}(p)}}:@{}"br"|-{\tn{po}}
"br" [dr] {p}="bbrr" "bl":@/_{1pc}/"bbrr"_{\sigma_p} "tr":@/^{1pc}/"bbrr"^{\tau_p} "br":@{.>}"bbrr"|-{\phi_p}} \]
where the cosource and target maps $\sigma_p$ and $\tau_p$ are obtained in the evident way from those for $\ca G^n(\Set)$ (ie using $\Coll{\ca T_{\leq n}} \catequiv \ca G^n(\Set)/\tn{Tr}_{\leq n}$ and Yoneda). The inductive definition of $\tilde{\ca I}_{\leq n}''$ within $\NColl{\ca T_{\leq n}}$ is described identically, except that the base case is that $\partial(p)$ is initial for all $p \in \tn{Tr}_1$.

\section{Weak $n$-categories with strict units}
\label{sec:strictly-unital-weak-n-cats}

\subsection{Overview}
\label{ssec:overview-weak-units}
It is natural to ask whether the lifting theorem gives us an inductive formulation of the notion weak $n$-category, as recalled above in section(\ref{sec:weak-n-cat-review}), via iterated enrichment. That is, one can ask whether there is a functor operad $\ca L_{\leq n}$ on the category $\ca G^n(\Set)^{\ca K_{\leq n}}$ of $\ca K_{\leq n}$-algebras, such that
\begin{equation}\label{eq:iterative-formula}
\Enrich {\ca L_{\leq n}} \iso \ca G^{n{+}1}(\Set)^{\ca K_{\leq {n{+}1}}}
\end{equation}
over $\ca G^{n{+}1}(\Set)$.  Note that given a $\ca T_{\leq n{+}1}$-operad $B$ over $\Set$, we can consider the unary part of its corresponding $\ca T^{\times}_{\leq n}$-multitensor, which is a $\ca T_{\leq n}$-operad, and denote this by $h(B)$. This is the object map of a functor
\[ h : \NOp{\ca T_{\leq n{+}1}} \longrightarrow \Op{\ca T_{\leq n}}.  \]
The operad $h(B)$ describes the structure that the homs of a $B$-algebra have. So what the lifting theorem does give us, when applied to $\ca K_{\leq n{+}1}$, is a functor operad $\ca L_{\leq n}$ on $\ca G^n(\Set)^{h(\ca K_{\leq n{+}1})}$ satisfying equation(\ref{eq:iterative-formula}). Thus our desired inductive formulation of the notion weak $n$-category would follow if
\begin{equation}\label{eq:structure-on-homs}
h(\ca K_{\leq n{+}1}) \iso \ca K_{\leq n}
\end{equation}
which from the point of view of algebras says that the structure that the homs of a weak $(n+1)$-category have is \emph{exactly} that of a weak $n$-category. In this section we shall see that the presence of weak units causes (\ref{eq:structure-on-homs}) to be false. However we shall give a notion of ``weak $n$-category with strict units'' by producing the operads $\ca K_{\leq n}^{\tn{(su)}}$ which describe such structures, which do verify the appropriate analogue of (\ref{eq:structure-on-homs}). Thus the lifting theorem produces the functor operads giving a formulation of weak $n$-categories with strict units via iterated enrichment.

\subsection{Weak units}
\label{ssec:weak units}
Equation(\ref{eq:structure-on-homs}) isn't true even in the case $n = 1$. Let $\ca B$ be a bicategory. Then by definition any hom $\ca B(x,y)$ of $\ca B$ is a category, whose objects are arrows $x \to y$, an arrow $f \to g$ of $\ca B(x,y)$ is a 2-cell $f \implies g$ in $\ca B$, and composition is given by vertical composition of 2-cells. However, this is not the only structure that $\ca B(x,y)$ has in general. For instance one can consider composition by the identities $1_x$ and $1_y$ giving endofunctors
\[ \begin{array}{lccr} {(-) \comp 1_x : \ca B(x,y) \to \ca B(x,y)} &&& {1_y \comp (-) : \ca B(x,y) \to \ca B(x,y),} \end{array} \]
and more complicated processes involving just composing with identities such as
\[ 1_y \comp ((1_y \comp 1_y) \comp (((-) \comp 1_x) \comp 1_x)) : \ca B(x,y) \to \ca B(x,y). \]
Thus one has a non-trivial monoid $M$ whose elements are such formal expressions and multiplication is given by substitution. In addition to all this the coherence isomorphisms of $\ca B$ give isomorphisms between all these endofunctors, and from the coherence theorem for bicategories, any diagram of such isomorphisms commutes.

Let us write $\tn{ch}:\Set \to \Cat$ for the right adjoint to the functor $\tn{ob} : \Cat \to \Set$ which sends a category to its set of objects. Given a set $X$, $\tn{ch}(X)$ is known as the indiscrete category on $X$, it has object set $X$ and exactly one arrow between any 2 objects. From the previous paragraph and the coherence theorem for bicategories, the hom of a bicategory is a category equipped with a strict action by the strict monoidal category $\tn{ch}(M)$. Thus in the case $n=1$, $\ca K_{\leq 2}$ is an operad whose algebras are bicategories, $\ca K_{\leq 1}$-algebras are categories, but $h(\ca K_{\leq 2})$-algebras are categories strictly acted on by $\tn{ch}(M)$. In particular note that the free $h(\ca K_{\leq 2})$-algebra on a graph $Z$ has underlying object set given by $Z_0 \times M$, and so $h(\ca K_{\leq 2})$ is not even over $\Set$.

\subsection{Reduced $\ca T_{\leq n}$-operads}
\label{ssec:reduced-operads}
Let $\alpha:A \to \ca T_{\leq n}$ be a $\ca T_{\leq n}$-collection. When $A$ underlies a $\ca T_{\leq n}$-operad and $p$ is a linear tree, then the elements of $A_p$ are unital operations of some type. For example an operation of arity $zU_0$ distinguishes a one cell $u_x:x \to x$ in any $A$-algebra $X$ for all $x \in X_0$. As another example, an operation of arity $z^2U_0$ will have source and target operations of arity $zU_0$, and so such an operation will distinguish one cells $u_x$ and $v_x:x \to x$ and 2-cells $u_x \to v_x$ in any $A$-algebra  $X$ for all $x \in X_0$. In general an operation of arity $z^rU_s$ distinguishes for each $s$-cell $x$ in an $A$-algebra $X$, a single $(r{+}s)$-cell $u_x$ whose $s$-dimensional source and target is $x$, and so the sources and targets of $u_x$ in dimensions between $s$ and $r{+}s$ will also be part of the data determined such an operation, and in any of these intermediate dimensions, these need not be the same cell.
\begin{definition}\label{def:reduced-operad}
Let $\alpha : A \to \ca T_{\leq n}$ be a $\ca T_{\leq n}$-collection. Then $\alpha$ is \emph{reduced} when for all linear trees $p \in \tn{Tr}_k$ where $k \leq n$, the set $A_p$ is singleton. A reduced $\ca T_{\leq n}$-operad is one whose underlying $\ca T_{\leq n}$-collection is reduced in this sense. We denote by $\RdColl{\ca T_{\leq n}}$ the full subcategory of $\Coll{\ca T_{\leq n}}$ consisting of the reduced $\ca T_{\leq n}$-collections, and by $\RdOp{\ca T_{\leq n}}$ the full subcategory of $\Op{\ca T_{\leq n}}$ consisting of the reduced $\ca T_{\leq n}$-operads.
\end{definition}
\noindent Thus for a reduced $\ca T_{\leq n}$-operad one has a unique operation of each ``unit type'' (i.e., of each arity $p$ where $p$ is a linear tree). Examples include the terminal operad $\ca T_{\leq n}$, the $\ca T_{\leq 2}$-operad for sesqui-categories, and the $\ca T_{\leq 3}$-operad for Gray categories, and one would expect that the (yet to be defined) operads describing the higher analogues of Gray categories to be of this form. To say that a $\ca T_{\leq n}$-collection $A$ is over $\Set$ is to say that $A_{U_0}$ is singleton, and so reduced $\ca T_{\leq n}$-collections and reduced $\ca T_{\leq n}$-operads are in particular over $\Set$.

The category $\RdColl{\ca T_{\leq n}}$ is a presheaf category, one has
\[ \RdColl{\ca T_{\leq n}} \catequiv [\ca R_{\leq n}^{\op},\Set] \]
where $\ca R_{\leq n}$ is the full subcategory of $\tn{el}(\tn{Tr}_{\leq n})$ consisting of those trees $p$ which are not linear. The inclusion $I:\RdColl{\ca T_{\leq n}} \hookrightarrow \Coll{\ca T_{\leq n}}$ can also be seen as the process of right Kan extending along the inclusion $i:\ca R_{\leq n} \hookrightarrow \tn{el}(\tn{Tr}_{\leq n})$, so its left adjoint $L$ is given by restriction along $i$. One has
\begin{equation}\label{diag:rdops-rdcolls}
\begin{aligned}
\xygraph{!{0;(2.5,0):(0,.5)::} {\RdOp{\ca T_{\leq n}}}="tl" [r] {\Op{\ca T_{\leq n}}}="tr" [d] {\Coll{\ca T_{\leq n}}}="br" [l] {\RdColl{\ca T_{\leq n}}}="bl" "tl":"tr"^-{J}:"br"^-{U}:@{<-}"bl"^-{I}:@{<-}"tl"^-{U^{\tn{(rd)}}} "tl":@{}"br"|-{\tn{pb}}}
\end{aligned}
\end{equation}
a pullback square in $\CAT$, where $I$ and $J$ are the inclusions, and $U$ and $U^{\tn{(rd)}}$ are the forgetful functors. Each of the categories in (\ref{diag:rdops-rdcolls}) are the categories of models of finite limit sketches and so are locally finitely presentable, and the square itself is induced by a square of morphisms of limit sketches. Thus each of the functors in (\ref{diag:rdops-rdcolls}) are finitary and monadic.

While not essential for this article, it is of interest that one has explicit descriptions of the left adjoints to each of the functors participating in (\ref{diag:rdops-rdcolls}). We already gave the explicit description of the left adjoint $L$ of $I$, and the construction of free operads as described in Appendix D of \cite{Leinster-HDA-book} gives an explicit description of the left adjoint $F$ of $U$. Given an explicit description $L'$ of $J$, $F^{\tn{(rd)}} = L'FI$ is then an explicit description of the left adjoint to $U^{\tn{(rd)}}$ as witnessed by the natural isomorphisms
\[ \begin{array}{rcl} {\RdOp{\ca T_{\leq n}}(L'FIA,B)} & {\iso} & {\Coll{\ca T_{\leq n}}(IA,UJB) = \Coll{\ca T_{\leq n}}(IA,IU^{\tn{(rd)}}B)} \\ 
{} & {\iso} & {\RdColl{\ca T_{\leq n}}(A,U^{\tn{(rd)}}B).} \end{array} \]

An explicit description of $L'$ follows from the work of Harvey Wolff \cite{Wolff-FreeMonads}. Theorem(2.1) of \cite{Wolff-FreeMonads} deals with the following situation. One has a bipullback square in $\CAT$ as on the left in
\begin{equation}\label{diag:Wolff-thm}
\begin{aligned}
\xygraph{{\xybox{\xygraph{{\ca C}="tl" [r] {\ca A}="tr" [d] {\ca B}="br" [l] {\ca D}="bl" "tl":"tr"^-{J}:"br"^-{U}:@{<-}"bl"^-{I}:@{<-}"tl"^-{V}:@{}"br"|-{\tn{bipb}}}}} [r(3)]
{\xybox{\xygraph{{FU}="tl" [r] {1_{\ca A}}="tr" [d] {P}="br" [l] {FILU}="bl" "tl":"tr"^-{\varepsilon}:"br"^-{\pi}:@{<-}"bl"^-{}:@{<-}"tl"^-{F\eta L}:@{}"br"|-{\tn{po}}}}}}
\end{aligned}
\end{equation}
where $\ca A$ has pushouts, $I$ and $J$ are fully faithful and regarded as inclusions, $I$ has left adjoint $L$ with unit $\eta$, and $U$ has left adjoint $F$ with counit $\varepsilon$. The conclusion of Wolff's theorem is that the pointed endofunctor $(P,\pi)$, constructed via the pushout as on the right in the previous display, is well-pointed (i.e., $P\pi = \pi P$) and its algebras are exactly the objects of $\ca C$ (i.e., those in the image of $J$). Thus in particular when $\ca A$ is locally finitely presentable and $U$ and $I$ are finitary, the left adjoint $L'$ of $J$ is constructed as the free monad on $(P,\pi)$ and is finitary. Then by virtue of well-pointedness, $L'$ is constructed as the colimit of the sequence
\[ \xygraph{{1}="p1" [r] {P}="p2" [r] {P^2}="p3" [r] {...}="p4" "p1":"p2"^-{\pi}:"p3"^-{P\pi}:"p4"^-{P^2\pi}} \]
-- see \cite{Kelly-Transfinite}.

One can apply Wolff's theorem in the case (\ref{diag:rdops-rdcolls}) since $U$ being monadic is an isofibration, and so the pullback of (\ref{diag:rdops-rdcolls}) is also a bipullback. However we shall obtain a more basic description of $L'$. Denote by $\tn{LinTr}_{\leq n} \hookrightarrow \tn{Tr}_{\leq n}$ subobject of $\tn{Tr}_{\leq n}$ consisting of the linear trees. Clearly these trees are closed under substitution of trees, and so $\tn{LinTr}_{\leq n}$ underlies a $\ca T_{\leq n}$-operad which we denote as $\tn{RGr}_{\leq n}$. To give an $n$-globular set $X$ an $\tn{RGr}_{\leq n}$-algebra structure, is to give common sections $z_k:X_k \to X_{k{+}1}$ for its source and target maps $s,t:X_{k{+}1} \to X_k$ for $0 \leq k < n$. That is, $\tn{RGr}_{\leq n}$-algebras are exactly reflexive $n$-globular sets. Both as a $\ca T_{\leq n}$-collection and as a $\ca T_{\leq n}$-operad, $\tn{RGr}_{\leq n}$ is \emph{subterminal}, that is to say, the unique map $\rho:\tn{RGr}_{\leq n} \to \ca T_{\leq n}$ into the terminal 
collection/operad is a monomorphism. For any $\ca T_{\leq n}$-operad $A$, the operations of the $\ca T_{\leq n}$-operad $A \times \tn{RGr}_{\leq n}$ are exactly the operations of $A$ whose arities are linear trees, and the projection $p_A : A \times \tn{RGr}_{\leq n} \to \tn{RGr}_{\leq n}$ assigns these operations to their arities. Thus to say that $A$ is reduced is to say that $p_A$ is an isomorphism, which amounts to saying that one has a bipullback square as on the left in
\[ \xygraph{{\xybox{\xygraph{!{0;(3,0):(0,.333)::} {\RdOp{\ca T_{\leq n}}}="tl" [r] {\Op{\ca T_{\leq n}}}="tr" [d] {\Op{\ca T_{\leq n}}/\tn{RGr}_{\leq n}}="br" [l] {1}="bl"
"tl":"tr"^-{J}:"br"^-{\Delta_{\tn{RGr}_{\leq n}}}:@{<-}"bl"^-{\ulcorner 1_{\tn{RGr}_{\leq n}}\urcorner}:@{<-}"tl"^-{}:@{}"br"|-{\tn{bipb}}}}} [r(4.75)u(.08)]
{\xybox{\xygraph{!{0;(2,0):(0,.5)::} {A \times \tn{RGr}_{\leq n}}="tl" [r] {A}="tr" [d] {PA}="br" [l] {\tn{RGr}_{\leq n}}="bl" "tl":"tr"^-{q_A}:"br"^-{\pi_A}:@{<-}"bl"^-{}:@{<-}"tl"^-{p_A}:@{}"br"|-{\tn{po}}}}}} \]
where $\Delta_{\tn{RGr}_{\leq n}}$ is given by pulling back along $\rho$. Applying Wolff's theorem to this situation, the left adjoint $L'$ to $J$ is the free monad on the pointed endofunctor $(P,\pi)$ constructed as on the right in the previous display, where $p_A$ and $q_A$ are the product projections. We are indebted to Steve Lack for pointing out that our explicit construction of $L'$ is an instance of Wolff's theorem, thereby enabling us to describe this construction here more efficiently.

\subsection{Defining weak $n$-categories with strict units}
\label{ssec:define-strictly-unital-weak-n-cat}
While a reduced $\ca T_{\leq n}$-operad has unique operations of each unit type, there is no guarantee that these really act as units with respect to the other operations. In this section we formalise this aspect, and then by mimicking the approach to defining weak $n$-categories in section(\ref{ssec:ncatopreview-overview}), we shall define weak $n$-categories with strict units. The main idea is to work with a richer notion of collection, namely one that remembers the unit operations of a $\ca T_{\leq n}$-operad, and then define a stronger notion of contractibility for such collections, by isolating an appropriate algebraic weak factorisation system.

Let us first recall, for the convenience of the reader and to fix notation, substitution for $\ca T_{\leq n}$-operads. For a $\ca T_{\leq n}$-operad $\alpha:A \to \ca T_{\leq n}$, the substitution is encoded by the multiplication of the monad $A$. The object $A \comp A$ of $\Coll{\ca T_{\leq n}}$ is the composite
\[ A^2 \xrightarrow{\alpha^2} \ca T_{\leq n}^2 \xrightarrow{\mu} \ca T_{\leq n} \]
where $\mu$ here is the multiplication of the monad $\ca T_{\leq n}$, and then $A$'s multiplication underlies a morphism $A \comp A \to A$ in $\Coll{\ca T_{\leq n}}$. To book-keep what is going on at the level of the sets of operations of $A$, it is convenient to use the notion of morphism of $k$-stage trees.

Given $k$-stage trees $p$ and $q$, a morphism $f:p \to q$ consists of functions $f^{(i)}$ as shown
\[ \xygraph{!{0;(1.5,0):(0,.6667)::} {p^{(k)}}="p1" [r] {p^{(k-1)}}="p2" [r] {...}="p3" [r] {p^{(1)}}="p4"  [r] {p^{(0)}}="p5"
"p1" [d] {q^{(k)}}="q1" [r] {q^{(k-1)}}="q2" [r] {...}="q3" [r] {q^{(1)}}="q4"  [r] {q^{(0)}}="q5"
"p1":"p2"^-{\partial_k}:"p3"^-{}:"p4"^-{}:"p5"^-{\partial_1} "q1":"q2"_-{\partial_k}:"q3"_-{}:"q4"_-{}:"q5"_-{\partial_1}
"p1":"q1"_{f^{(k)}} "p2":"q2"_{f^{(k{-}1)}} "p4":"q4"^{f^{(1)}} "p5":"q5"^{f^{(0)}}} \]
making the above diagram commute serially (in $\Set$), and such that for $0 < i \leq k$, $f^{(i)}$ is order preserving on the fibres of $\partial_i$. In this way one has a category $\Omega_k$, whose set of objects is $\tn{Tr}_k$, and an $n$-globular category $\Omega_{\leq n}$ whose underlying $n$-globular set (i.e its object of objects when one views globular categories as category objects in $\ca G^n(\Set)$) is $\tn{Tr}_{\leq n}$. Recall from \cite{BataninStreet-OmegaUProp} that $\Omega_{\leq n}$ has a universal property -- it is the free categorical $\ca T_{\leq n}$-algebra containing an internal $\ca T_{\leq n}$-algebra.

For any given node $x$ of $q$ of height $r$ one can consider a morphism $\tilde{x}:z^{k-r}U_r \to q$ which picks out $x$ and all its ancestors. Pulling the components of $\tilde{x}$ back along those of $f$ produces a tree $f^{-1}(x) \in \tn{Tr}_r$, and one has an inclusion $z^{k-r}f^{-1}(x) \hookrightarrow p$ of $k$-stage trees. When $x$ is the $i$-th leaf of $q$ we denote $f^{-1}(x)$ as $p_i$, and all the $p_i$'s together are called the \emph{fibres} of $f$. Now the trees $p_i$ provide a labelling of $q$ in the sense that truncating $p_i$ and $p_{i{+}1}$ to the height of the highest common ancestor of the $i$-th and $(i{+}1)$-th leaves of $q$ gives the same tree for $1 \leq i < a$, where $a$ is the number of leaves of $q$. From the explicit description of $\ca T_{\leq n}$ then, $f:p \to q$ may be identified with an element of $\ca T^2_{\leq n}(1)_k$, the effect of $\ca T_{\leq n}(!)$ (where $!:\ca T_{\leq n}(1) \to 1$ is the unique map) on this element is $q$, and of $\mu_1$ on this element is $p$. In other 
words the result of substituting the $p_i$ into $q$ is $p$. Now for a $\ca T_{\leq n}$-operad $A$, the data defining the operad substitution of $A$ amounts to sets and functions
\[ \sigma_f : A_q \times \left(A_{p_1} \times_{A_{p_1'}} ... \times_{A_{p_{a{-}1}'}} A_{p_a}\right) \longrightarrow A_p \]
for all morphisms $f:p \to q$ of $k$-stage trees, where $p_i'$ is the result of truncating $p_i$ to the height of the highest common ancestor of the $i$-th and $(i{+}1)$-th leaves of $q$.
\begin{definition}\label{def:tree-inclusions}
A morphism $f:p \to q$ of $k$-stage trees is an \emph{inclusion} when its components -- the $f^{(i)}$ -- are injective functions. We denote by $\Omega_{\leq n}^{\tn{(incl)}}$ the sub-$n$-globular category consisting of all the trees and the inclusions between them.
\end{definition}
It is also worth noting that the categories $\Omega_{\leq n}$ and $\Omega_{\leq n}^{\tn{(incl)}}$ sit naturally within Joyal's category $\Theta_n$. In the language of \cite{Berger-CellularNerve}, $\Omega_{\leq n}$ (resp. $\Omega_{\leq n}^{\tn{(incl)}}$) is dual to the subcategory of $\Theta_n$ consisting of all the objects and the covers (resp. degeneracies) between them.

Clearly a morphism of $k$-stage trees is an inclusion iff its fibres are linear. Thus $\sigma_f$ for $f$ an inclusion encodes the substitution of operations of unit type into operations of (an arbitrary) arity $q$. In this case when $A$ is reduced, the sets $A_{p_i}$ and $A_{p_i'}$ are all singletons, and so $\sigma_f$ may be regarded simply as a function $\sigma_f:A_q \to A_p$. It is these substitutions that we wish to include in our richer notion of collection. Denote by $\Psi_{\leq n}$ the full subcategory of $\tn{el}(\Omega_{\leq n}^{\tn{(incl)}})$ consisting of the non-linear trees.
\begin{definition}\label{def:pointed-reduced-collection}
A \emph{pointed reduced $\ca T_{\leq n}$-collection} is a presheaf $A : \Psi_{\leq n}^{\op} \to \Set$. We denote by $\PtRdColl{\ca T_{\leq n}} = [\Psi_{\leq n}^{\op},\Set]$
the category of pointed reduced $\ca T_{\leq n}$-collections.
\end{definition}
Thus a pointed reduced collection $A$ contains a set of operations of arity $p$ for any tree $p$ that is not linear, just as with a reduced collection, and in addition one has extra functions between those sets that codify the substitution of unique operations of unit type into arbitrary operations. Recall that $\RdColl{\ca T_{\leq n}} \catequiv [\ca R_{\leq n}^{\op},\Set]$ where $\ca R_{\leq n}$ is the full subcategory of $\tn{el}(\tn{Tr}_{\leq n})$ consisting of the trees which are not linear, and so there is by definition an identity on objects functor $i:\ca R_{\leq n} \to \Psi_{\leq n}$, with the morphisms of $\Psi_{\leq n}$ not in the image of $i$ coming from inclusions of trees. Denoting by $j:\ca R_{\leq n} \to \tn{el}(\tn{Tr}_{\leq n})$ the fully faithful inclusion, we have adjunctions
\[ \xygraph{!{0;(3.5,0):(0,1)::} {\PtRdColl{\ca T_{\leq n}}}="l" [r] {\RdColl{\ca T_{\leq n}}}="m" [r] {\Coll{\ca T_{\leq n}}}="r"
"l":@<-1.2ex>"m"_-{\tn{res}_i}|-{}="b1":@<-1.2ex>"r"_-{\tn{ran}_j}|-{}="b2" "l":@{<-}@<1.2ex>"m"^-{\tn{lan}_i}|-{}="t1":@{<-}@<1.2ex>"r"^-{\tn{res}_j}|-{}="t2"
"t1":@{}"b1"|{\perp} "t2":@{}"b2"|{\perp}} \]
given by left Kan extension and restriction along $i$, and right Kan extension and restriction along $j$. Note in particular the effect of $\tn{res}_j$ and $\tn{lan}_i$ on representables. One has $\tn{res}_j(p) = 0$ if $p$ is a linear tree, otherwise one can take $\tn{res}_j(p) = p$, and one has $\tn{lan}_i(p) = p$ for all representables. Moreover from the description of boundaries of representables as pushouts in $\Coll{\ca T_{\leq n}}$, it is clear that $\tn{res}_j(\partial p) = 0$ if $\tn{tr}(p)$ is linear.

Let us now describe the appropriate cofibrantly generated algebraic weak factorisation system on $\PtRdColl{\ca T_{\leq n}}$. We shall define the functor
\[ J : \ca V_{\leq n} \longrightarrow (\PtRdColl{\ca T_{\leq n}})^{[1]} \]
where $\ca V_{\leq n}$ is the subcategory of $\Psi_{\leq n}$ consisting of all the non-linear trees, and just the inclusions between them. Note that the category $\tn{el}(\Omega^{\tn{(incl)}}_{\leq n})$, as the domain of a split fibration into $\mathbb{G}_{\leq n}$, comes with a strict factorisation system in which the left class are the vertical arrows, and the right class are the chosen cartesian arrows (i.e., chosen by the given split cleavage). Thus $\Psi_{\leq n}$ inherits such a factorisation system, $\ca R_{\leq n}$ is the subcategory of morphisms in the right class -- consisting of the cosource and cotarget morphisms, and $\ca V_{\leq n}$ is the subcategory of morphisms in the left class -- the vertical morphisms, which in this case are the inclusions of trees. We shall define the object and arrow maps of $J$ in an inductive manner.

Recall that sources and targets for $\tn{Tr}_{\leq n}$ agree and are given by truncation. For $0 \leq r \leq k \leq n$ let us denote by $\tn{tr}^{k-r}:\tn{Tr}_k \to \tn{Tr}_r$ the function which sends $p \in \tn{Tr}_k$ to its height-$r$ truncation, that is to say, to the tree $p^{(r)} \to ... \to p^{(0)}$. When $k-r=1$ this was denoted above simply as $\tn{tr}$. The \emph{non-linear height} $\tn{ht}(p)$ of $p \in \tn{Tr}_k$ is greatest $h \in \N$ such that $\tn{tr}^{k-h}(p)$ is a linear tree. Since $\tn{tr}^k(p)=U_0$ is linear, $\tn{ht}(p)$ is well-defined for all $p \in \tn{Tr}_k$, and by definition a tree $p$ is linear iff $\tn{ht}(p)=0$.

For a non-linear tree $p$ we write $J(p) = \phi_p:\partial(p) \to p$ in $\PtRdColl{\ca T_{\leq n}}$, and these maps are defined by induction on $\tn{ht}(p)$. For the base case $\tn{ht}(p)=1$, we take $\partial(p)$ to be initial giving us a unique map $\phi_p:\partial(p) \to p$. Similarly as in section(\ref{ssec:operads-with-chosen-contractions}) $\partial(p)$ and $\phi_p$, for $p$ of non-linear height $> 1$, are defined by
\[ \xygraph{!{0;(2,0):(0,.5)::} {\partial\tn{tr}(p)}="tl" [r] {\tn{tr}(p)}="tr" [d] {\partial(p)}="br" [l] {\tn{tr}(p)}="bl" "tl":"tr"^-{\phi_{\tn{tr}(p)}}:"br"^-{\tau_p'}:@{<-}"bl"^-{\sigma_p'}:@{<-}"tl"^-{\phi_{\tn{tr}(p)}}:@{}"br"|-{\tn{po}}
"br" [dr] {p}="bbrr" "bl":@/_{1pc}/"bbrr"_{\sigma_p} "tr":@/^{1pc}/"bbrr"^{\tau_p} "br":@{.>}"bbrr"|-{\phi_p}} \]
though now in the category $\PtRdColl{\ca T_{\leq n}}$. Note that the image of the object map of $J$ just described agrees with applying $\tn{lan}_i\tn{res}_j$ to $\ca I_{\leq n}'$, except for the presence of copies of $1_0$ (coming from $\phi_p$ and $\phi_p'$ in $\Coll{\ca T_{\leq n}}$ with $p$ linear).

Given an inclusion $f:p \to q$ of non-linear $k$-stage trees we shall define $\partial(f)$ as on the left in
\[ \begin{array}{lccr} {\partial(f):\partial(p) \to \partial(q)} &&& {J(f)=(\partial(f),f):\phi_p \to \phi_q} \end{array} \]
such that $\phi_q\partial(f)=f\phi_p$, which thus gives a morphism $J(f)$ in $(\PtRdColl{\ca T_{\leq n}})^{[1]}$ as on the right in the previous display. We define $\partial(f)$ by induction on the non-linear height of $p$. For the base case where $\tn{ht}(p)=1$, $\partial(p)=0$ and so $\partial(f)$ is uniquely defined and one has $\phi_q\partial(f)=f\phi_p$ by initiality. When $\tn{ht}(p)>1$, $\partial(f)$ is constructed in
\[ \xygraph{!{0;(1.5,0):(0,.6667)::} {\partial\tn{tr}(p)}="tl" [r] {\tn{tr}(p)}="tr" [d] {\partial(p)}="br" [l] {\tn{tr}(p)}="bl"
"tl":"tr"^-{\phi_{\tn{tr}(p)}}:"br"^-{\tau_p'}:@{<-}"bl"^-{\sigma_p'}:@{<-}"tl"^-{\phi_{\tn{tr}(p)}}:@{}"br"|-{\tn{po}}
"tl" [l(.75)u] {\partial\tn{tr}(q)}="tl2" "tr" [r(.75)u] {\tn{tr}(q)}="tr2" "br" [r(.75)d] {\partial(q)}="br2" "bl" [l(.75)d] {\tn{tr}(q)}="bl2"
"tl2":"tr2"^-{\phi_{\tn{tr}(q)}}:"br2"^-{\tau_q'}:@{<-}"bl2"^-{\sigma_q'}:@{<-}"tl2"^-{\phi_{\tn{tr}(q)}}
"tl":"tl2"_-{\partial(\tn{tr}(f))} "tr":"tr2"^-{\tn{tr}(f)} "bl":"bl2"^-{\tn{tr}(f)} "br":@{.>}"br2"_-{\partial(f)}} \]
in which the existence of $\partial(\tn{tr}(f))$ and $\phi_{\tn{tr}(q)}\partial(\tn{tr}(f)) = \tn{tr}(f)\phi_{\tn{tr}(p)}$ follow by induction, and so $\partial(f)$ is defined uniquely so that $\partial(f)\sigma_p'= \sigma_q'\tn{tr}(f)$ and $\partial(f)\tau_p'= \tau_q'\tn{tr}(f)$. In order that this induction goes through we must verify $\phi_q\partial(f)=f\phi_p$. Since $(\sigma_p',\tau_p')$ are jointly epic, it suffices to show $\phi_q\partial(f)\sigma_p'=f\phi_p\sigma_p'$ and $\phi_q\partial(f)\tau_p'=f\phi_p\tau_p'$. One has
\[ \phi_q\partial(f)\sigma_p' = \phi_q\sigma_q'\tn{tr}(f) = \sigma_q\tn{tr}(f) = f\sigma_p = f\phi_p\sigma_p', \]
$\phi_q\partial(f)\tau_p'=f\phi_p\tau_p'$ follows similarly. The functoriality of $J$ is clear by construction. Finally we define the set $\ca S$ of strict objects of $\ca V_{\leq n}$ to be the non-linear $n$-stage trees.
\begin{proposition}\label{prop:algebraic weak factorisation system-for-PtRedColl}
The algebraic weak factorisation system $(L,R)$ on $\PtRdColl{\ca T_{\leq n}}$ cofibrantly generated by $(\ca S,J)$ defined above is finitary.
\end{proposition}
\begin{proof}
By lemma(\ref{lem:fin-cofgen-algebraic weak factorisation system}) it suffices to show that the domains and codomains of arrows of $\PtRdColl{\ca T_{\leq n}}$ in the image of $J_{\ca S}$ are all finitely presentable. The domains are non-linear trees viewed as objects of $\PtRdColl{\ca T_{\leq n}}$, are representables, and so are finitely presentable objects. The codomains are clearly finite colimits of representables, and so are also finitely presentable.
\end{proof}
\begin{definition}\label{def:choice-unital-contractions}
A \emph{choice of unital contractions} for a pointed reduced $\ca T_{\leq n}$-collection $A$ is an $R_1$-algebra structure on $A$ for the algebraic weak factorisation system $(L,R)$ of proposition(\ref{prop:algebraic weak factorisation system-for-PtRedColl}). A choice of unital contractions for a reduced $\ca T_{\leq n}$-operad is one for its underlying pointed reduced collection.
\end{definition}
Explicitly a choice of unital contractions for $A$ amounts to
\begin{itemize}
\item choice of contractions in the usual sense for the underlying collection, and
\item a compatibility condition on these choices with respect to the process of substituting the unique unit operations to general operations.
\end{itemize}
For $p$ an object of $\Psi_{\leq n}$, a morphism $p \to A$ in $\PtRdColl{\ca T_{\leq n}}$ is by Yoneda an operation of $A$ of arity $p$, and a morphism $\partial(p) \to A$ is, by the definition of $\partial(p)$, a pair of operations $(a,b)$ of (the underlying $\ca T_{\leq n}$-collection of) $A$ both of arity $\tn{tr}(p)$, whose sources and targets agree (called ``parallel operations'' in \cite{Leinster-HDA-book}). So within $\PtRdColl{\ca T_{\leq n}}$ a choice of contractions $\gamma$ amounts to choices of fillers as in
\begin{equation}\label{diag:choice-of-contractions}
\begin{aligned}
\xygraph{!{0;(2,0):(0,.5)::} {\partial(p)}="tl" [r] {A}="tr" [d] {1}="br" [l] {p}="bl" "tl":"tr"^-{(a,b)}:"br"^-{}:@{<-}"bl"^-{}:@{<-}"tl"^-{\phi_p} "bl":@{.>}"tr"|-{\gamma(p,a,b)}} 
\end{aligned}
\end{equation}
for all $0 < k \leq n$, non-linear $k$-stage trees $p$, and parallel operations $(a,b)$ as shown.

We remind the reader that as pointed out in \cite{Leinster-HDA-book}, the choice $\gamma$ encodes two aspects of higher category theory in one go -- compositions and coherences. For example taking $p$ to be the $1$-stage tree on the left
\[ \xygraph{{\xybox{\xygraph{!{0;(.5,0):(0,1)::} {\bullet}="l" [dr] {\bullet}="b" [ur] {\bullet}="r" "l"-"b"-"r"}}} [r(3)]
{z\left(\xybox{\xygraph{*=(1,1)\xybox{\xygraph{!{0;(.5,0):(0,1)::} {\bullet}="l" [dr] {\bullet}="b" [ur] {\bullet}="r" [l] {\bullet}="m" "l"-"b"(-"r",-"m")}}}}\right)}} \]
then an operation of arity $p$ is a binary composition of $1$-cells, and in (\ref{diag:choice-of-contractions}) $(a,b)$ are determined uniquely (i.e., $a$ and $b$ are both the operadic unit in dimension $0$), and so the contraction in this instance gives a distinguished binary composition of $1$-cells as part of the resulting algebraic structure. On the other hand taking $p$ to be the $2$-stage tree on the right in the previous display, then $(a,b)$ of (\ref{diag:choice-of-contractions}) will be a pair of operations that compose chains of $1$-cells of length $3$, and then the filler will give distinguished coherence $2$-cells between these. So in general $\gamma$ \emph{distinguishes} a higher categorical operation of each arity.

The compatibility for a choice of unital contractions for $A$ says that if one substitutes some unique unit operations into a given distinguished operation -- either one that gives some kind of composition, or one that gives some kind of coherence -- then the result is another distinguished operation. For example part of the data of a choice of contractions for a reduced operad $A$ will be a choice of operation that composes chains of one-cells of length $n$ for all $n \geq 2$. Compatibility says in particular that, if one starts with a chain
\[ w \xrightarrow{f} x \xrightarrow{g} y \xrightarrow{h} z \]
of one-cells in a given $A$-algebra, and then regards it as a longer chain by adding in some identity arrows
\[ w \xrightarrow{f} x \xrightarrow{1_x} x \xrightarrow{1_x} x \xrightarrow{g} y \xrightarrow{h} z \xrightarrow{1_z} z  \]
which are part of the structure by virtue of the reducedness of $A$ (i.e., what we denote here as identity one-cells are provided by the unique operation of $A$ of arity $z(U_0)$), then applying the chosen method of composition to this longer chain should agree with applying the chosen method of composition to the original chain.

As the foregoing discussion indicates, the existence of a choice of unital contractions for a given reduced $\ca T_{\leq n}$-operad $A$ amounts to saying that one can distinguish operations of $A$ of each arity, in such a way that \emph{this choice} is compatible with the process of substituting in unit operations. While this may seem like a strong condition at first glance, we shall see now that such choices can be exhibited in low-dimensional cases of interest, and moreover one has universal such $\ca T_{\leq n}$-operads that we will use to define weak $n$-categories with strict units. In section(\ref{sec:strictly-unital-weak-n-cats-via-iterated-enrichment}) we shall see that this choice of unital contractions is exactly what we need to get an alternative iterative formulation.
\begin{example}\label{ex:chosen-unital-contractions-normal-bicat}
Let $A$ be the $\ca T_{\leq 2}$-operad whose algebras are bicategories in the sense of B\'{e}nabou \cite{Benabou-Bicategories} for which the unit coherence $2$-cells are identities, and we write ``$\comp$'' for the given binary horizontal composition of $1$-cells. The strictness of the units ensures that $A$ is reduced. The coherence theorem for bicategories ensures that one has unique choices of contractions $\gamma(p,a,b)$ when $p$ is a $2$-stage tree. So to give a chosen contraction for $A$ we must choose a distinguished composite of chains of $1$-cells of length $n$, for $n \geq 2$. The set of all such compositions may be identified with the set of binary brackettings of $n$ symbols, so for each $n$ we must choose one of these. There are two obvious choices, to bracket completely to the left or completely to the right, as illustrated in
\[ \begin{array}{lcccccr} {x_0 \xrightarrow{f_1} x_1 \xrightarrow{f_2} x_2 \xrightarrow{f_3} x_3 \xrightarrow{f_4} x_4}
&&& {((f_4 \comp f_3) \comp f_2) \comp f_1} &&& {f_4 \comp (f_3 \comp (f_2 \comp f_1))} \end{array} \]
in the case $n=4$. Clearly each of these schemes is compatible with units, for example
\[ (((((1_{x_4} \comp f_4) \comp f_3) \comp 1_{x_2}) \comp 1_{x_2}) \comp f_2) \comp f_1 = ((f_4 \comp f_3) \comp f_2) \comp f_1, \]
and so in two ways, one can exhibit a choice of unital contractions for $A$.
\end{example}
\begin{example}\label{ex:chosen-unital-contractions-GrayCat}
Let $A$ be the $\ca T_{\leq 3}$-operad for Gray categories. The underlying $1$-operad is terminal, $A$'s contractibility ensures that one has unique choices of contractions $\gamma(p,a,b)$ when $p$ is a $2$-stage tree, and so the only choices that aren't forced concern what one takes to be the distinguished composites of non-degenerate 2-dimensional pasting diagrams in a Gray category. The set of all ways to compose a given such pasting diagram, say
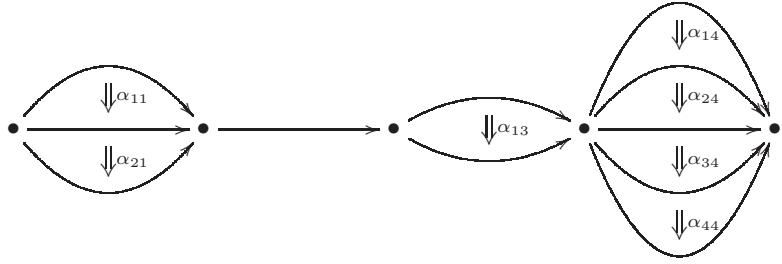
\[ \xygraph{!{0;(2,0):(0,.5)::} {\bullet}="p1" [r] {\bullet}="p2" [r] {\bullet}="p3" [r] {\bullet}="p4" [r] {\bullet}="p5" 
"p1":@/^{2pc}/"p2"|-{}="p11" "p1":"p2"|-{}="p21" "p1":@/_{2pc}/"p2"|-{}="p31" "p2":"p3" "p3":@/^{1pc}/"p4"|-{}="p13" "p3":@/_{1pc}/"p4"|-{}="p23"
"p4":@/^{4pc}/"p5"|-{}="p14" "p4":@/^{2pc}/"p5"|-{}="p24" "p4":"p5"|-{}="p34" "p4":@/_{2pc}/"p5"|-{}="p44" "p4":@/_{4pc}/"p5"|-{}="p54"
"p11":@{}"p21"|(.5){}="m11" "m11" [u(.15)] :@{=>}[d(.3)]^{\alpha_{11}} "p21":@{}"p31"|(.5){}="m21" "m21" [u(.15)] :@{=>}[d(.3)]^{\alpha_{21}}
"p13":@{}"p23"|(.5){}="m13" "m13" [u(.15)] :@{=>}[d(.3)]^{\alpha_{13}}
"p14":@{}"p24"|(.5){}="m14" "m14" [u(.15)] :@{=>}[d(.3)]^{\alpha_{14}} "p24":@{}"p34"|(.5){}="m24" "m24" [u(.15)] :@{=>}[d(.3)]^{\alpha_{24}}
"p34":@{}"p44"|(.5){}="m34" "m34" [u(.15)] :@{=>}[d(.3)]^{\alpha_{34}} "p44":@{}"p54"|(.5){}="m44" "m44" [u(.15)] :@{=>}[d(.3)]^{\alpha_{44}}} \]
to a $2$-cell, that one has in a Gray category, can be identified with linear orders $\leq$ on the set of $\alpha_{ij}$ such that $i_1 \leq i_2$ implies $\alpha_{i_1j} \leq \alpha_{i_2j}$, or in other words, with $(2,1,4)$-shuffles. There are two evident unit-compatible choices. In terms of linear orders one can define $\alpha_{i_1j_1} < \alpha_{i_2j_2}$ iff $j_1 < j_2$ or ($j_1=j_2$ and $i_1<i_2$); or alternatively $\alpha_{i_1j_1} < \alpha_{i_2j_2}$ iff $j_1 > j_2$ or ($j_1=j_2$ and $i_1<i_2$), which for the above example correspond to the linear orders
\[ \begin{array}{lccr} {(\alpha_{11},\alpha_{21},\alpha_{13},\alpha_{14},\alpha_{24},\alpha_{34},\alpha_{44})} &&&
{(\alpha_{14},\alpha_{24},\alpha_{34},\alpha_{44},\alpha_{13},\alpha_{11},\alpha_{21})} \end{array} \]
which correspond to ``moving'' down each column completely, and from left-to-right or right-to-left respectively.
\end{example}
Not every ``systematic looking'' choice of compositions need give a unit-compatible choice of contractions. For instance in this last example another choice of contractions could be to move from left to right and row by row, just as one reads text in European languages, so that in the case of the above pasting diagram one would choose the order
\[ (\alpha_{11},\alpha_{13},\alpha_{14},\alpha_{21},\alpha_{24},\alpha_{34},\alpha_{44}). \]
The choice of contractions so determined is not unit-compatible, because the distinguished composite for the pasting diagram on the left
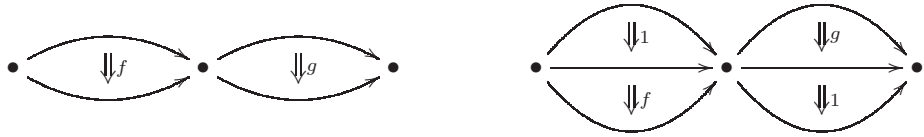
\[ \xygraph{!{0;(2,0):(0,.5)::} {\bullet}="l" [r] {\bullet}="m" [r] {\bullet}="r"
"l":@/^{1pc}/"m"|-{}="p11" "l":@/_{1pc}/"m"|-{}="p21" "m":@/^{1pc}/"r"|-{}="p12" "m":@/_{1pc}/"r"|-{}="p22"
"p11":@{}"p21"|(.5){}="m11" "m11" [u(.15)] :@{=>}[d(.3)]^{f} "p12":@{}"p22"|(.5){}="m12" "m12" [u(.15)] :@{=>}[d(.3)]^{g}
"r" [r(.75)] {\bullet}="l2" [r] {\bullet}="m2" [r] {\bullet}="r2"
"l2":@/^{2pc}/"m2"|-{}="p112" "l2":"m2"|-{}="p212" "l2":@/_{2pc}/"m2"|-{}="p312" "m2":@/^{2pc}/"r2"|-{}="p122" "m2":"r2"|-{}="p222" "m2":@/_{2pc}/"r2"|-{}="p322"
"p112":@{}"p212"|(.5){}="m112" "m112" [u(.15)] :@{=>}[d(.3)]^{1} "p212":@{}"p312"|(.5){}="m212" "m212" [u(.15)] :@{=>}[d(.3)]^{f}
"p122":@{}"p222"|(.5){}="m122" "m122" [u(.15)] :@{=>}[d(.3)]^{g} "p222":@{}"p322"|(.5){}="m222" "m222" [u(.15)] :@{=>}[d(.3)]^{1}} \]
corresponds to the order $(f,g)$, whereas substituting in identities to the distinguished composite as indicated on the right in the previous display, gives the operation that corresponds to the other order $(g,f)$.

We now proceed to give the definition of weak $n$-category with strict units, following the approach of section(\ref{sec:weak-n-cat-review}).
\begin{proposition}\label{prop:finitary-monad-for-RedOpChosenUnitalContractions}
There is a finitary monad on $\PtRdColl{\ca T_{\leq n}}$ whose algebras are reduced operads equipped with a choice of unital contractions.
\end{proposition}
\begin{proof}
It suffices to give two finitary monads on $\PtRdColl{\ca T_{\leq n}}$, one whose algebras are reduced collections equipped with chosen unital contractions, and the other whose algebras are reduced operads, because then the required monad will just be the coproduct of these. In the first case one has the monad $R_1$ for the finitary algebraic weak factorisation system $(L,R)$ of proposition(\ref{prop:algebraic weak factorisation system-for-PtRedColl}). In the second case one can use a general finite limit sketch argument to establish the existence of the required finitary monad. Alternatively for a construction, one has forgetful functors
\[ \xygraph{!{0;(3,0):(0,1)::} {\RdOp{\ca T_{\leq n}}}="l" [r] {\PtRdColl{\ca T_{\leq n}}}="m" [r] {\RdColl{\ca T_{\leq n}},}="r" "l":"m"^-{U}:"r"^-{\tn{res}_i}} \]
$\tn{res}_i$ is cocontinuous and monadic by a standard application of the Beck Theorem (using that $i$ is a bijective on objects functor), we noted the finitariness of $U^{\tn{(rd)}}=\tn{res}_iU$ and constructed explicitly its left adjoint in section(\ref{ssec:reduced-operads}), and so Dubuc's adjoint triangle theorem \cite{Dubuc-KanExtensions} gives an explicit construction of the left adjoint of $U$. Moreover $U$ is finitary since both $\tn{res}_i$ and $U^{\tn{(rd)}}$ create filtered colimits, and monadic by another standard application of the Beck theorem.
\end{proof}
Thus the category of reduced operads equipped with a choice of unital contractions is locally finitely presentable, and we denote its initial object as $\ca K_{\leq n}^{(su)}$.
\begin{definition}\label{def:strictly-unital-weak-n-cat}
A \emph{weak $n$-category with strict units} is a $\ca K_{\leq n}^{(su)}$-algebra.
\end{definition}
Since as we noted above chosen unital contractions are in particular chosen contractions, there is an operad morphism $\ca K_{\leq n} \to \ca K_{\leq n}^{(su)}$ by the universal property of $\ca K_{\leq n}$, and thus every weak $n$-category with strict units has an underlying weak $n$-category.

\section{Strictly unital weak $n$-categories via iterated enrichment}
\label{sec:strictly-unital-weak-n-cats-via-iterated-enrichment}
The bulk of this section will be occupied with the proof of
\begin{lemma}\label{lem:key-lemma-for-iterative-defn}
For all $n \in \N$ one has an isomorphism of $\ca T_{\leq n}$-operads
\[ h(\ca K_{\leq n{+}1}^{\tn{(su)}}) \iso \ca K_{\leq n}^{\tn{(su)}}. \]
\end{lemma}
\noindent As explained in section(\ref{ssec:overview-weak-units}), from lemma(\ref{lem:key-lemma-for-iterative-defn}) and the lifting theorem one obtains the following result.
\begin{theorem}\label{thm:iterative-defn}
For all $n \in \N$ there is a distributive functor operad $\ca L_{\leq n}$ on the category $\ca G^n(\Set)^{\ca K_{\leq n}^{\tn{(su)}}}$ of algebras of $\ca K_{\leq n}^{\tn{(su)}}$ and an isomorphism of categories
\begin{equation}\label{eq:successive-enrichment}
\Enrich{\ca L_{\leq n}} \iso \ca G^{n{+}1}(\Set)^{\ca K_{\leq n{+}1}^{\tn{(su)}}}
\end{equation}
over $\ca G^{n{+}1}(\Set)$.
\end{theorem}
\noindent Thus the lifting theorem has produced the tensor products -- the $\ca L_{\leq n}$'s -- enabling us to recover the notion of weak $n$-category with strict units of definition(\ref{def:strictly-unital-weak-n-cat}) inductively by successive enrichment -- as expressed by (\ref{eq:successive-enrichment}) at the level of objects.

In general we have functors
\[ \begin{array}{lccr} {h : \NOp{\ca T_{\leq n{+}1}} \longrightarrow \Op{\ca T_{\leq n}}} &&&
{r : \Op{\ca T_{\leq n}} \longrightarrow \NOp{\ca T_{\leq n{+}1}}.} \end{array} \]
Recall from section(\ref{ssec:overview-weak-units}) that for a $\ca T_{\leq n{+}1}$-operad $B$ over $\Set$, the operad $h(B)$ describes the structure that the homs of a $B$-algebra have, and is constructed formally by taking the unary part of the $\ca T_{\leq n}^{\times}$-multitensor associated to $B$ by \cite{Weber-MultitensorsMonadsEnrGraphs} theorem(6.1.1). On the other hand given a $\ca T_{\leq n}$-operad $A$, $r(A)$ is the $\ca T_{\leq n{+}1}$-operad $\Gamma(A^{\times})$, whose algebras are categories enriched in $\ca G^n(\Set)^A$ using the cartesian product by proposition(3.2.1) of \cite{Weber-MultitensorsMonadsEnrGraphs} and proposition(2.8) of \cite{BataninWeber-EnHop}. Since $(\Gamma(A^{\times}))_1 = A$ one has that $rh=1$.

We shall prove lemma(\ref{lem:key-lemma-for-iterative-defn}) by first proving
\begin{lemma}\label{lem:h-r-reduced-operads}
The functors $h$ and $r$ defined above restrict to an adjunction
\[ \xygraph{!{0;(3,0):(0,1)::} {\RdOp{\ca T_{\leq n{+}1}}}="l" [r] {\RdOp{\ca T_{\leq n}}}="r" "l":@<1.2ex>"r"^-{h}|-{}="top":@<1.2ex>"l"^-{r}|-{}="bot" "top":@{}"bot"|-{\perp}} \]
between categories of reduced operads.
\end{lemma}
\noindent and then denoting by $\CtRdOp{\ca T_{\leq n}}$ the category of reduced $\ca T_{\leq n}$-operads with chosen unital contractions, we shall prove
\begin{lemma}\label{lem:h-r-reduced-contractible-operads}
The adjunction of lemma(\ref{lem:h-r-reduced-operads}) lifts to an adjunction
\[ \xygraph{!{0;(3,0):(0,1)::} {\CtRdOp{\ca T_{\leq n{+}1}}}="l" [r] {\CtRdOp{\ca T_{\leq n}}}="r" "l":@<1.2ex>"r"^-{h}|-{}="top":@<1.2ex>"l"^-{r}|-{}="bot" "top":@{}"bot"|-{\perp}} \]
between categories of reduced operads with chosen unital contractions.
\end{lemma}
\noindent Since lemma(\ref{lem:key-lemma-for-iterative-defn}) is then just the statement that this lifted $h$ preserves the initial object, it follows immediately from lemma(\ref{lem:h-r-reduced-contractible-operads}).

We now describe the object maps of $h$ and $r$ in more elementary terms. Recall from \cite{Weber-MultitensorsMonadsEnrGraphs} section(3), that given a category $V$ with an initial object, a finite sequence $(X_1,...,X_m)$ of objects of $V$ can be regarded as a $V$-graph whose set of objects is $\{0,...,m\}$, $(X_1,...,X_m)(i,i{+}1)=X_i$ and the other homs are initial. This simple construction enables an inductive definition of ``$k$-dimensional globular pasting diagram'':
\begin{itemize}
\item base case $k=0$: the representable globular set $0$, with one vertex and no edges, is the unique $0$-dimensional pasting diagram.
\item inductive step: a $(k{+}1)$-dimensional pasting diagram $p$ is a sequence $(p_1,...,p_m)$ of $k$-dimensional pasting diagrams.
\end{itemize}
which makes the connection with trees very transparent. So a $(k{+}1)$-stage tree $p$ can be regarded as a sequence of maps in $\Delta_+$ as on the left
\[ \begin{array}{lccr} {p^{(k{+}1)} \to ... \to p^{(1)} \to p^{(0)}{=}1} &&& {\sum_{i=1}^m p^{(k)}_i \to ... \to \sum_{i=1}^m p^{(0)}_i \to 1} \end{array} \]
or as a sequence of $k$-stage trees $(p_1,...,p_m)$ as indicated on the right in the previous display where $p^{(1)}=m$. In particular, the tree $(p)$ is obtained from the tree $p$ by adding a new root and an edge from the new root to the old root. With these preliminaries in hand one has, by unpacking the definitions, the formulae
\[ \begin{array}{lccr} {h(B)_p = B_{(p)}} &&& {r(A)_p = \prod\limits_{i=1}^k A_{p_i}} \end{array} \]
for $p=(p_1,...,p_m) \in \tn{Tr}_{k+1}$, expressing the effect of $h$ and $r$ at the level of sets of operations.

For a given morphism of $(k{+}1)$-stage trees $f:q \to p$ with fibres $(r_1,...,r_s)$ we shall need, for the proof of lemma(\ref{lem:h-r-reduced-operads}) below, to unpack a description of corresponding substitution function
\[ \sigma^{rh(B)}_f : rh(B)_p \times \left(rh(B)_{r_1} \times_{rh(B)_{r'_1}} ... \times_{rh(B)_{r'_{s{-}1}}} rh(B)_{r_s}\right) \longrightarrow rh(B)_q \]
for $rh(B)$ in terms of substitution functions for $B$. To do so, we must describe the morphism $f:q \to p$ of $(k{+}1)$-trees in the same inductive terms as we did trees in the previous paragraph. So we let $q^{(1)}=l$ and denote by $(l_1,...,l_m)$ the cardinality of the fibres of $f^{(1)}:l \to m$ which is order preserving by definition. Now let us reindex the sequence of fibres $(r_1,...,r_s)$ of $f$ to keep track of over which $p_i$ a given fibre of $f$ lives, so we write
\[ (r_1,...,r_s) = (r_{11},...,r_{ms_m}) \]
that is to say, $r_{ij}$ for $1 \leq i \leq m$ and $1 \leq j \leq s_i$, is the $j$-th fibre of $f$ that lives over $p_i$. Moreover we shall describe each of the $r_{ij}$ as sequences of trees one dimension lower as in
\[ r_{ij} = (r_{ij1},...,r_{ijl_i}). \]
Thus the inductive description of the components of $f$ is as
\[ f^{(\delta{+}1)} = (f^{(\delta)}_{11},...,f^{(\delta)}_{1l_1}) + ... + (f^{(\delta)}_{m1},...,f^{(\delta)}_{ml_m}) \]
where $f_{i\alpha}:q_{i\alpha} \to p_i$ for $1 \leq i \leq m$ and $1 \leq \alpha \leq l_m$, is the morphism of $k$-stage trees whose fibres are $(r_{i1\alpha},...,r_{is_i\alpha})$. By this notation the number of leaves of $p$ is $s$, and the number of leaves of $p_i$ is $s_i$. Let us denote by $t_{ij\alpha}$ the number of leaves of $r_{ij\alpha}$. Thus $r_{ij}$ has $\sum_{\alpha} t_{ij\alpha}$ leaves, and $q$ has $\sum_{i,j,\alpha} t_{ij\alpha}$ leaves. With these details in hand, one may readily verify that $\sigma^{rh(B)}_f$ may be rewritten as a product, varying over $i$ and $\alpha$, of the functions
\[ \sigma^{B}_{f_{i\alpha}} : B_{(p_i)} \times \left(B_{(r_{i1\alpha})} \times_{B_{(r'_{il\alpha})}} ... \times_{B_{(r'_{is_{i{-}1}\alpha})}} B_{(r_{is_i\alpha})}\right) \longrightarrow  B_{(q_{i\alpha})}. \]
\begin{proof}
(\emph{of lemma(\ref{lem:h-r-reduced-operads})}): by the above explicit formulae, $h$ and $r$ send reduced operads to reduced operads. Given a reduced $\ca T_{\leq n{+}1}$-operad $B$, we shall describe a morphism as on the left in
\[ \begin{array}{lccr} {\nu_B : B \longrightarrow rh(B)} &&& {\nu_{B,p,i} : B_p \longrightarrow B_{(p_i)}} \end{array} \]
\[  \]
of $\ca T_{\leq n{+}1}$-operads, naturally in $B$, such that $h\nu_B = \id$ and $\nu_B r = \id$, so that $\nu$ is the unit of the desired adjunction in which the counit is the identity. To give the morphism of underlying $\ca T_{\leq n{+}1}$-collections, it suffices by the explicit descriptions of $r$ and $h$ spelled out above, to give functions $\nu_{B,p,i}$ as indicated on the right in the previous display, where $0 < k \leq n$, $p=(p_1,...,p_m) \in \tn{Tr}_{k{+}1}$ and $1 \leq i \leq m$, satisfying appropriate naturality conditions. Denoting by $\pi_i:(p_i) \to p$ the canonical inclusion of trees, we define $\nu_{B,p,i}$ from the base change function in the underlying pointed reduced collection, namely as $\nu_{B,p,i} = B(\pi_i)$. This clearly defines a morphism of $\ca T_{\leq n}$-collections.

Intuitively, $\nu_B$ is compatible with the operad structures since it was defined by substitution of unit operations. Formally, for any morphism $f:q \to p$ of trees whose fibres are are denoted as $(r_1,...,r_s)$, we must show that
\begin{equation}\label{diag:operadicness-nuB}
\begin{aligned}
\xygraph{!{0;(5,0):(0,.3)::} {B_p \times \left(B_{r_1} \times_{B_{r_1'}} ... \times_{B_{r_{s{-}1}'}} B_{r_s}\right)}="tl" [r] {B_q}="tr" [d]
{rh(B)_q}="br" [l] {rh(B)_p \times \left(rh(B)_{r_1} \times_{rh(B)_{r_1'}} ... \times_{rh(B)_{r_{s{-}1}'}} rh(B)_{r_s}\right)}="bl" "tl":"tr"^-{\sigma^{B}_f}:"br"^-{\nu_{B,q}}:@{<-}"bl"^-{\sigma^{rh(B)}_f}:@{<-}"tl"|-{\nu_{B,p} \times (\nu_{B,r_1} \times_{\nu_{B,r_1'}} ... \times_{\nu_{B,r_{s{-}1}'}} \nu_{B,r_s})}}
\end{aligned}
\end{equation}
commutes. It suffices to verify this commutativity after composition with each projection $\tn{proj}_{i\alpha}:rh(B)_q \to B_{(q_{i\alpha})}$. We do this now using the notation established above for $f:q \to p$. For
\[ (b,b_{11},...,b_{ms_m}) \in B_p \times \left(B_{r_1} \times_{B_{r_1'}} ... \times_{B_{r_{s{-}1}'}} B_{r_s}\right) \]
the effect of $\tn{proj}_{i\alpha}$ composed with the top path of (\ref{diag:operadicness-nuB}) on this element is
\begin{equation}\label{eq:top-path-opnuB}
\left(b[b_{11},...,b_{ms_m}]\right)\left[\underbrace{v,...,v,}_{\sum_X t_{i'j\alpha'}} \underbrace{u,...,u,}_{\sum_j t_{ij\alpha}} \underbrace{v,...,v}_{\sum_Y t_{i'j\alpha'}} \right]
\end{equation}
where
\[ \begin{array}{lll} X &=& {\left\{(i',\alpha') \,\, : \,\, i' < i \,\, \tn{or} \,\, (i'=i \,\, \tn{and} \,\, \alpha'<\alpha)\right\}} \\
Y &=& {\{(i',\alpha') \,\, : \,\, i' > i \,\, \tn{or} \,\, (i'=i \,\, \tn{and} \,\, \alpha'>\alpha)\}.} \end{array} \]
On the other hand the effect on $(b,b_{11},...,b_{ms_m})$ of $\tn{proj}_{i\alpha}$ composed with the bottom path of (\ref{diag:operadicness-nuB}) is
\begin{align}
\left(b[\underbrace{v,...,v,}_{\sum_{i'<i,j}s_{i'j}} \underbrace{u,...,u,}_{\sum_j s_{ij}} \underbrace{v,...,v}_{\sum_{i'>i,j}s_{i'j}}]\right)
& \left[b_{i1}[\underbrace{v,...,v,}_{\sum_{\alpha'<\alpha}t_{i1\alpha'}} \underbrace{u,...,u,}_{t_{i1\alpha}} \underbrace{v,...,v}_{\sum_{\alpha'>\alpha}t_{i1\alpha'}}], ... \right. \nonumber \\
& \left. \qquad ..., b_{is_i}[\underbrace{v,...,v,}_{\sum_{\alpha'<\alpha}t_{is_i\alpha'}} \underbrace{u,...,u,}_{t_{is_i\alpha}} \underbrace{v,...,v}_{\sum_{\alpha'>\alpha}t_{is_i\alpha'}}]\right] \label{eq:bot-path-opnuB}
\end{align}
Using the associativity of substitution, and that for all operations $x$ of $B$ one has $x[v,...,v]=v$ by reducedness, (\ref{eq:top-path-opnuB}) equals
\begin{align}
b\left[\underbrace{v,...,v,}_{\sum_{i'<i,j}s_{i'j}} \right. & b_{i1}[\underbrace{v,...,v,}_{\sum_{\alpha'<\alpha}t_{i1\alpha'}} \underbrace{u,...,u,}_{t_{i1\alpha}} \underbrace{v,...,v}_{\sum_{\alpha'>\alpha}t_{i1\alpha'}}],...\nonumber \\
& \left. ..., b_{is_i}[\underbrace{v,...,v,}_{\sum_{\alpha'<\alpha}t_{is_i\alpha'}} \underbrace{u,...,u,}_{t_{is_i\alpha}} \underbrace{v,...,v}_{\sum_{\alpha'>\alpha}t_{is_i\alpha'}}],\underbrace{v,...,v}_{\sum_{i'>i,j}s_{i'j}}\right]\label{eq:for-opnuB}
\end{align}
and by applying operad associativity and one of the operad unit laws to (\ref{eq:bot-path-opnuB}), one obtains (\ref{eq:bot-path-opnuB})$=$(\ref{eq:for-opnuB}).
\end{proof}
\begin{proof}
(\emph{of lemma(\ref{lem:h-r-reduced-contractible-operads})}): we must explain how the functors $r$ and $h$ of lemma(\ref{lem:h-r-reduced-operads}) lift to the level of reduced operads with chosen unital contractions, and then show that the operad map $\nu_B$ is compatible with the chosen unital contractions.

Recall that the data of a choice $\gamma$ of contractions for $B \in \RdOp{\ca T_{\leq n}}$ amounts to morphisms
\[ \gamma(p,a,b) : p \longrightarrow B \]
in $\PtRdColl{\ca T_{\leq n}}$, for all $0 \leq k \leq n$, non-linear $p \in \tn{Tr}_{k{+}1}$ and $a,b \in B_{\tn{tr}(p)}$ such that $s(a)=s(b)$ and $t(a)=t(b)$, and that this choice is unital when it satisfies
\[ \gamma(q,a\tn{tr}(f),b\tn{tr}(f)) = \gamma(p,a,b)f \]
for all $p,a,b$ as above and tree inclusions $f:q \to p$. Given $\gamma$ and recalling the explicit description of $h(B)$ given above, $\gamma'(p,a,b)=\gamma((p),a,b)$ defines a choice $\gamma'$ of contractions for $h(B)$. The unitality of $\gamma'$ with respect to an inclusion $f:q \to p$ follows from that of $\gamma$ with respect to the evident inclusion $(f):(q) \to (p)$. Clearly if $F:B \to B'$ preserves chosen contractions, then so does $h(F)$. Thus $h$ lifts to the level of operads with chosen unital contractions.

Suppose that a choice $\psi$ of unital contractions for $A \in \RdOp{\ca T_{\leq n}}$ is given. Recalling the explicit description of $r(A)$ given above,
\[ \psi'(p,(a_1,...,a_m),(b_1,...,b_m)) = (\psi(p_1,a_1,b_1),...,\psi(p_m,a_m,b_m)) \]
where $p=(p_1,...,p_m)$, defines a choice of contractions for $r(A)$. Note that the inductive description of a tree morphism $f:q \to p$ given above simplifies in the case where $f$ is an inclusion, because $f^{(1)}$ is injective and so $l_i=0$ or $1$, and so one may write $f$ as a composite
\[ (q_1,...,q_l) \xrightarrow{(f_1,...,f_l)} (p_{f^{(1)}(1)},...,p_{f^{(1)}(l)}) \xrightarrow{\tn{ss}_{f^{(1)}}} (p_1,...,p_m) \]
where $\tn{ss}_{f^{(1)}}$ is the subsequence inclusion corresponding to the injection $f^{(1)}$ interpretted as an inclusion of trees in the evident way, and the $f_i$ are themselves inclusions of trees. By definition $\psi'$ is compatible with such subsequence inclusions. Since $\psi$ is unital, $\psi'$ is compatible with tree inclusions of the form $(f_1,...,f_l)$ as above, and so $\psi'$ is itself unital. Clearly if $G:A \to A'$ preserves chosen contractions, then so does $r(G)$. Thus $r$ lifts to the level of operads with chosen unital contractions.

Let us now verify that $\nu_B$ is compatible with chosen contractions. So given $0 \leq k \leq n$, $p=(p_1,...,p_m) \in \tn{Tr}_{k{+}1}$ and $a,b \in B_{\tn{tr}(p)}$ such that $s(a)=s(b)$ and $t(a)=t(b)$, we must verify
\[ \nu_B(\gamma(p,a,b)) = \gamma''(p,\nu_B(a),\nu_B(b)) \]
and it suffices to verify this equation for all $1 \leq i \leq m$ after applying the projection $\tn{proj}_i:rh(B)_p \to B_{(p_i)}$. Recalling the inclusion $\pi:(p_i) \to p$ of trees, by the definition of $\nu_B$ we have
\[ \tn{proj}_i\nu_B(\gamma(p,a,b)) = \gamma(p,a,b)\pi_i \]
whereas
\[ \begin{array}{lllll} {\tn{proj}_i\gamma''(p,\nu_B(a),\nu_B(b))} &=& {\gamma'(p_i,a\tn{tr}(\pi_i),b\tn{tr}(\pi_i))} &=& {\gamma((p_i),a\tn{tr}(\pi_i),b\tn{tr}(\pi_i))} \\
&=& {\gamma(p,a,b)\pi_i} && \end{array} \]
in which the first two equalities follow from the definitions, and the last follows from the unitality of $\gamma$ with respect to $\pi_i$.
\end{proof}
\begin{remark}\label{rem:product-comparisons-reduced-context}
For a reduced $\ca T_{\leq n{+}1}$-operad $B$ we have the component $\nu_B:B \to rh(B)$ of the unit of the above adjunctions. If we denote by $E$ the $\ca T_{\leq n}^{\times}$-multitensor such that $\Gamma E = B$, then $E_1^{\times}$ is the $\ca T_{\leq n}^{\times}$-multitensor such that $\Gamma E_1^{\times} = rh(B)$, and then $\nu_B$ is the effect of $\Gamma$ on a $\ca T_{\leq n}^{\times}$-multitensor map $\kappa_E:E \to E_1^{\times}$, whose unary part is the identity. The natural transformation $\kappa_E$ can be regarded as the coherence data for a lax monoidal functor
\[ (1,\kappa_E) : (\ca G^n(\Set),E_1^{\times}) \longrightarrow (\ca G^n(\Set),E), \]
applying $C$ from proposition(\ref{prop:2-fun-lifting-thm}) to this gives
\[ \begin{array}{c} {(1,\kappa_E') : (\ca G^n(\Set)^{E_1},\prod) \longrightarrow (\ca G^n(\Set)^{E_1},E')} \end{array} \]
and the components of $\kappa_E'$ are of the form
\[ \begin{array}{c} {(\kappa_E')_{Z_1,...,Z_n} : \opEpr\limits_{i=1}^n Z_i \longrightarrow \prod\limits_{i=1}^n Z_i} \end{array} \]
defined for all finite sequences $(Z_1,...,Z_n)$ of $E_1$-algebras. In particular one has comparison maps between the functor operads $\ca L_{\leq n}$ of theorem(\ref{thm:iterative-defn}) and the cartesian product of weak $n$-categories with strict units.
\end{remark}

\section*{Acknowledgements}
\label{sec:Acknowledgements}
We would like to acknowledge Clemens Berger, Richard Garner, Andr\'{e} Joyal, Steve Lack, Joachim Kock, Jean-Louis Loday, Paul-Andr\'{e} Melli\`{e}s, Ross Street and Dima Tamarkin
for interesting discussions on the substance of this paper. In particular the third author would like to thank Andr\'{e} Joyal, who already in 2006 during some inspiring discussions indicated that a strictly unital version of operadic higher category theory ought to be better behaved, and gave some indication at that time already of how the theory of reduced operads should work. We would also like to acknowledge the Centre de Recerca Matem\`{a}tica in Barcelona for the generous hospitality and stimulating environment provided during the thematic year 2007-2008 on Homotopy Structures in Geometry and Algebra.

The first author would like to acknowledge the financial support on different stages of this project of the Scott Russell Johnson Memorial Foundation, the Australian Research Council grant No.DP0558372, and L'Universit\'{e} Paris 13. For the later stages of the project, the first and third authors acknowledge the financial support of the Australian Research Council grant grant No.DP1095346. The second author would like to acknowledge the support of the ANR grant no. ANR-07-BLAN-0142. The third author would like to acknowledge the laboratory PPS (Preuves Programmes Syst\`{e}mes) in Paris, the Max Planck Institute in Bonn, the IHES and the Macquarie University Mathematics Department for the excellent working conditions he enjoyed during this project.

\bibliographystyle{plain}

\end{document}